\documentclass[11pt,reqno,tbtags]{amsart}

\usepackage{amsthm,amssymb,amsfonts,amsmath}
\usepackage{dsfont}
\usepackage{amscd} 
\usepackage{mathrsfs}
\usepackage[numbers, sort&compress]{natbib} 
\usepackage{subfigure, graphicx}
\usepackage{vmargin}
\usepackage{setspace}
\usepackage{paralist}
\usepackage{hyperref}
\hypersetup{
    colorlinks=true, 
    linkcolor=blue,  
    filecolor=magenta,      
    urlcolor=cyan,           
}
\usepackage{color}
\usepackage{xfrac}
\usepackage[normalem]{ulem}
\usepackage{stmaryrd} 
\usepackage[refpage,noprefix,intoc]{nomencl} 
\usepackage{comment}

\providecommand{\R}{}
\providecommand{\Z}{}
\providecommand{\N}{}
\providecommand{\NZ}{}

\renewcommand{\R}{\mathbb{R}}
\renewcommand{\Z}{\mathbb{Z}}
\renewcommand{\N}{{\mathbb N}}
\renewcommand{\NZ}{{\mathbb N_{0}}}

\newcommand{\1}[1]{\mathds{1}}
\newcommand{\indc}{\textbf{1}}
\newcommand{\E}[1]{{\mathbf E}\left[#1\right]}

\newcommand{\p}[1]{{\mathbf P}\left\{#1\right\}}

\newcommand\cB{\mathcal B}
\newcommand\cC{\mathcal C}

\newcommand\cS{{\mathcal S}}

\newcommand{\tildep}{\tilde{p}}

\newcommand{\smallest}{\mathop{\rm smallest}}

\newcommand{\dist}{\mathrm{dist}}
\newcommand{\bdd}{\mathcal{B}^{+}_{\sfrac{1}{2}}(\Z^{d})\cap L^{1}(\Z^{d})}

\providecommand{\ora}[1]{}
\renewcommand{\ora}[1]{\overrightarrow{#1}}
\DeclareRobustCommand{\SkipTocEntry}[5]{} 
\expandafter\def\csname SCALED:x\endcsname{\xi}
\expandafter\def\csname SCALED:t\endcsname{\tau}
\newcommand{\scaled}[1]{
    \ifcsname SCALED:#1\endcsname
        \csname SCALED:#1\endcsname
    \else
        \textbf{[ERROR IN \textbackslash scaled MACRO]}
    \fi
}

\newtheorem{thm}{Theorem}
\newtheorem{lem}[thm]{Lemma}
\newtheorem{prop}[thm]{Proposition}
\newtheorem{cor}[thm]{Corollary}
\newtheorem{dfn}[thm]{Definition}
\newtheorem{definition}[thm]{Definition}

\theoremstyle{remark}
\newtheorem{remark}[thm]{Remark}
\numberwithin{equation}{section}
\numberwithin{thm}{section}

\makenomenclature										
\newcommand{\tmesh}{\Delta_t}
\newcommand{\xmesh}{\Delta_x}
\newcommand{\tgrid}{\NZ\tmesh}
\newcommand{\xgrid}{\Z^{d}\xmesh}
\newcommand{\norm}[1]{\|#1\|}
\newcommand{\BD}[1]{\mathcal{B}_{#1}^{+}}

\newcommand{\floor}[1]{\left\lfloor#1\right\rfloor}

\DeclareMathOperator{\BV}{BV}
\DeclareMathOperator{\TV}{TV}
\DeclareMathOperator{\loc}{loc}

\DeclareMathOperator{\CM}{CM}
\DeclareMathOperator{\divergence}{div}

\newcommand{\ve}{\varepsilon}
\newcommand{\om}{\omega}
\newcommand{\vp}{\varphi}

\newcommand{\La}{\Lambda}

\newcommand{\Ga}{\Gamma}
\newcommand{\al}{\alpha}
\newcommand{\ga}{\gamma}

\gdef\x{k}
\gdef\y{\ell}
\gdef\time{{T_0}}

\newcommand{\scheme}{\mathcal{S}}

\begin{document}
\title{Symmetric Cooperative Motion in Higher Dimensions} 
\author{Louigi Addario-Berry}
\author{Gavin Barill}
\author{Hannah Cairns}
\author{Jessica Lin}
\email{louigi.addario@mcgill.ca}
\email{gavinpcb@gmail.com}
\email{hannah.abigail.cairns@gmail.com}
\email{jessica.lin@mcgill.ca}

\keywords{recursive distributional equations, convergence of finite difference schemes}
\subjclass[2010]{Primary: 60F05, 60K35, Secondary: 65M12} 

\begin{abstract} 
We prove a distributional convergence result for a multidimensional version of symmetric cooperative motion which was introduced and studied in one dimension in \cite{HRW, SCM1}. Our approach relies on framing the associated recursive distributional equation as a discretization of the porous medium equation. A major challenge is to analyze the behaviour of finite difference schemes which approximate weak solutions of the porous medium equation with unbounded initial data. In overcoming this difficulty, we perform a detailed analysis of the probability mass function of symmetric cooperative motion, in which we introduce several new comparison arguments for the discrete process. Consequently, along the way, we establish a novel multidimensional convergence result for a finite difference scheme approximating the ZKB/Barenblatt solution of the porous medium equation, which is of independent interest. 
\end{abstract}

\maketitle

\section{Introduction} \label{sec:intro}

We consider a multidimensional generalization of the process known as symmetric cooperative motion, which was considered in \cite{HRW, SCM1}. Our process takes place in the lattice $\mathbb Z^d$, where the neighbours of a point $\x \in \mathbb Z^d$ are the points $\x \pm e_i$ for $i = 1, \ldots, d$, where $\{e_{1},\dots,e_{d}\}$ is the standard basis of $\R^{d}$.

For an integer $m \geq 1$,  a {\em $\CM(m,d)$-cooperative motion started from $0\in \Z^{d}$} is a $\Z^d$-valued stochastic process $(X^{n}, n \geq 0)$ defined as follows. Let $(E^{n}, n\geq 0)$ be IID with $\p{E^1=\pm e_{i}}=\frac{1}{2d}$ for all $i\in \left\{1, \ldots, d\right\}=:[d]$. Starting from $X^{0}=0$, we let
\begin{equation}\label{eq:defining_equation}
	X^{n+1}=\begin{cases}
		X^{n}+E^{n}&\text{if $X^{n}=X^{n,1}=X^{n,2}=\ldots=X^{n,m}$}\\
		X^{n}&\text{otherwise},
	\end{cases}
\end{equation}
where $X^{n,1},\ldots X^{n,m}$ are $m$ independent copies of $X^n$. 

As in \cite{HRW, SCM1}, one way to realize $\CM(m,d)$-cooperative motion is using a tree-indexed random process. Let $\mathcal{T}$ be the complete rooted $(m+1)$-ary tree, with root labeled $\emptyset$ and node $v$ having children $(vi: 1\leq i\leq m+1)$. Let $\mathcal{T}_{n}$ denote the subtree of $\mathcal{T}$ containing nodes at distance at most $n$ from the root, and let $\mathcal{L}_{n}$ denote the leaves of $\mathcal{T}_{n}$.

For each $n\in \mathbb{N}$, we let $E^{n}=(E_v: v\in \mathcal{T}_{n}\setminus \mathcal{L}_{n})$ be as above. For $v\in \mathcal{L}_{n}$, let $\Sigma^{n}_{v}=0\in \Z^{d}$. For $v\in \mathcal{T}_{n}\setminus \mathcal{L}_{n}$, we recursively define 
\begin{equation*}
\Sigma^{n}_{v}:=\begin{cases}\Sigma^{n}_{v1}+E_{v}&\text{if $\Sigma^{n}_{v1}=\Sigma^{n}_{v2}=\ldots=\Sigma^{n}_{v(m+1)}$}\\
\Sigma^{n}_{v1}&\text{otherwise.}
\end{cases}
\end{equation*}
This recursion leads to an ``output value'' $\Sigma^{n}_{\emptyset}$ at the root, which has the same distribution as $X^{n}$. As described in \cite{HRW}, this model has natural connections to the random hierarchical lattice introduced by Hambly and Jordan \cite{MR2079916}. In the case when $m$ is positive and non-integer, one can still define a $\CM(m,d)$-distributed process, using a \emph{recursive distributional equation} (RDE) (i.e. a process whose probability mass function defined by \eqref{e.pdef} satisfies \eqref{e.pscheme}). Moreover, by translating the process, one can also consider a $\CM(m,d)$-cooperative motion started at $x\in \Z^{d}$.

The main result (and primary motivation) of this paper is the following distributional convergence result for a $\CM(m,d)$-distributed process in dimension $d>1$. 
\begin{thm} \label{t.main}
	Let $d>1$, $m\geq 1$, and $(X^{n}, n \geq 0)$ be $\CM(m,d)$-distributed started from $x\in \Z^{d}$. There exists an $\R^{d}$-valued random variable $B$, whose density with respect to Lebesgue measure is given by $\bar{u}(\cdot,1)$ as defined in \eqref{e.Bdef}, such that 
	\begin{equation*}
		\frac{X^{n}}{n^{1/(dm+2)}} \xrightarrow[n\to\infty]{d} B.
	\end{equation*}
\end{thm}
Observe that it is enough for us to prove Theorem \ref{t.main} in the case when $x=0$, since starting at an arbitrary $x\in \Z^{d}$ does not affect the limiting distribution. 

The function $\bar{u}$ is defined by 
\begin{equation}\label{e.Bdef}
	\bar{u}(x,t):=t^{-d\beta}(C-\ga|x|^{2}t^{-2\beta})_{+}^{\frac{1}{m}},
\end{equation}
where 
\begin{equation}\label{d.constants}
	\beta:=\frac{1}{dm+2}, \qquad \ga:=\frac{d m \beta}{m+1}, 
\end{equation}
and $C$ is a constant such that $\int \bar{u}(x, t)\, dx=1$ for all $t>0$. This function is known as the ZKB/Barenblatt solution \cite{barenblatt, zel1950collection} (with mass $1$) of the porous medium equation (PME), which we will soon introduce. In the case where $d=1$ and $m=1$, the first and third author (with collaborators) proved this result in \cite{HRW} for a generic initial distribution $\mu$ (i.e. not necessarily started at a single point). In a subsequent work \cite{SCM1}, in the case where $d=1$ and $m>0$ is arbitrary, the first and fourth author proved this result (with a collaborator) also for a generic initial distribution $\mu$. In all of the prior cases, the limiting random variable in those results agrees with $B$ as in Theorem \ref{t.main}. It is for this reason that Theorem \ref{t.main} and all subsequent results of this paper are stated for $d>1$. We expect that our proof technique could be adapted to the case $d=1$, with suitable modifications, but we did not pursue this since the case $d=1$ was already covered by the existing literature.

Our approach is to study the probability mass function of $X^{n}$. For $k=(k_{1}, \ldots, k_{d})\in \mathbb{Z}^{d}$ and $n \in \N_{0} := \N \cup \{0\}$, let
\begin{equation} \label{e.pdef}
	p(k,n) := \p{X^{n}=k}.
\end{equation}

As we will see in the sequel, $p(k,n)$ satisfies an RDE which will tie the process to solutions of the PME. Our approach crucially relies on various properties of $p(k,n)$, which require substantial new ideas to establish. This denotes the second major contribution of the current work. {To state this contribution, we require one additional definition: for $q: \mathbb Z^d \to \mathbb R$, we define the total variation of $q$ on a set $D \subset \Z^{d}$ as 
\begin{equation}\label{e.dtv}
[q]_{\TV(D)} := 
\frac{1}{2}\sum_{\{u,v\in D:\, u\sim v\}} |q(v)-q(u)|.
\end{equation}}
\begin{thm} \label{prop.pfacts}
	Let $d>1$, $m\geq 1$ and $(X^{n}, n \geq 0)$ be $\CM(m,d)$-distributed started from 0. Let $p(\cdot, n)$ be the PMF of $X^{n}$ (as in \eqref{e.pdef}). There exists $C = C(d,m) > 0$ such that,
	\begin{enumerate}[(i)]
		\item \label{p.Linfty} For all $k \in \Z^{d}$ and $n \in \N$, $p(k,n) \leq Cn^{-d\beta},$
		\item \label{p.conc} For all $r > 0$ and $n \in \N$, $\sum_{k\in \Z^{d}\cap B_{rn^{\beta}}^{c}}p(k,n) \leq C \exp(-r/C)$,
		\item \label{p.BV} For all $r>0$ and $n \in \N$, $[p(\cdot,n)]_{\TV(B_{rn^{\beta}}\cap\Z^{d})} \leq C r^{d-1}n^{-\beta},$
	\end{enumerate}
	where $B_{rn^{\beta}}$ denotes the ball of radius $rn^{\beta}$ in $\R^{d}$, $B_{rn^{\beta}}^{c}$ denotes its complement, and the discrete total variation $[\cdot]_{\TV(B_{rn^{\beta}}\cap\Z^{d})}$ is defined in \eqref{e.dtv}.
\end{thm}

Properties (i)-(iii) are discrete versions of properties satisfied by the ZKB/Barenblatt solution $\overline{u}$, which can be observed by direct inspection. Indeed, it is easy to see that $\bar{u}(\cdot,t) \le ct^{-d\beta}$, $\overline{u}(\cdot, t)$ is supported in a ball of radius $ct^{\beta}$, and by a short calculation, for all $r\leq C$ it holds that $[\overline{u}]_{\TV(B_{rt^{\beta}})}\sim \int_{B_{rt^{\beta}}}|\nabla_{x}\overline{u}(x,t)|\, dx\leq Cr^{d-1}t^{-\beta}$.  Respectively, property (i) says that $p$ shares the same decay in time, property (ii) conveys that most of the mass of $p(\cdot, n)$ is concentrated in a ball of radius $n^{\beta}$, and property (iii) obtains a similar bound on the discrete total variation of $p(\cdot, n)$. 

In the next subsection, we describe the RDE satisfied by $p(k,n)$, and then frame the analysis of that RDE in the context of convergence of a finite difference scheme. 

\subsection{Recursive Distributional Equations and Finite Difference Schemes: Ideas and Challenges of the Problem at Hand}

As a consequence of the definition of the process in \eqref{eq:defining_equation}, it is easy to verify that the following relation holds:
\begin{equation} \label{e.pscheme}
\begin{aligned}
	p(k,n+1) &=p(k,n)(1-p(k,n)^{m})+\frac{1}{2d}\sum_{\ell\sim k} p(\ell,n)^{m+1}\\
	&=p(k,n)+\frac{1}{2d}\sum_{i=1}^{d} \left[p(k+e_{i},n)^{m+1}-2p(k,n)^{m+1}+p(k-e_{i},n)^{m+1}\right]\\
	&=p(k,n)+\frac{1}{2d}\sum_{i=1}^d \sum_{\zeta \in \{-1, 1\}} \left[p(k+\zeta e_{i},n)^{m+1}-p(k,n)^{m+1}\right].
\end{aligned}
\end{equation}

The above recurrence looks like the discretization of a PDE ``at scale 1''. Throughout the paper, we will refer to recursive relations as \emph{schemes} (short for finite difference schemes which are considered in numerical analysis). In order to connect \eqref{e.pscheme} with a PDE, we introduce discrete meshes in time and space. For $N\in \N$, we define
\begin{align*}
	\tmesh = \tmesh(N) :=  N^{-1} = \text{time mesh} \quad\text{and}\quad \xmesh = \xmesh(N) := N^{-1/(dm+2)} = \text{space mesh}
\end{align*}
These parameters satisfy the relation $\tmesh = (\xmesh)^{dm+2}$, which, as we will soon argue, is a natural scaling relation for the problem which keeps the associated PDE scale-invariant. We now define $u_{N}: \xgrid\times\tgrid \to \R$ by
\begin{equation} \label{e.UNdef}
	  u_{N}(k\xmesh, n\tmesh):= \frac{1}{(\xmesh)^{d}}\p{X^{n}=k} = \frac{1}{(\xmesh)^{d}}p(k,n). 
\end{equation}
Using the relation $\tmesh = (\xmesh)^{dm+2}$, we may rewrite \eqref{e.pscheme} as 
\begin{align} \label{e.longscheme}
	& \frac{u_{N}(k\xmesh, (n+1)\tmesh)-u_{N}(k\xmesh, n\tmesh)}{\tmesh}\notag\\
		& = \frac{1}{2d (\xmesh)^2 } \sum_{\ell \sim k} (u_N(\ell \xmesh,n\tmesh)^{m+1}-u_N(k \xmesh,n\tmesh)^{m+1})\notag\\
	& =\frac{1}{2d}\sum_{i=1}^{d}\frac{u_{N}(k\xmesh+e_{i}\xmesh,n\tmesh)^{m+1}-2u_{N}(k\xmesh,n\tmesh)^{m+1}+u_{N}(k\xmesh-e_{i}\xmesh,n\tmesh)^{m+1}}{(\xmesh)^{2}}.
\end{align}
This is a finite difference scheme which, formally, approximates the partial differential equation 
\begin{equation}\label{e.genpm}
		u_{t}=\frac{1}{2d}\Delta(u^{m+1}), 
\end{equation} 
which is known as the porous medium equation (PME). The PME is a well-studied nonlinear PDE which has a rich theory of analysis (see for example \cite{vazquez2007porous} for a broad overview of the PME). The ZKB/Barenblatt solution $\bar{u}$ introduced in \eqref{e.Bdef} is a type of ``weak solution'' of the PME, with initial condition the Dirac delta measure $\delta$, where the notion of solution and initial condition will be made precise in Section \ref{ss.sol}. 

In light of the relation \eqref{e.UNdef}, Theorem \ref{t.main} follows easily if we can prove that
\begin{equation}\label{e.punchline}
\lim_{N\to\infty} u_{N}(\cdot ,1)=\bar{u}(\cdot ,1)\quad\text{in $L^{1}_{\loc}(\R^{d})$.}
\end{equation}
This statement is a convergence result for a scheme \eqref{e.longscheme} of the porous medium equation \eqref{e.genpm}. The typical perspective is to start from an initial condition $u_{0}:\R^{d}\rightarrow \R$, and use this to define $u_{N}(\cdot, 0)$ by a suitable discretization (see, for example, \eqref{e.initL1}, below). By generating $u_{N}$ at later times using the scheme, the goal is to show that $u_{N}$ converges to the solution of the PME with initial condition $u_{0}$. Thus, if we had at our disposal a convergence result for $(u_{N})_{N\geq 1}$, for which the discretization of the initial condition was given exactly by 
\begin{equation*}
u_{N}(k\xmesh, 0)=\frac{1}{(\xmesh)^{d}}p(k,0)=\begin{cases}\frac{1}{(\xmesh)^{d}}&\text{if $k=0$},\\
0&\text{otherwise,}
\end{cases}
\end{equation*}
then \eqref{e.punchline} would be automatic. 

As in \cite{HRW, ACM, SCM1}, the introduction of the finite difference schemes perspective has been fruitful for proving distributional convergence results for cooperative motions. The approach has further been used by Morfe \cite{morfe} and Chen, Duquesne, and Shi \cite{cds} in analyzing a variety of discrete random models, including the resistance of the series-parallel graph. Similar to the challenges in \cite{HRW, ACM, SCM1}, the key issue is that the only initial condition $u_{0}$ for which a suitable discretization yields the above initial condition for $u_{N}$ is for $u_{0}=\delta$, the Dirac delta, the probability distribution which assigns mass 1 to the point 0 in $\R^{d}$. Convergence results for finite difference schemes approximating PME with initial conditions $u_{0}\in L^{\infty}(\R^{d})$ have been well-studied in the literature (see for example \cite{karlsenrisebro01} which holds for inviscid and viscous conservation laws, for which the PME is a special case). However, to our knowledge, there are limited convergence results for finite difference schemes approximating weak solutions to the PME with unbounded, measure-valued initial conditions in dimensions greater than 1. 

In the case when $d=1$, it was possible to circumvent this lack of convergence results in two different ways. In \cite{HRW}, the authors used a convergence result of Evje and Karlsen \cite{EK} for finite difference schemes approximating entropy solutions of the PME with carefully chosen, bounded, initial conditions. By a coupling argument, the authors were then able to establish the statement of Theorem \ref{t.main} in the case when $d=1$ and $m=1$. In \cite{SCM1}, the authors proved distributional convergence by considering the cumulative distribution function (CDF) of the process, which, in one dimension, satisfies an RDE which approximates another nonlinear PDE known as the the parabolic $p$-Laplace equation. By considering the CDF, one essentially gains an additional level of regularity (and access to a more robust theory of convergence results for finite difference schemes, such as the one established in the work of Barles and Souganidis \cite{BS}). Using these convergence results for finite difference schemes with suitable initial conditions, combined with monotone coupling arguments, the authors of \cite{SCM1} established the distributional convergence in Theorem \ref{t.main} for all $m>0$ and $d=1$. 

In dimensions larger than 1, neither of these approaches seemed feasible, because the coupling arguments in both of the prior works relied heavily on the one-dimensional nature of the process. Consequently, due to the lack of convergence results for finite difference schemes approximating the PME with unbounded, measure-valued initial data, we overcame this by proving our own. We prove that \eqref{e.punchline} holds in Theorem \ref{t.mainpde}, which is the third major contribution of our work. Our proof relies heavily on the properties of the probability mass function $p$ established in Theorem \ref{prop.pfacts}; this allows us to perform a compactness argument for a suitable family of approximations to extract a convergent limit with the correct initial condition. For the approximations, we rely on the priorly mentioned convergence results of Karlsen and Risebro \cite{karlsenrisebro01} for finite difference schemes approximating entropy solutions of the PME with $L^{\infty}$ initial data. 

While writing this paper, we became aware of a convergence result recently proven by Di Francesco and Matthes \cite[Theorem 4]{dpme} for a discrete scheme (which is not a finite difference scheme) approximating distributional solutions to the PME in one dimension with initial condition a probability measure with density $u_{0}\in L^{1}(\R)$. More generally, the authors of \cite{dpme} were interested in analyzing the evolution of a discrete interacting particle system, which is another very interesting way in which the PME has arisen in probabilistic settings. Nevertheless, their result and their methods (which are specific to $d=1$ and the particular discrete scheme they work with) do not apply to our setting. In general, it seems that there is a fundamental challenge to proving convergence results for schemes approximating the PME with unbounded initial data, which requires new ideas to overcome.

The prior discussion highlights another challenge in this setting, which is that the Barenblatt solution can be interpreted as many different types of ``weak solutions'' of the PME. It is both a nonnegative entropy solution and a distributional solution (see Section \ref{ss.sol} for a detailed discussion). In this paper, we navigate between these two notions, which adds an additional layer of technical care which must be taken into account in our analysis. 

\subsection{Ideas and Challenges Pertaining to the proof of Theorem \ref{prop.pfacts}}
As priorly mentioned, the proofs of Theorem \ref{t.main} and of Theorem \ref{t.mainpde} rely crucially on the properties of the probability mass function $p=p(k,n)$ of the $\CM(m,d)$-distributed process $(X_{n}, n\geq 0)$. In this subsection, we describe some of the main ideas in the proof of Theorem \ref{prop.pfacts}. In light of the discussion following the statement of Theorem \ref{prop.pfacts}, the three properties of $p$ which are established are analogous to three properties observed by the Barenblatt solution $\bar{u}$. 

The first property, which is a type of $L^{\infty}$-estimate for solutions of the PME, is traditionally proven using the celebrated \emph{Aronson-Benilan estimate} on the pressure \cite{AB, vazquez2007porous}. While there have been some works examining a discrete-in-space, continuous-in-time version of the Aronson-Benilan estimate on locally finite graphs \cite{dab}, as well as the aforementioned work \cite{dpme} for an interacting particle system when $d=1$, we do not establish a discrete version of the Aronson-Benilan estimate in this work. However, our approach does bear some similarities to theirs. 

We begin by restricting our analysis to solutions of the finite difference scheme \eqref{e.pscheme} which are radially symmetric in space in an $\ell^{1}$ sense (we refer to these functions as volcanic functions in space). Under certain assumptions on the $\ell^{\infty}$-norm of the initial condition, the scheme \eqref{e.pscheme} is monotone, in the sense that solutions of the scheme with ordered initial data remain ordered at all future times. Generally speaking, the theory of convergence for monotone finite difference schemes is quite robust, because it allows one to access comparison methods. We establish several properties for solutions of finite difference schemes with volcanic initial data, which remain volcanic in space for all times (we hereby refer to these functions simply as volcanic). 

Upon restricting to volcanic functions, we give a new interpretation of how a solution of the finite difference scheme can be seen as a type of \emph{lazy random walk}.  This allows us to prove a general estimate, Theorem \ref{stepfive} below, which says that if a volcanic solution of the finite difference scheme is bounded from below in a certain scaled neighborhood of the origin, then it must be bounded from above by a comparable bound in the entire space. The rough idea is to use the local lower bound as a lower bound on the ``jump probability'' of associated lazy random walk. In the theory of PDEs, this would be like having a lower bound on the ellipticity of a differential operator. This (local) lower bound can then be upgraded to a uniform upper bound, which plays a similar role to that played by Aronson-Benilan estimates on the pressure. 

In order to establish the desired local lower bound on a volcanic solution, we develop a comparison theory between solutions of the scheme and so-called \emph{approximate solutions} of the scheme. In particular, we use a discretization of the Barenblatt solution as an approximate solution, and based on the comparison theory we establish, we prove the local lower bound on a volcanic solution. The final challenge is that the probability mass function $p$ is not actually volcanic for all times. Thus, we require one additional level of comparison to relate $p$ to a volcanic solution. This completes the proof sketch of Theorem \ref{prop.pfacts}(i). 

The proofs of Theorem \ref{prop.pfacts}(ii), (iii) follow relatively easily from Theorem \ref{prop.pfacts}(i). To prove Theorem \ref{prop.pfacts}(ii), we use a concentration inequality known as Freedman's inequality (Theorem \ref{t.freedman}). In applying Freedman's inequality, the bound in Theorem \ref{prop.pfacts}(i) is crucial as it allows us to control the conditional variance of projections of the original process. The proof of Theorem \ref{prop.pfacts}(iii) also follows relatively easily from the upper bound in Theorem \ref{prop.pfacts}(i); we just need to introduce a suitable approximation to compare the true solution $p$ of the finite difference scheme to a volcanic one.

\subsection{Outline of the Paper}
The structure of the remainder of the paper is as follows. In Section \ref{sec:notation}, we begin by introducing some notation, terminology, and basic properties used throughout the paper. In Section \ref{s.lrw}, we present the proof of Theorem \ref{stepfive}, which gives another interpretation of volcanic solutions of the finite difference scheme in terms of a lazy random walk. In Section \ref{s.p1}, we complete the proof of
Theorem \ref{prop.pfacts}(i). In Section \ref{s.p23}, we prove the remaining parts of Theorem \ref{prop.pfacts}. Finally, in Section \ref{s.scheme}, we prove both Theorem \ref{t.mainpde} and the remaining parts of Theorem \ref{t.main}; we use the properties of $p$ obtained from Theorem \ref{prop.pfacts} to perform the compactness argument for suitable approximations. In the Appendix, we provide the proofs of some technical estimates which are used in Section \ref{s.lrw}.

\section{Notation, Terminology, and some Basic Properties Used Throughout the Paper}\label{sec:notation}
\subsection{Notation}\label{s.note}
We use the notation $[d]:=\left\{1,2, \ldots, d\right\}$. We let $\N$ be the positive integers and define $\N_{0}:=\N \cup \{0\}$. For $T > 0$, we let $Q_{T} := \R^{d} \times (0,T)$. We use $C$ and $c$ to denote positive constants which depend only on $d$ and $m$, and which may change from line to line. We use $x, y, z$ for points in $\mathbb R^d$, and $\x, \y$ for points in $\Z^{d}$. If $k,\ell\in \Z^{d}$ are neighbors, i.e. $k-\ell=\pm e_{i}$ for some $i\in [d]$, then we write $k\sim \ell$. For $x\in\mathbb R^d$, we denote the Euclidean norm of the point using single vertical bars, $|x| = (x_1^2 + \cdots + x_d^2)^{1/2}$, while $|x|_1 = \sum_{i=1}^d |x_i|$ refers to the $1$-norm of the point.

For measurable functions $v: U \to \R$ where $U \subset \R^{d}$ or $U \subset\R^{d}\times[0,\infty)$, and $p\in (0, \infty)$, we denote the $L^p$-norm of $v$ by $\norm{v}_{L^{p}(U)}=(\int_U |v (x)|^p\, dx)^{1/p}$. Similarly, for $q: \mathbb{Z}^d \to \mathbb R$, we write $\|q\|_{L^p(\Z^{d})}=(\sum_{k \in \Z^d} |v(k)|^p)^{1/p}$. We write $L^{\infty}(\R^{d})$ for the space of essentially bounded functions on $\R^{d}$, and $L^{\infty}(\Z^{d})$ (respectively, $L^{\infty}(\xgrid)$) for the set of bounded functions on the discrete lattice $\Z^{d}$ (respectively, $\xgrid$). 

We next extend the definition of total variation from lattice functions to more general functions, and introduce the {\em bounded variation norm}. 
If $U \subset \R^d$ open and $v \in L^1(U)$, then the total variation of $v$ is given by
\begin{equation}\label{eq:tv_cont}
	[v]_{\TV(U)} := \sup\left\{\int_{U} v \divergence(w)\,dx : w \in C^{\infty}_{c}(U;\R^{d}), |w(\cdot)| \leq 1\right\},
\end{equation}
where $C^{\infty}_{c}(U;\R^{d})$ is the set of smooth functions with compact support from $U$ to $\R^{d}$. We define the \emph{bounded variation} norm of $v$ by $\norm{v}_{\BV(U)} = \norm{v}_{L^{1}(U)} + [v]_{\TV(U)}$, and we denote the Banach space $\BV(U) := \{v : U \to \R : \norm{v}_{\BV(U)} < \infty\}$.

For $\La>0$ and $U \subset \R^d$, we let 
\begin{equation*}
\cB_{\La}^+(U)= \{u:U \to [0,\La]\}. 
\end{equation*} 

Lastly, we introduce some additional notation regarding sets and extensions of functions from $\Z^{d}$ to all of $\R^{d}$. For $N\in \mathbb{N}$, we define
\begin{equation*}
	\square_{N} := \left[-\frac{\xmesh}{2},\frac{\xmesh}{2}\right)^{d}= \left[-\frac{\xmesh(N)}{2},\frac{\xmesh(N)}{2}\right)^{d}
	= \left[-\frac{N^{-1/(dm+2)}}{2},\frac{N^{-1/(dm+2)}}{2}\right)^{d},
\end{equation*} 
and note that $|\square_N|:= \int \indc_{\square_N}(x)dx = (\xmesh)^d$.
For every $k\in \Z^d$, we let 
\begin{equation*}
	\square_{N}(k\xmesh) = \square_{N} + k\xmesh.
\end{equation*}
We say that $w_{N}: \R^{d} \to \R$ is \emph{$(\xgrid)$-piecewise constant} if
\begin{equation*}
	w_{N}(x) = w_{N}(k\xmesh) \text{ for all } x \in \square_{N}(k\xmesh).
\end{equation*}
Likewise, we say that $w_{N}: \R^{d} \times [0,\infty) \to \R$ is \emph{$(\xgrid\times\tgrid)$-piecewise constant} if
\begin{equation*}
	v_{N}(x,t) = v_{N}(k\xmesh,n\tmesh) \text{ for all } (x,t) \in \square_{N}(k\xmesh)\times [n\tmesh,(n+1)\tmesh).
\end{equation*}

\subsection{Schemes and Monotonicity}
Given $A: \R\rightarrow \R$, define $\scheme=\scheme[A]: L^{\infty}(\Z^{d}) \to L^{\infty}(\Z^{d})$ by
\begin{equation}\label{d.Sdef}
\scheme q(\x) := q(\x) +\frac{1}{2d}\sum_{\y \sim \x} \left[A(q(\y)) - A(q(\x))\right],\quad\text{where $q\in L^{\infty}(\Z^{d})$.}
\end{equation}
We say a sequence of functions $q=(q_{n})_{n \ge n_{0}}$ is a \emph{solution} of scheme $\scheme$ on $[n_{0}, \infty)$ with initial condition $f: \Z^{d}\rightarrow \R$ if 
\begin{equation}\label{e.operator}
\begin{cases}
q_{n+1} = \scheme q_n=q_{n}(\x) +\frac{1}{2d}\sum_{\y \sim \x} \left[A(q_{n}(\y)) - A(q_{n}(\x))\right]&\text{for $n\geq n_{0}$,}\\
q_{n_{0}}=f.
\end{cases}
\end{equation}

Throughout the paper, we will frequently consider functions $\tilde{p}=(\tilde{p}_{n}(k))_{n\geq n_{0}}$ with initial condition $f$ solving 
\begin{equation}\label{timeevol}
\begin{cases}
\tilde{p}_{n+1}(k) = \tilde{p}_{n}(k) + \frac{1}{2d} \sum_{\ell \sim k} \left[A(\tilde{p}_{n}(\ell)) - A(\tilde{p}_{n}(k))\right]&\text{with $A(u) = u^{m+1}$, for $n\geq n_{0}$},\\
\tilde{p}_{n_{0}}=f. &{}
\end{cases}
\end{equation}
Instead of referring to $\tilde{p}$ as the solution of $\mathcal{S}[A]$ with $A(u)=u^{m+1}$, we simply refer to $\tilde{p}$ as the solution of the discrete PME on $[n_{0}, \infty)$ with initial condition $f$. 

For the function $p$ defined in \eqref{e.pdef} and appearing in \eqref{e.pscheme}, setting $p_n(\x) = p(\x, n)$, we see that $p$ is a solution of the discrete PME \eqref{timeevol} on $[0, \infty)$ with initial condition $\indc_{0}$.

We next define a notion of monotonicity of schemes. Fix a collection of functions $\cC\subset L^\infty(\Z^d)$ and a mapping $\scheme:\cC\to L^\infty(\Z^d)$. We say $\scheme$ is {\em monotone on $\cC$} if $\scheme(\cC)\subset \cC$, and $\scheme q \le \scheme \tilde{q}$ pointwise whenever $q,\tilde{q} \in \cC$ are functions such that $q \le \tilde{q}$ pointwise. We say $\scheme$ is {\em locally monotone on $\cC$} if $\cS(\cC)\subset \cC$, and for any $k \in \Z^d$, $(\scheme q)(k) \le (\scheme \tilde{q})(k)$ whenever $q,\tilde{q} \in \cC$ are such that $q(\ell)\le \tilde{q}(\ell)$ for all $\ell \in \Z^d$ with $|\ell-k|_{1}\leq 1$. Note that local monotonicity implies monotonicity, since if $q \le \tilde{q}$ pointwise then in particular $q(\ell)\le \tilde{q}(\ell)$ for all $\ell$ with $|\ell-k|_{1}\leq 1$, for any $k \in \Z^d$.
\begin{lem}\label{l.Smonotone}
Fix a differentiable function $A:\R \to\R$ with $A'(\cdot) \in [0,1]$ and define an operator $\scheme=\scheme[A]:L^\infty(\Z^d)\to L^\infty(\Z^d)$ by \eqref{d.Sdef}. Then $\scheme$ is locally monotone (and consequently monotone) on $L^\infty(\Z^d)$, and for any non-negative $q \in L^\infty(\Z^d)$, it holds that $0 \le \scheme q \le \|q\|_{L^\infty(\Z^{d})}$. 
\end{lem}
\begin{proof}
Observe that $A(u)$ and $u-A(u)$ are nondecreasing functions in $u$, since $A'(\cdot) \in [0,1]$. Fix functions $q, \tilde{q} \in L^{\infty}(\Z^d)$, and suppose $k\in \mathbb{Z}^{d}$ is such that $q(\ell)\leq q'(\ell)$ for all $|\ell-k|_{1}\leq 1$. Then by the previous observation,
\begin{align*}\mathcal{S}q(\x) &= q(\x) + {1 \over 2d}  \sum_{\y \sim \x} \left[A(q(\y)) - A(q(\x))\right]
\\&= q(\x)  - A(q(\x)) + {1 \over 2d} \sum_{\y \sim \x} A(q(\y))
\\&\le q'(\x)  - A(q'(\x))+ {1 \over 2d} \sum_{\y \sim \x} A(q'(\y)).
\\&=\mathcal{S}q'(\x).\end{align*}
Since constants are solutions of the scheme $\scheme$, it follows that $0 \le \scheme q \le  \|q\|_{L^\infty(\Z^{d})}$, which implies that $\scheme q\in L^{\infty}(\Z^{d})$. Thus, $\mathcal{S}$ is indeed locally monotone (and consequently monotone).
\end{proof}
For $\Lambda > 0$, the same argument applied to the function class $\cB_{\Lambda}^+(\Z^d)$ gives the following result, whose proof is omitted.
\begin{cor}\label{cor:newsmon}
Fix a constant $\Lambda>0$ and a function $A:\R \to\R$ such that $A'(u) \in [0,1]$ whenever $u \in [0,\Lambda]$. Define an operator $\scheme=\scheme[A]:L^\infty(\Z^d)\to L^\infty(\Z^d)$ by \eqref{d.Sdef}. Then $\scheme$ is locally monotone (and consequently monotone) on $\cB_{\Lambda}^+(\Z^{d})$. 
\end{cor}

The operator $\mathcal{S}$ also has some desirable properties with respect to $\norm{\cdot}_{L^1(\Z^{d})}$.

\begin{lem} \label{lem.S}
Assume that there exists $\La > 0$ such that $\scheme=\scheme[A]$ is monotone on $\mathcal{B}^{+}_{\La}(\Z^{d})$. Assume further that $A: \R\rightarrow \R$ is such that $A'(u) \in [0,1]$ for $u \in [0,\La]$. It follows that

\begin{enumerate}[(i)]
\item For all $q \in L^1(\Z^{d}) \cap \mathcal{B}^{+}_{\La}(\Z^{d})$, $\sum_{\x \in \Z^{d}} \scheme q(\x) = \sum_{\x \in \Z^{d}}  q(\x)$,
\item For all $q, q' \in L^1(\Z^{d}) \cap \mathcal{B}^{+}_{\La}(\Z^{d})$, $\Vert \scheme q - \scheme q' \Vert_{L^1(\Z^d)} \le \Vert q - q' \Vert_{L^1(\Z^d)}$. 
\end{enumerate}
\end{lem}

\begin{proof}
To prove (i), notice that since $A'(u) \in [0,1]$ for all $u \in [0,\La]$, it follows that $0 \leq A(q) \leq q$, where $q\in L^{1}(\Z^{d})$. Hence, we can rearrange the sum,
\begin{align*}
    \sum_{\x \in \Z^{d}}  (\scheme q(k) - q(k))
        &= {1 \over 2d} \sum_{i=1}^d \sum_{\zeta \in \{-1, 1\}} \left[\sum_{\x \in \Z^{d}}  A(q(\x + \zeta e_i)) - \sum_{\x \in \Z^{d}}  A(q(\x))\right] 
        \\&= {1 \over 2d} \sum_{i=1}^d \sum_{\zeta \in \{-1, 1\}} \left[\sum_{\y \in \Z^{d}}  A(q(\y)) - \sum_{\x \in \Z^{d}}  A(q(\x))\right]
        \\&= 0, 
\end{align*}
which is precisely the statement of (i). The inequality in (ii) is a consequence of \cite[Proposition 1]{crantartar}.

\end{proof}

In the case of the discrete PME, when $A(u) = u^{m+1}$, we obtain the following.
\begin{cor}\label{cor.Sell1}
	When $A(u) = u^{m+1}$, the associated operator $\scheme=\scheme[A]$ satisfies
	\begin{enumerate}[(i)]
	\item The scheme $\mathcal{S}[A]$ is monotone on $\mathcal{B}^{+}_{\sfrac{1}{2}}(\Z^{d})$.
		\item For all $q \in L^1(\Z^{d}) \cap \mathcal{B}^{+}_{\sfrac{1}{2}}(\Z^{d})$, $\sum_\x \scheme q(\x) = \sum_\x q(\x)$,
		\item For all $q, q' \in L^1(\Z^{d}) \cap \mathcal{B}^{+}_{\sfrac{1}{2}}(\Z^{d})$, $\Vert \scheme q - \scheme q' \Vert_{L^1(\Z^{d})} \le \Vert q - q' \Vert_{L^1(\Z^{d})}$. 
	\end{enumerate}
\end{cor}
\begin{proof}
By Corollary \ref{cor:newsmon}, we simply need to check that $A'(u)\in [0,1]$ whenever $u\in [0, \sfrac{1}{2}]$. By direct computation, $A'(u) = (m+1)u^m$, and since $m\geq 1$, it follows that $A'(0)\leq A'(u)\leq A'(\sfrac{1}{2})$. In the case when $m=1$, it is clear that $0\leq A'(u)\leq 1$, which implies the claim. Moreover, it can be seen that $A'(\sfrac{1}{2})$ is decreasing in $m\ge 1$, since
$${d \over dm} (m+1)\frac{1}{2^{m}} = (1 - (m+1)\log 2)\frac{1}{2^{m}},$$
which is strictly negative when $m>1$. The second two claims follow by Lemma \ref{lem.S}.  
\end{proof}

Akin to the above, we can also consider the operator perspective on the scaled lattice $\xgrid\times\tgrid$. Let $\mathcal{S}_{N}: L^{\infty}(\xgrid) \to L^{\infty}(\xgrid)$ be defined by
\begin{equation} \label{e.SN}
	\mathcal{S}_{N}v(k\xmesh)= v(k\xmesh)+\frac{\tmesh}{2d (\xmesh)^2}\sum_{\ell \sim k}\left[A(v(\ell \xmesh))-A(v(k\xmesh))\right], 
\end{equation}
where $k, \ell\in \mathbb{Z}^{d}$. We may also identify $\mathcal{S}_{N}$ with a scheme $\mathcal{S}^{\xmesh}_{N}: L^{\infty}(\Z^{d})\rightarrow L^{\infty}(\Z^{d})$, by using input functions of the form $q(k):=v(\xmesh k)$. By doing so, the notions of local monotonicity and monotonicity can naturally be extended to $\mathcal{S}_{N}$. The finite difference scheme for $u_{N}$, \eqref{e.longscheme}, can then be written as
\begin{equation} \label{e.scheme}
	u_{N}(\cdot,(n+1)\tmesh) = \mathcal{S}_{N}u_{N}(\cdot,n\tmesh).
\end{equation}

For initial conditions of $u_{N}(\cdot, 0)$, we introduce two notions. Given a Radon measure $\mu$ on $\R^{d}$, we say that \emph{$u_{N}$ is generated by $\mathcal{S}_{N}$ with initial measure $\mu$} if $u_{N}$ is the unique $(\xgrid \times \tgrid)$-piecewise constant function satisfying \eqref{e.scheme} and
\begin{equation} \label{e.initmeas}
	u_{N}(k\xmesh,0) = \frac{1}{|\square_{N}|}\mu(\square_{N}(k\xmesh)).
\end{equation}
If there exists $u_{0} \in L^{1}_{\loc}(\R^{d})$ such that $d\mu=u_{0}\, dx$, then we say that \emph{$u_{N}$ is generated by $\mathcal{S}_{N}$ with initial condition $u_{0}$}; in other words,
\begin{equation}  \label{e.initL1}
	u_{N}(k\xmesh,0) = \frac{1}{|\square_{N}|}\int_{\square_{N}(k\xmesh)}u_{0}(x)\,dx.
\end{equation}

\begin{remark} \label{rmk.uNtwodefs}
	Let $u_{N}: \R^{d} \times [0,\infty)\rightarrow \R$ be the $(\xgrid\times\tgrid)$-piecewise constant function defined by
		\begin{equation*}
		u_{N}(k\xmesh,n\tmesh) := \frac{1}{|\square_{N}|} \p{X^{n} = k} = \frac{1}{(\xmesh)^{d}}p(k,n).
	\end{equation*}
By rearranging the scheme \eqref{e.longscheme} satisfied by $u_{N}$, we see that $u_{N}$ is generated by $\mathcal{S}_{N}[A]$ with $A(u) = u^{m+1}$. As for initial conditions, let $\delta$ denote the probability measure on $\R^{d}$ with $\delta(\left\{0\right\})=1$ (we also refer to this as the Dirac $\delta$). Then
	\begin{align*}
		u_{N}(k\xmesh,0) = \frac{1}{|\square_{N}|} \p{X^{0} = k} = \frac{\delta(\square_{N}(k\xmesh))}{|\square_{N}|}.
	\end{align*}
	Hence, ``$u_{N}$ generated by $\mathcal{S}_{N}[A]$ with $A(u) = u^{m+1}$ and initial measure $\delta$'' and ``$u_{N}$, the unique $(\xgrid\times\tgrid)$-piecewise constant function satisfying \eqref{e.UNdef}'' are two equivalent ways of defining the same function. 	
\end{remark}

Analogous to Corollary~\ref{cor.Sell1}, we next show that $\mathcal{S}_{N}$ has desirable properties on a certain family. However, due to the additional parameter $N$, we are able to allow for a more general family.

\begin{lem} \label{lem.SN} Fix $\La > 0$. Suppose $N$ is sufficiently large such that
	\begin{equation} \label{e.cfl}
		N^{dm\beta} \geq 4(m+1)\La^{m}.
	\end{equation}
	Let $\mathcal{S}_{N}=\mathcal{S}_N[A]$ be as defined in \eqref{e.SN} with $A(u)=u^{m+1}$. For $n\in \N$, let $\mathcal{S}_{N}^{n}$ be the $n$-fold composition of $\mathcal{S}_{N}$. Then the following holds,
	\begin{enumerate}[(i)]
		\item The scheme $\mathcal{S}_{N}$ is monotone on $L^{\infty}(\xmesh \Z^{d})\cap \mathcal{B}^{+}_{\La}(\R^{d})$.
		\item For $w \in L^{\infty}(\xmesh \Z^{d})\cap\mathcal{B}^{+}_{\La}(\R^{d})$, for any $n\geq 1$, 
		\begin{equation*}
		\norm{\mathcal{S}_{N}^{n}w}_{L^{\infty}(\xmesh \Z^{d})}\leq \norm{w}_{L^{\infty}(\xmesh \Z^{d})}
		\end{equation*}
		and thus $\mathcal{S}_{N}^{n}\in L^{\infty}(\xmesh\Z^{d})\cap\mathcal{B}_{\La}^{+}(\R^{d})$.
		\item For $w, w' \in L^{1}(\R^{d}) \cap \mathcal{B}^{+}_{\La}(\R^{d})\cap L^{\infty}(\xmesh\Z^{d})$,
		$$\norm{\mathcal{S}_{N}^{n}w - \mathcal{S}_{N}^{n}w'}_{L^{1}(\R^{d})} \leq \norm{w - w'}_{L^{1}(\R^{d})}.$$
		
	\end{enumerate} 
\end{lem}
\begin{proof}
To show (i), we consider $\tilde{A}(u):=\frac{\tmesh}{(\xmesh)^{2}}A(u)$. Then \eqref{e.SN} can be rewritten as 
\begin{equation*}
\mathcal{S}_{N}v(k\xmesh)=v(k\xmesh)+\frac{1}{2d}\sum_{k\sim \ell}\left[\tilde{A}(v(\ell \xmesh))-\tilde{A}(v(k\xmesh))\right].
\end{equation*}
Using the identification with $\mathcal{S}_{N}^{\xmesh}: L^{\infty}(\Z^{d})\rightarrow L^{\infty}(\Z^{d})$, we now appeal to Lemma \ref{l.Smonotone}. For this, it is enough to check that $\tilde{A}'(u)\in [0,1]$ for all $u\in [0, \La]$.  
Noting that $\tilde{A}'(u)=\frac{\tmesh}{(\xmesh)^{2}}(m+1)u^{m}$, and recalling that $\frac{\tmesh}{(\xmesh)^2} = \frac{N^{-1}}{N^{-2/(dm+2)}} = N^{-dm\beta}$, this implies that 
\begin{equation*}
\tilde{A}'(u)=N^{-dm\beta}(m+1)u^{m}.
\end{equation*}
For $u\in [0, \La]$, it follows by hypothesis \eqref{e.cfl} that $\tilde{A}'(u)\in [0,1]$. Hence, an application of Lemma \ref{l.Smonotone} yields (i). 

To see (ii), note that the second conclusion of Lemma \ref{l.Smonotone} gives that $0 \leq \mathcal{S}_{N}w\leq \norm{w}_{L^{\infty}(\R^{d})} \leq \La$, meaning $\mathcal{S}_{N}w \in \mathcal{B}^{+}_{\La}(\R^{d})\cap L^{\infty}(\xmesh\Z^{d})$. By iteration, we have (ii).

To prove property (iii), we remark that for a $\xmesh \Z^{d}$-piecewise constant function $w$, $\int_{\R^{d}} |w(x)|\, dx=\sum_{k\in \Z^{d}} (\xmesh)^{d}|w(k\xmesh)|=(\xmesh)^{d}|w(\cdot \xmesh)|_{1}$. Therefore, the case $n=1$ follows directly from Lemma \ref{lem.S}. By setting $w'_{N} \equiv 0$, we have $\norm{\mathcal{S}_{N}w_{N}}_{L^{1}(\R^{d})} \leq \norm{w_{N}}_{L^{1}(\R^{d})}$. This fact, along with property (ii) says that $\mathcal{S}_{N}w_{N} \in L^{1}(\R^{d}) \cap \mathcal{B}^{+}_{\La}(\R^{d})\cap L^{\infty}(\xmesh\Z^{d})$ whenever $w_{N} \in L^{1}(\R^{d}) \cap \mathcal{B}^{+}_{\La}(\R^{d})\cap L^{\infty}(\xmesh\Z^{d})$. Hence, we can conclude property (iii) for all $n$ by iteratively applying Lemma 2.4.
\end{proof}

\subsection{Volcanic functions in space.} 

We say that a function $q: \mathbb Z^d \to \mathbb R$ is \emph{nondecreasing towards the origin} if, for every two neighbours $\x \sim \y$ in $\mathbb Z^d$ with $|\x|_{1} > |\y|_{1}$,
we always have $q(\x) \le q(\y)$. That is, if we take one step toward the origin on $\Z^d$, in the $\ell^{1}$-sense, the function does not decrease.

We note that this does not imply that $q(\x) \le q(\y)$ for any two points $\x, \y$ with $|\x|_{1} > |\y|_{1}$. For example, in two dimensions, we can define $$q(\x) = \begin{cases}1&\text{ if } \x_1 = 0 \text{ or } \x_2 = 0\\0&\text{ otherwise.}\end{cases}$$ This is nondecreasing toward the origin. However, $|(0,3)|_{1}=3> |(1,1)|_{1}=2$, while $q((0,3))=1> q((1,1))=0$.

We say that a function $q: \mathbb Z^d \to \mathbb R$ is \emph{symmetric} if the value of $q$ remains the same when the input arguments are reflected in any coordinate plane, and when any pair of coordinates is transposed. That is,  for any permutation $\pi \in S_d$ and $\varepsilon_1, \ldots, \varepsilon_d \in \{\pm1\}$,
\begin{equation}\label{e.symmdef}
q(\varepsilon_1 \x_{\pi(1)}, \ldots, \varepsilon_d \x_{\pi(d)}) = q(\x_1, \ldots, \x_d).
\end{equation}

If $q$ is symmetric and nondecreasing toward the origin, we call $q$ a \emph{volcanic function}. For $\La>0$, we say that the scheme (\ref{timeevol}) is \emph{volcano-preserving on $\mathcal{B}^{+}_{\La}(\Z^{d})$} if $\scheme q$ is volcanic whenever $q \in \mathcal{B}^{+}_{\La}(\Z^{d})$ is volcanic.

We now provide a sufficient condition on $A$ to guarantee that $\mathcal{S}$ is volcano-preserving.

\begin{lem}\label{monotonepreserved}
Fix $\Lambda>0$ and let $A: \R\rightarrow \R$ be such that $0\leq A'(\cdot)\leq \sfrac{2}{3}$ on $[0,\Lambda]$. Then the scheme
\[
\scheme q(\x) := q(\x) +\frac{1}{2d}\sum_{\y \sim \x} \left[A(q(\y)) - A(q(k))\right]
\]
defined by \eqref{d.Sdef} is volcano-preserving on $\mathcal{B}^{+}_{\Lambda}(\Z^{d})$. 
\end{lem}
\begin{proof}
Let $q$ be a volcanic function. We now check that $\mathcal{S}(q)$ is volcanic, under the hypotheses. 

By the definition of $\scheme$, if we compose $q$ with some reflection or coordinate transposition, then $\scheme q$ will undergo the same reflection or transposition, so $\scheme q$ is symmetric.

Let $\x, \y$ be neighbours with $|\x|_{1} > |\y|_{1}$. The proof is complete if we can show $\scheme q(\x) \le \scheme q(\y)$. By symmetry, we can assume without loss of generality that $k$ and $\ell$ differ in the first coordinate, and $\y_1 \ge 0$, so that the two points are $$\y = (\y_1, \y_2, \ldots, \y_d), \qquad \x = (\y_1 + 1, \y_2, \ldots, \y_d).$$
We consider two cases.

\medskip
\emph{Case 1. $\y_1 \ge 1$.} We use the local monotonicity property of $\mathcal{S}$. In this case, every coordinate $\y_i$ has the same sign as the corresponding coordinate $\x_i$. Let $v\sim 0$ be a neighbour of the origin, and we will compare $k+v$ and $\ell+v$. If $v = \pm e_i$ with $i\neq 1$, then $|\ell_{1}|<|k_{1}|$ and $|\ell_{i}\pm1|=|k_{i}\pm1|$. If $v=\pm e_{1}$, then since $\ell_{1}\geq 1$, we still have $|\ell_{1}\pm1|<|k_{1}\pm1|$, and $|\ell_{i}|=|k_{i}|$. 

Thus, $|\y+v|_{1}<|\x+v|_{1}$, and they are clearly neighbors, so since $q$ is volcanic, $q(\x+v)\leq q(\y+v)=q(k+e_{1}+v)$. By the local monotonicity of $\mathcal{S}$, this implies that $\mathcal{S}q(k)\leq \mathcal{S}q(k+e_{1})=\mathcal{S}q(\ell)$, as desired.

\medskip \emph{Case 2. $\y_1 = 0$.}
In this case, the analysis is more delicate. Using the symmetry of $q$, without loss of generality, we may write $\y = (0, \ldots, 0, \y_{m+1}, \ldots, \y_d)$, where $m\geq 1$ denotes the number of zero coordinates of $\y$. 

Since $q$ is volcanic and $\x$ and $\y$ are neighbours with $|\x|_{1}> |\y|_{1}$, $q(\x)\leq  q(\y)$. To prove the desired claim, we compute
\begin{align*}&\scheme q(\y) - \scheme q(\x)\\
&=q(\y)-q(\x)+{1\over2d}\sum_{i=1}^{m}[A(q(\y \pm e_{i}))-A(q(\y))]-{1\over2d}\sum_{i=1}^{m}[A(q(\x \pm e_{i}))-A(q(\x))]\\
&\quad +{1\over2d}\sum_{i=m+1}^{d}[A(q(\y \pm e_{i}))-A(q(\y))]-[A(q(\x \pm e_{i}))-A(q(\x))]
\end{align*}

We will analyze each of the above sums separately. By symmetry, we have that for all $1\leq i\leq m$, $q(\y\pm e_{i}) = q(\y+e_{1})=q(k)$, so
\begin{equation*}
    \label{monotonezeropiecea}
    R_1 := {1\over2d}\sum_{i=1}^{m}[A(q(\y \pm e_{i}))-A(q(\y))] = {2m \over 2d} [A(q(k)) - A(q(\ell))].
\end{equation*}

For the next term, we note that for $1< i\leq m$, $|k\pm e_{i}|_{1}>|k|_{1}$ and $|k+e_{1}|_{1}>|k|_{1}$. Since $q$ is volcanic and $A'\geq 0$, we have $A(q(k+v))-A(q(k))\leq 0$ for all $v\sim 0$ with $v\neq -e_{1}$. This implies
\begin{align*}
R_2 :=-{1\over2d}\sum_{i=1}^{m}[A(q(\x \pm e_{i}))-A(q(\x))]&\geq -\frac{1}{2d}[A(q(\x-e_{1})-A(q(\x))]\\
&=\frac{1}{2d}[A(q(\x))-A(q(\ell))].
\end{align*}

For the last term, we observe that for $m+1\leq i\leq d$, $|k\pm e_{i}|_{1}>|\ell\pm e_{i}|_{1}$, and thus by the same argument as in the last step, $A(q(\y \pm e_{i})) - A(q(\x \pm e_{i})) \ge 0.$ This implies 
\begin{align*}  \label{monotonezeropiecec}
R_3 &:={1\over2d}\sum_{i=m+1}^{d}[A(q(\y \pm e_{i}))-A(q(\y))]-[A(q(\x \pm e_{i}))-A(q(\x))]\notag\\
&\geq {1\over2d}\sum_{i=m+1}^{d}[A(q(\x))-A(q(\y))]={2d - 2m \over 2d} [A(q(\x))-A(q(\y))]
\end{align*}
Adding up the three prior inequalities, we get the bound
\begin{align*}
    \scheme q(\y)-\scheme q(\x) &= q(\y)-q(\x) + R_{1}+R_{2}+R_{3}
                    \\&\ge q(\y)-q(\x) + {2d+1 \over 2d} [A(q(\x))-A(q(\y))].
\end{align*}
Recall that $q(k) \le q(\ell)$, and since $0\leq A'(\cdot) \le 2/3$, we have 
\begin{equation*}
0\leq A(q(\ell)) - A(q(k)) \le \frac{2}{3} (q(\ell) - q(k)). 
\end{equation*}
Therefore, since $d>1$, we conclude
\begin{align*}
    \scheme q(\y) - \scheme q(k)
        &\ge q(\ell)- q(k) - {(2d + 1) \over 3d}[q(\ell)-q(k)]
        \\
        &\ge [q(\ell)- q(k)]\left[1 - {2d+1 \over 3d}\right]
        \\&\ge 0,
\end{align*}
and this completes the proof. 
\end{proof}

\begin{cor}\label{cflmonotonespace}If $A(u) = u^{m+1}$ with $m \ge 1$, then the scheme (\ref{timeevol}) is volcano-preserving on $\mathcal{B}^{+}_{\sfrac{1}{3}}(\Z^{d})$. 
\end{cor}
\emph{Proof.} By Lemma \ref{monotonepreserved}, we smply need to check that $0 \le A'(\cdot) \le 2/3$ on the interval $[0,1/3]$. 

Since $A'(u) = (m+1)u^m$, if $u \le 1/3$, then $A'(u)$ decreasing in $m$ for $m \ge 1$, because
\begin{align*}
    {\partial A'(u) \over \partial m}=(1 + (m+1) \ln u) u^m\le (1 + 2\ln u) u^m\le 0.
\end{align*}
Therefore $(m+1)u^m \le (1+1)u^1 \le 2/3$. So $A'(u) \le 2/3$ as long as $u \in [0, 1/3]$ and the result follows.\qed

\subsection{The Barenblatt Solution}\label{ss.bb}

\begin{equation*}
	\bar{u}(x,t):=t^{-d\beta}(C-\ga|x|^{2}t^{-2\beta})_{+}^{\frac{1}{m}},
\end{equation*}
where $\beta=\frac{1}{dm+2}$ and $\ga=\frac{m\beta}{2(m+1)}$, and $C$ is chosen so that $\int \bar{u}(x,t)\, dx=1$ for all $t$. 

This is the unique distributional solution (see Section \ref{ss.sol}) to
\begin{equation} \label{e.pmdelta}
	\begin{cases}
		u_{t}-\frac{1}{2d}\Delta(u^{m+1})=0&\text{in $\R^{d}\times (0, \infty)$},\\
		u(x,0)=\delta(x)&\text{in $\R^{d}$},
	\end{cases}
\end{equation}
where $\delta$ here is the Dirac delta. A reference for the construction of the Barenblatt solution is given in \cite{vazquez2007porous}. 

We notice that the parameter $C$ is used to tune the mass of the solution. Moreover, we also note that for any $t_{0}\geq 0$, $\bar{u}(\cdot,\cdot+t_{0})$ solves the PME with initial condition $\bar{u}(\cdot, t_{0})$. This implies that we can think of the Barenblatt solutions as a 2-parameter family of solutions to the PME. 

Instead of parametrizing by the mass and a time-shift, we now write this family in terms of different parametrizations which are better suited for our analysis. Let $R, \Gamma$ be positive real parameters. We now define
\begin{equation}\label{hformula0}\bar{u}^{(R,\Gamma)}(x, t) := {R^d \Gamma \over r(t)^d} \left(1 - \left({|x| \over r(t)}\right)^2\right)^{1/m}_+,~ r(t) := R \left(1 + {t \over \time}\right)^\beta,~ \time := {2d\gamma R^2 \over \Gamma^m},\end{equation}
where we recall $\beta=\frac{1}{dm+2}$ and $\ga=\frac{m\beta}{2(m+1)}$.

It is clear that the function $\bar{u}^{(R,\Gamma)}(\cdot, t)$ is positive in the open ball $|x| < r(t)$, and 0 outside it. As such, we call $B_{r(t)}$ the \emph{positive region}, and its complement the \emph{zero region}. The parameters have clear meanings: $R = r(0)$ is the radius of the positive region at time zero, and $\Gamma = \bar{u}^{(R,\Ga)}(0, 0)$ is the largest possible value of $\bar{u}$ (for all $x, t$).

\begin{lem}\label{l.baruprop}
For each $R, \Gamma>0$ the Barenblatt solution $\bar{u}=\bar{u}^{(R,\Gamma)}$ satisfies the following properties:
\begin{enumerate}[(i)]
\item For $m>0$, fixed $x$, $t\mapsto \bar{u}(x,t)$ is unimodal (i.e. $\max_{t} \bar{u}(x,t)$ is unique). 
\item For $m\geq 1$, for any $i\in [d]$, ${\partial^4 \over \partial x_i^4} \bar{u}(x,t)\geq 0$ for all $|x|<r(t)$. 
\end{enumerate}
\end{lem}
\begin{proof}[Proof of Lemma \ref{l.baruprop}]
To prove the unimodality claim, we recall that for fixed $x$, $\bar{u}(x,\cdot)$ could be initially 0, and once there exists $T$ such that $|x|<r(T)$, then $\bar{u}(x,t)$ is strictly positive for all times $t\geq T$ (so the time derivative must be positive at time $T$). We now compute $\tfrac{\partial\overline{u}}{\partial t}(x,t)$ assuming that $|x|\leq r(t)$ (since this is where the maximum of $t\mapsto u(x,t)$ will occur). We introduce the shorthand $\vartheta := \tfrac{|x|}{r(t)}$ and proceed by logarithmic differentiation. Since
$$\log \bar{u}= \log(R^d \Gamma) - d \log r(t) + {1 \over m} \log(1 - \vartheta^2),$$
using the fact that $\tfrac{\partial \vartheta}{\partial t} = -\tfrac{|x|r'(t)}{r(t)^2} = -\tfrac{\vartheta r'(t)}{r(t)}=- \tfrac{\vartheta\beta}{(t + \time)}$, and $md\beta+2\beta=1$, we obtain
\begin{align} 
    {1 \over \bar{u}} {\partial \bar{u} \over \partial t} &= {-d \beta \over t+T_0} - {1 \over m(1-\vartheta^2)} {\partial\vartheta^2 \over \partial t} 
    \notag
    \\&={1 \over m(t+T_0)} \left(-md\beta + 2\beta {\vartheta^2 \over 1 - \vartheta^2}\right)
    \notag
    \\
      &={1 \over m(t+T_0)} \left({1-md\beta \over 1 - \vartheta^2} - 1\right)\notag\\
  &={1 \over m(t+T_0)} \left({2\beta \over 1 - \vartheta^2} - 1\right).
    \label{first.time.derivative}
\end{align}
We conclude that $\tfrac{\partial\overline{u}}{\partial t}$ has the same sign as $2\beta(1 - \vartheta^2)^{-1} - 1$, and this is decreasing in time for a fixed $x$, because $\vartheta = |x|/r(t)$ is decreasing in time (and less than 1). Therefore, for fixed $x$, the time derivative of $\bar{u}(x,t)$ is initially positive and then becomes negative for higher values of $t$. This implies that $\bar{u}(x,t)$ is unimodal in $t$ for a fixed $x$.

For the second claim, fix $i\in [d]$, and we begin by rewriting
 \[
{\partial^4 \over \partial x_i^4} \bar{u}^{m+1} = \left({R^d \Gamma \over r(t)^d}\right)^{m+1} {\partial^4 \over \partial x_i^4} \left(s-\left({x_i \over r(t)}\right)^2\right)^{(m+1)/m}_+, 
\]
where $s := 1 - (\sum_{j \ne i} x_j^2) / r(t)^2$ does not depend on the $i$-th coordinate, and thus $s - x_i^2 / r(t)^2 = 1 - |x|^2/r(t)^2$ is greater than zero by the assumption that $|x| < r(t)$.

For $y\in \R$, we make the change of variables $x_i = s^{1/2} r(t) y$. Then $${\partial^4 \over \partial x_i^4} \bar{u}^{m+1} = C(R, \Gamma, t, x_1, \ldots, x_{i-1}, x_{i+1}, \ldots, x_d) {d^4 \over dy^4} (1-y^2)^{(m+1)/m}$$ where $C({\cdots})$ is a positive constant that depends on everything except $x_i$. In other words, the sign of $\partial_i^4 \bar{u}^{m+1}$ will be the same as the sign of $(d/dy)^4 (1-y^2)^{(m+1)/m}$. We also notice that $s > x_i^2/r(t)^2$, so $|y| < 1$.

We now calculate\footnote{This calculation can be verified directly by hand, or perhaps more preferably, with numerical/computational assistance.} the fourth derivative of $(1-y^2)^{(m+1)/m}$, given by $${d^4 \over dy^4} (1-y^2)^{(m+1)/m} = {4(1+m) (1-y^2)^{1/m-3} \over m^4} (3m^2+(6m^2-12m)y^2+(4-m^2)y^4).$$
Let $\psi(y) := 3m^2 + (6m^2-12m)y^{2} + (4-m^2) y^4$ be the factor on the right. Then since $m\geq 1$, it follows that $\psi(y)\geq 0$ for all $|y|\leq 1$.

The other factors are positive, so $(d/dy)^4 (1-y^2)^{(m+1)/m}$ is nonnegative, and $${\partial^4 \over \partial x_i^4} \bar{u}^{m+1} = C {d^4 \over d^4y} (1-y^2)^{(m+1)/m} \ge 0.$$ This inequality proves the claim about the fourth derivative.

\end{proof}

\section{The proof of Theorem \ref{stepfive}: Symmetric Cooperative Motion as a Lazy Random Walk}\label{s.lrw}

The main goal of this section is to prove the following result. Recall that $\beta=\tfrac{1}{dm+2}$. 
\begin{thm}\label{stepfive}
Let $\tildep=(\tilde{p}(\cdot, n))_{n\geq n_{0}}$ be a solution of \eqref{timeevol} on $[n_{0}, \infty)$ with initial condition $\tilde{p}_{n_{0}}$. Assume that $\tilde{p}(\cdot, n)\in \mathcal{B}^{+}_{\sfrac{1}{3}}(\Z^{d})\cap L^{1}(\Z^{d})$ and $\tilde{p}(\cdot, n)$ is volcanic for every $n\geq n_{0}$. If there are positive constants $C$ and $N$ so that for all $n \ge N$,
\begin{equation*}
\tildep(\x, n) \ge Cn^{-d\beta}\quad\text{for all $|\x| \le Cn^\beta$},
\end{equation*}
Then there is a constant $C'=C'(C, d,m, |\tilde{p}_{n_{0}}|_{1},N)>0$ such that for all $n\geq N$, 
\begin{equation*}
\sup_{\x\in\Z^d} \tildep(\x,n)\le C'n^{-d\beta}.
\end{equation*}
\end{thm}

Theorem \ref{stepfive} will be used to prove Theorem \ref{prop.pfacts}(i). 
The key to establishing Theorem \ref{stepfive} is to interpret solutions of \eqref{timeevol} as a type of lazy random walk. This interpretation will allows us to control the values of $\tilde{p}$ with an associated optimization problem (see Lemma \ref{l.pupperopt} for a precise statement.)

We begin by introducing the framework of lazy random walks. Let $n_0\in \mathbb{N}_{0}$. Fix functions $r: \mathbb Z^d \times \mathbb \{n_0, n_0+1,\ldots \} \to [0, 1]$ and a probability mass function $h: \mathbb Z^d \to [0,1]$. We now define the $r$-lazy random walk $(Z_n)_{n\geq n_{0}}$ on $\mathbb Z^d$, started from an initial distribution $h$, as follows. 

The starting location is a random point $Z_{n_0}$ in $\mathbb Z^d$ such that $\p{Z_{n_0}= \x} = h(\x)$ (we write this as $Z_{n_{0}}\sim h$). We define the random walk recursively. At every time step, we choose a neighbour of $Z_n$ uniformly at random, and call it $Y_n$. Then $$Z_{n+1} = \begin{cases}Y_n&\text{with probability }r(Z_n,n)\\Z_n&\text{ otherwise}.\end{cases}$$
We highlight that the jump probability of the lazy random walk depends on both its location and the time the random walker steps.

Writing 
\[
\p{Z_{n+1} = \x} = \sum_{\y \in \mathbb Z^d} \p{Z_n = \ell} \p{Z_{n+1} = k \mid Z_n = \ell}\, ,
\]
by the definition of $Z_{n+1}$, we see that
 \begin{equation*}
 \p{Z_{n+1} = k \mid Z_n = \ell}=\begin{cases} 1 - r(\x, n)&\text{if $\y = \x$,}\\
\frac1{2d}r(\ell, n)&\text{if $\y\sim \x$,}\\
0&\text{otherwise.}
\end{cases}
\end{equation*}
Letting $f(\x, n):= \p{Z_n = \x}$, we obtain 
\begin{align}
    \notag\label{lin}f(\x, n+1)
        &= (1 - r(\x, n))f(\x, n) + \frac1{2d}\sum_{\y \sim \x} r(\y, n) f(\y, n)
    \\&= f(\x,n) + {1 \over 2d} \sum_{\y \sim \x} \left[r(\y,n) f(\y,n) - r(\x,n) f(\x,n)\right].
\end{align}
This recurrence has a unique solution for a given initial condition $f(\x, n_0) = h(\x)$. Here, we see that functions $r(k,n)$, corresponding to the jump probabilities, act effectively as diffusivity parameters in this model.

Thus, we see that a $\CM(m,d)$-cooperative motion $X=(X_n)_{n\geq 0}$ is a lazy random walk started from $\indc_{0}$, with $r(\x, n) = p(\x, n)^{m}$. Moreover, if we have any solution $\tildep$ of the scheme \eqref{timeevol} on any domain, and we set $r(\x, n) = \tildep(\x, n)^m$, then $f(\x, n) = \tildep(\x, n)$ solves the above recurrence (\ref{lin}).

\subsection{Volcanic solutions of \eqref{timeevol} and an associated Optimization Problem}\label{optimizationproblem} 
Let $Z=(Z_n)_{n\geq n_{0}}$ be a lazy random walk as described above, with jump probability $r(\x, n)$ that varies with space and time. We will introduce an optimization problem over these walks and use it to get a bound on solutions of \eqref{timeevol}. 

We fix starting and ending times $n_0 < n_1$ and a fixed volcanic initial distribution $h(\x)$ with $Z_{n_0} \sim h$. Loosely speaking, our goal is to choose a jump probability function $r(\x, n)$ that maximizes the probability $\p{Z_{n_1} = 0}$ of being at the origin at time $n_1$.

However, for the problem to be useful in our setting, we add some more conditions as follows. We are given two more parameters: a radius $\rho > 0$ and a jump probability $b\in [0, 2/3]$. We will only allow jump probability functions $r=(r(\x, n))_{n}$ that satisfy the following three conditions for every time step $n \in \{n_0, \ldots, n_1 - 1\}$.
\begin{equation}
\label{e.rconds}
\begin{aligned}
&\text{(i) $r(\cdot, n)$ is volcanic; }
\\&\text{(ii) $\p{Z_n = \cdot}$ is volcanic for volcanic initial data; }
\\&\text{(iii) $r(\x, n) \ge b$ for any point $\x \in \mathbb Z^d$ with $|\x| \le \rho$.}
\end{aligned}
\end{equation}
The second condition is a restriction on $r(\x, n)$ because as can be seen by the recursion, $\p{Z_n = \x}$ depends continuously on the initial distribution and the jump probabilities.

We first show that it is possible to satisfy all these conditions by taking $r \equiv b$.

\begin{lem}\label{constantworks}The constant jump probability $r \equiv b$ satisfies all three conditions (\ref{e.rconds}).\end{lem}
\begin{proof}
The first and third conditions are obviously satisfied. For the second condition, we must check that the probabilities $\p{Z_n = \x}$ are volcanic when $r \equiv b$.

By assumption, the initial condition $h$ is volcanic. We observe that the sequence of functions $q=(q_n(\cdot))_{n_{0}\leq n\leq n_{1}}$ given by $q_{n}(\x) := \p{Z_n = \x}$ is a solution of \eqref{e.operator} on $[n_{0}, n_{1}]$, with $A(u) = bu$ and initial condition $h$. Then since $A'(\cdot) = b \in [0, 2/3]$, Lemma \ref{monotonepreserved} tells us that volcanicity is preserved. Hence, $q_n(\x) = \p{Z_n = \x}$ is volcanic for $n_0 \le n < n_1$, and the second condition is also satisfied. 
\end{proof}
We now try to solve the following optimization problem:
\begin{center}
\fbox{
\begin{minipage}{0.6\textwidth}
How can we choose $r(\x, n)$ to maximize the probability
that $Z_{n_1} = 0$, given the restrictions \eqref{e.rconds}?
\end{minipage}
}
\end{center}

Let $V^{(\rho, b)}_{n_0 \to n_1}(h)$ be the optimal value in this problem:
\[
V^{(\rho, b)}_{n_0 \to n_1}(h) := \sup\left\{\vphantom{\strut}\p{Z_{n_1} = 0 \mid Z_{n_0} \sim h}\;\middle\vert\;r\text{ satisfies \eqref{e.rconds}}\right\}.
\] 
This value depends on all the parameters $\rho, b, n_0, n_1, h$, but to keep the notation from becoming cluttered, we will drop the superscript $(\rho, b)$ and simply write $V_{n_0 \to n_1}(h)$.

We begin with a lemma that connects this optimal value to our main goal of the section, the proof of Theorem~\ref{stepfive}. It gives us an upper bound on $\tildep$ in terms of $V_{n_0 \to n_1}$.
\begin{lem}\label{l.pupperopt} Let $n_0 \le n_1$ be nonnegative integers. Let $\tildep: \mathbb Z^d \times \{n_0, \ldots, n_1\} \to [0, 1]$ be a function such that $\tilde{p}=(\tilde{p}(\cdot,n))_{n_{0}\leq n\leq n_{1}}$ is a solution of the scheme (\ref{timeevol}) on $[n_{0}, n_{1}]$ with initial condition $\tilde{p}(\cdot, n_{0})$, and for all time steps $n \in \{n_0, \ldots, n_1\}$,
\begin{itemize}
\item $\tildep(\cdot, n)$ is volcanic and
\item $\tildep(\x, n) \ge b^{1/m}$ whenever $|\x| \le \rho$.
\end{itemize}
Assume the $\sigma := \sum_k \tildep(\x, n_0)<\infty$. Let $h(\x) := \tildep(\x, n_0) / \sigma$. Then for all $k\in \mathbb{Z}^{d}$, $$\tilde{p}(k,n_{1})\leq \sigma V^{(\rho,b)}_{n_0 \to n_1}(h).$$ 
\end{lem}

\begin{proof} Let $(Z_n)_{n_{0}\leq n\leq n_{1}}$ be a lazy random walk with jump probability $r(\x, n) = \tildep(\x, n)^m$ and initial distribution $h$, as defined above. We first observe that $\tilde{p}(k,n) = \sigma \p{Z_n = \x}$, because the function $f(\x, n) := \sigma \p{Z_n = \x}$ satisfies the relation (\ref{lin}) with initial condition $f(\x, n_{0}) = \sigma h(\x) = \tildep(\x, n_{0})$. Since $\tilde{p}$ also satisfies \eqref{lin} with the same initial condition, we conclude that the two functions are equal: $\sigma \p{Z_n = \x} = \tildep(\x, n)$ for $\x \in \mathbb Z^d$ and $n \in \{n_0, \ldots, n_1\}$.

We now argue that $r(\x, n) = \tildep(\x, n)^m$ satisfies \eqref{e.rconds}. Since $\tildep$ is volcanic for each $n\in [n_{0}, n_{1}]$, the jump probability $r(\x, n) = \tildep(\x, n)^m$ and the distributions $\p{Z_n = \x}$ are also volcanic for each $n\in [n_{0}, n_{1}]$, so the first two conditions hold. We have assumed that $\tildep(\x, n) \ge b^{1/m}$ when $|\x| \le \rho$, which implies that $r(\x, n) = \tildep(\x, n)^m \ge b$, so the third condition holds and $r$ satisfies the three conditions \eqref{e.rconds}.

This tells us that $\p{Z_{n_1} = 0}$ is bounded above by $V_{n_0 \to n_1}(h)$ by definition. Therefore $$\widetilde p(0, n_1) = \sigma \p{Z_{n_1} = 0} \le \sigma V_{n_0 \to n_1}(h).$$
Our assumption that $\tilde{p}(\cdot, n_{1})$ is volcanic implies in particular that $\sup_\x \tildep(\x, n_1) = \tildep(0, n_1)$, so $\tilde{p}(k,n_1) \le \tilde{p}(0,n_1) \le \sigma V_{n_0 \to n_1}(h)$ for all points $\x \in \mathbb Z^d$.
\end{proof}

By Lemma \ref{l.pupperopt}, it is now clear that if we can obtain an upper bound on $V^{(\rho,b)}_{n_0 \to n_1}$, then this paves a path to prove Theorem \ref{stepfive}. This also allows us to give some intuition as to why Theorem \ref{stepfive} should hold. The hypothesis of the theorem states that on a ball centered at the origin, the jump probabilities $\tilde{p}(k,n)^{m}$ are sufficiently large. Heuristically, this suggests that the probability of being at the origin cannot be too large. We make these heuristics precise in the rest of this section, and this will give us an upper bound on precisely the quantity $V^{(\rho,b)}_{n_0 \to n_1}$. This is how we will establish Theorem \ref{stepfive}.

\subsection{A formula for $V_{n_0 \to n_1}(h)$}
\label{ss.dislap}
In order to make use of Lemma \ref{l.pupperopt}, we establish a formula for $V_{n_0 \to n_1}(h)$ which will hold as long as $n_{1}-n_{0}$ is less than a quantity $T(\rho, b)$ called the \emph{breakdown time}. We then show that the supremum in $V_{n_0 \to n_1}(h)$ is in fact attained by a lazy random walk with \emph{constant} jump probability $b$.  

For $f: \mathbb Z^d \to \mathbb R$, recall that the \emph{discrete Laplacian} $\Delta f: \mathbb Z^d \to \mathbb R$ is defined by $$(\Delta f)(\x) = \sum_{\y \sim \x} (f(\y) - f(\x)).$$ 

We first collect some identities for the discrete Laplacian, which we will use throughout this section. We think of the lattice $\mathbb Z^d$ as the vertex set of a directed graph with nearest-neighbour edges $(\x, \x \pm e_i)$ for $\x \in \mathbb Z^d$ and $i\in [d]$. Each vertex has $2d$ edges out and $2d$ edges in. Let $\sum_{e = (\x, \y)}$ denote the sum over all directed edges.

The \emph{discrete gradient} is the function $\nabla f(e) = f(\y) - f(\x)$ on directed edges $e = (\x, \y)$.
\begin{lem}
\label{laprearranger} Let $f,g:\mathbb{Z}^{d}\to\mathbb{R}$, and suppose at least one of $f$ and $g$ has bounded support. Then we have
\begin{enumerate}[(i)]
\item Green's identity:
\begin{equation*}
\sum_{k\in \Z^{d}} \Delta f(\x) g(\x) = \sum_{k\in \Z^{d}} f(\x) \Delta g(\x).
\end{equation*}
\item Discrete integration by parts: 
$$\sum_{e~\text{an edge of}~\Z^d} \nabla f(e) \nabla g(e) = -2\sum_{\x\in\mathbb{Z}^{d}} f(\x) \Delta g(\x).$$
\item Product rule: for $e=(k,\ell)$,
$$\nabla(fg)(e) = f(\y) \nabla g(e) + g(\x) \nabla f(e)$$\end{enumerate}
\end{lem}

\begin{proof}[Proof of (i)]
By symmetry and the definition of the discrete Laplacian, 
\begin{align*}\sum_{\x\in \Z^{d}} \Delta f(\x) g(\x) &=  \sum_{\x\in \Z^{d}} \sum_{\y \sim \x} f(\y) g(\x) - 2d\sum_{\x\in \Z^{d}}f(\x) g(\x) \\&=  \sum_{\x\in \Z^{d}} \sum_{\y \sim \x} f(\x) g(\y) - 2d\sum_{\x\in \Z^{d}} f(\x) g(\x) \\&=\sum_{\x\in \Z^{d}} f(\x) \Delta g(\x)\end{align*} The infinite sums are all well-defined by the assumption of bounded support. 
\end{proof}
\begin{proof}[Proof of (ii)]
We check:\begin{align*}
    \sum_e \nabla f(e) \nabla g(e)
        &=\sum_{\x\in \Z^{d}} \sum_{\y \sim \x} (f(\y) - f(\x)) (g(\y) - g(\x))
        \\&= 4d \sum_{\x\in \Z^{d}} f(\x) g(\x) - 2\sum_{\x\in \mathbb{Z}^{d}} \sum_{\y \sim \x} f(\x) g(\y)
        \\&= - 2\sum_{\x\in \Z^{d}} f(\x) \Delta g(\x),
\end{align*}
where in the second line we switch $\y$ and $\x$ in two of the sums.\end{proof}

\begin{proof}[Proof of (iii)]
We expand both sides:
\begin{align*}
\nabla(fg)(e) &= f(\y)g(\y) - f(\x) g(\x)\\
f(\y) \nabla g(e) + g(\x) \nabla f(e) &= f(\y) g(\y) - f(\y) g(\x) + g(\x) f(\y) - g(\x) f(\x)
\end{align*}
and $-f(\y) g(\x) + g(\x) f(\y) = 0$, so the two lines are the same.
\end{proof}

With the discrete Laplacian notation, the recurrence (\ref{lin}) can be rewritten as 
\begin{equation*}
f_{n+1} = f_n + \frac{1}{2d}\Delta(r_n f_n),
\end{equation*}
where $f_n(k)=f(k,n)$ and $r_n(k)=r(k,n)$. 

Let $(Y_{n})_{n\geq 0}$ be a lazy random walk with constant jump probability $b$ and initial distribution $\indc_{0}$. We call this a $b$-lazy random walk. Let $q^{(b)}_n(\x) := \mathbb P[Y_n = \x]$. Then $q_n^{(b)}$ satisfies the recurrence $q_{n+1}^{(b)} = q_n^{(b)} + \tfrac{b}{2d}\Delta q_n^{(b)}$, and for each $n$, $q_{n}^{(b)}$ has bounded support. We will see later that the discrete Laplacian $q_n^{(b)}(\x)$ is nonnegative in a certain region of space/time. In order to express this region precisely, let the \emph{breakdown time} $T(\rho, b)$ be given by 
\begin{equation}
    \label{defoftrb}
    T(\rho, b) := \min\left\{n \ge 0: \text{there exists } k \in \mathbb Z^d \text{ with }|\x| > \rho \text{ and }\Delta q^{(b)}_n(\x) < 0\right\}.
\end{equation}

We now calculate the exact value of $V_{n_0 \to n_1}$ when $n_1 - n_0 \le T(\rho, b)$. Our argument breaks down when the time gap is larger, hence the name ``breakdown time.''

\begin{thm}\label{optim}
Let $\rho \ge 0$ and $b \in [0, 2/3]$. Let $q^{(b)}=(q_{n}^{(b)})_{n\geq 0}$ be the distribution of a $b$-lazy random walk as above. For all nonnegative integers $n_0 \le n_1$ with $n_1 - n_0\leq T(\rho, b)$, for any volcanic probability mass function $h: \Z^d \rightarrow [0,1]$,
\[
V_{n_0 \to n_1}(h) = \sum_{\x\in \mathbb{Z}^{d}} h(\x) q^{(b)}_{n_1 - n_0}(\x).
\]
\end{thm}

\begin{proof}[Proof of Theorem \ref{optim}]
For ease of notation, we write $q_{n}(\x) = q^{(b)}_n(\x)$ and $r_n(\x) = r(\x, n)$.

We claim that for all $n\in \left\{n_{0}, n_{0}+1, \ldots, n_{1}\right\}$, 
\begin{equation}
\label{optiminduction}
V_{n \to n_1}(h) = \sum_{k\in \Z^{d}} h(\x) q_{n_1-n}(\x).
\end{equation} We will prove this by induction downward in $n$, starting from the base case $n = n_1$ where $V_{n_1 \to n_1}(h) = h(0)$ and the statement is obvious.

Suppose $n\in [n_{0}, n_1)$. We know by the induction assumption that $$V_{(n+1) \to n_1}(h) = \sum_{\x\in \Z^{d}} h(\x) q_{n_1 - (n+1)}(\x).$$

We will regard optimization of the jump probabilities $r$ as an optimal control problem: instead of choosing $r$ for all values of $k$ and $n$, we will choose $r(\x, n)$ at the current step $n$ and then we will choose $r(\x, n')$ for $n < n' \le n_1 - 1$. By the Bellman optimality principle, we can find $V_{n \to n_1}(h)$ by choosing the jump probability $r_n$ that gives us the best value at step $n+1$.

If the lazy random walk $Z_n$ has $\p{Z_n = \x} = h(\x)$, then the distribution of $Z_{n+1}$ is
\begin{align*}
    \p{Z_{n+1} = \x} = h(\x) + {1 \over 2d} \sum_{\y \sim \x} \left[r(\y, n) h(\y) - r(\x, n) h(\x)\right]= h(\x) + \frac{1}{2d} \Delta(r_nh)(\x).
\end{align*}
We can now ask about the value of the optimal control problems at time step $n + 1$, starting from these new distributions. We consider the set of jump probabilities $r_n: \mathbb Z^d \to [0, 1]$ that satisfy the following three conditions:
\begin{equation}
\label{e.rcondstwo}
\text{(i) $r_n$ is volcanic; } \quad
\text{(ii) $h +\tfrac{1}{2d} \Delta(r_n h)$ is volcanic; } \quad
\text{(iii) $r_n(\x) \ge b$ for $|\x| \le \rho$}.
\end{equation}
We highlight that \eqref{e.rcondstwo} is a condition for only $r_{n}$, whereas \eqref{e.rconds} is a condition for $r=(r_{n})_{n_{0}\leq n\leq n_{1}-1}$, as well as on $\p{Z_{n}=\cdot}$ for $n\in \left\{n_{0}, \ldots, n_{1}-1\right\}$. 

Let
\begin{equation}
    v := \sup_{r_n} V_{(n+1) \to n_1}\left(h + {1 \over 2d} \Delta(r_n h)\right) = \sup_{r_n} \sum_{\x\in \Z^{d}} \left(h(\x) + {1 \over 2d}\Delta(r_n h ) \right) q_{n_1 - (n+1)}(\x),
\end{equation}
where we take the supremum over jump probabilities $r_n$ that satisfy \eqref{e.rcondstwo}, and the second equality holds by the induction hypothesis.

We claim that $V_{n \to n_1}(h) = v$. To prove this, we show that $V_{n \to n_1}(h) \le v$ and that $v \le V_{n \to n_1}(h)$ separately. First, if $r$ is a jump probability that satisfies the conditions \eqref{e.rconds} with initial time $n$, then $r$ satisfies \eqref{e.rconds} with initial time $ n+1$ and starting distribution $h(\x) = \p{Z_{n+1} = \x} = h + \tfrac{1}{2d} \Delta(r_n h)$. Thus,$$\p{Z_{n_1} = 0} \le V_{(n+1) \to n_1}\left(h + {1 \over 2d} \Delta(r_n h)\right) \le v,$$because $r_n$ satisfies the conditions \eqref{e.rcondstwo}. Therefore $V_{n \to n_1}(h) \le v$.

Conversely, if $(r_{n'})_{n+1\leq n'\leq n_{1}}$ satisfies \eqref{e.rconds}, and $r_{n}$ satisfies \eqref{e.rcondstwo}, then we claim that $(r_{n'})_{n\leq n'\leq n_{1}}$ satisfies \eqref{e.rconds}. Indeed, $\p{Z_{n}=\cdot}=h$ is volcanic, and $\p{Z_{n+1}=\cdot}=h+\frac{1}{2d}\Delta h$ is volcanic by property (ii) of \eqref{e.rcondstwo}. Hence, we have that 
\begin{equation*}
V_{(n+1) \to n_1}\left(h + {1 \over 2d} \Delta(r_n h)\right)\leq V_{n \to n_1}\left(h\right),
\end{equation*}
and taking a supremum over $r_{n}$ satisfying \eqref{e.rcondstwo}, we obtain that $v\leq  V_{n \to n_1}\left(h\right)$.

We next prove that $v = \sum_{\x\in \Z^{d}} h(\x) q_{n_1 -n}(\x)$, which will give us the conclusion. By Lemma \ref{constantworks}, constant jump probability $r_n \equiv b$ satisfies all three conditions \eqref{e.rcondstwo}, so
\begin{align}
    \notag
    \sum_{\x \in \mathbb Z^d}
        h(\x) q_{n_1 - n}(\x)
    &=
    \sum_{\x\in \Z^{d}} h(\x) \left(
        q_{n_1 - (n+1)}(\x) + {b \over 2d} \Delta q_{n_1 - (n+1)}(\x)
    \right)
    \\\label{e.achieved}
    &=
    \sum_{\x\in \Z^{d}} \left(h(\x)
        + {1 \over 2d} \Delta(bh)\right) q_{n_1 - (n+1)}(\x),
\end{align}
where in the first line, we use the recurrence $q_{n+1} = q_n + (b/2d) \Delta q_n$, and in the second line, we use part (i) of Lemma~\ref{laprearranger}. Therefore, $\sum_{\x\in \Z^{d}} h(\x) q_{n_1 - n}(\x) \le v$.

We now prove that $\sum_{\x\in \Z^{d}} h(\x) q_{n_1 - n}(\x) \ge v$, which is more difficult and requires the hypothesis that $n_1 - n_0 \le T(\rho, b)$. We will first show that it is enough to consider functions $r_n$ with $r_n \ge b$ pointwise everywhere, and then show that $r_n \equiv b$ is an optimal choice.

Suppose $r_{n}:\mathbb Z^d \to [0, 1]$ is a jump probability that satisfies conditions \eqref{e.rcondstwo}(i) and \eqref{e.rcondstwo}(iii). Our inductive hypothesis tells us that
\begin{equation}\label{e.valueinductivehypothesis}V_{(n+1) \to n_1}\left(h + {1 \over 2d}\Delta(r_nh)\right) = \sum_{k\in \Z^{d}} h(\x) q_{n_1 - (n +1)}(\x) + {1 \over 2d} \sum_{k\in \Z^{d}} \Delta(r_nh)(\x) q_{n_1 - (n + 1)}(\x).\end{equation}
We will focus on analyzing the second term in the sum. We use the first part of Lemma \ref{laprearranger}(i) to move the discrete Laplacian to the other factor:
\begin{equation}\label{value.one}\frac{1}{2d}\sum_{k\in \Z^{d}} \Delta(r_nh)(\x) q_{n_1 - (n + 1)}(\x) = \frac{1}{2d}\sum_{k\in \Z^{d}} r_n(\x) h(\x) \Delta q_{n_1 - (n + 1)}(\x).\end{equation}

Consider the modified jump probability $r'(\x) := r_n(\x) \vee b$. Then we claim that for any $\x \in \mathbb Z^d$, $$r_n(\x) h(\x) \Delta q_{n_1 - (n + 1)}(\x) \le r'(\x) h(\x) \Delta q_{n_1 - (n + 1)}(\x).$$ To see this, note that the two sides are equal for $|\x| \le \rho$, because $r_n$ satisfies the third condition of \eqref{e.rcondstwo}. On the other hand, if $|\x| > \rho$, then the discrete Laplacian $\Delta q_{n_1 - n - 1}(\x)$ is nonnegative by definition of $T(\rho, b)$, and $h \ge 0$ and $r' \ge r_{n}$, so the relation also holds. Summing this up over all $\x$, we have
\begin{equation*}
\sum_{\x\in \Z^{d}} r_n(\x) h(\x) \Delta q_{n_1 - (n + 1)}(\x)\le \sum_{\x\in \Z^{d}} r'(\x) h(\x) \Delta q_{n_1 - (n + 1)}(\x).
\end{equation*}
We now perform integration by parts on the term on the right, using Lemma~\ref{laprearranger}(ii) and (iii),
\begin{align*}
    &\sum_{\x\in \Z^{d}} r'(\x) h(\x) \Delta q_{n_1-(n+1)}(\x)
    \\&\qquad\qquad
        = - \frac12\sum_{e=(\x,\y)} \nabla(r'h)(e) \nabla q_{n_1-(n+1)}(e)
    \\&\qquad\qquad
        = - \frac12\left(\sum_{e=(\x,\y)} r'(\y) \nabla h(e) \nabla q_{n_1-(n+1)}(e) + \sum_{e=(\x,\y)} h(\x) \nabla r'(e) \nabla q_{n_1-(n+1)}(e)\right),
\end{align*}

The functions $r', h, q_{n_1-(n+1)}$ are volcanic, the first two by assumption and the last one by Lemma~\ref{constantworks}. This determines the sign of the discrete gradient across any edge. For example, $r'(\y) - r'(\x)$ has the same sign as $h(\y) - h(\x)$. Therefore, the products of two discrete gradients in the above sums are nonnegative. This means that the second sum is nonnegative, and $r' \ge b$, so the whole expression in brackets is at least $\sum_e b \nabla h(e) \nabla q_{n_1 - (n+1)}(e)$. Because there is a minus sign in front, we obtain the upper bound, 
\begin{align*}\sum_{\x \in \mathbb Z^d} r'(\x) h(\x) \Delta q_{n_1-(n+1)}(\x) &\le -\frac12 \sum_{e=(k,\ell)} b \nabla h(e) \nabla q_{n_1-(n+1)}(e) \\&= b \sum_{\x \in \mathbb Z^d} h(\x) \Delta q_{n_1-(n+1)}(\x),\end{align*}  where in the last step, we performed another integration by parts using Lemma \ref{laprearranger}(ii). We now put all this together. By \eqref{e.valueinductivehypothesis} and the inequality we have just proved,
\begin{align*}
    V_{(n+1) \to n_1}\left(h + {1 \over 2d} \Delta(r_n h)\right) &\le  \sum_{\x\in \Z^{d}} h(\x) q_{n_1-(n+1)}(\x) + \frac{1}{2d}\sum_{\x\in \Z^{d}} r'(\x) h(\x)\Delta q_{n_1-(n+1)}(\x) \\&\le \sum_{\x\in \Z^{d}} h(\x) q_{{n_1}-(n+1)}(\x) + \frac{1}{2d} \sum_{\x\in \Z^{d}} b h(\x) \Delta q_{{n_1}-(n+1)}(\x)
    \\&=\sum_{\x\in \Z^d} h(\x) \left(q_{n_1 - (n+1)}(\x) + {b \over 2d} \Delta q_{n_1 - (n+1)}(\x)\right)
    \\&=\sum_{\x\in\Z^d} h(\x) q_{n_1 - n}(\x).
\end{align*}
We have proven the above bound for any jump probability $r_n$ which satisfies \eqref{e.rcondstwo}(i) and \eqref{e.rcondstwo}(iii), so it certainly also holds for jump probabilities $r_n$ which satisfy all three conditions of \eqref{e.rcondstwo}. Taking the supremum over all such $r_n$, we get $v \le \sum_{\x \in \mathbb Z^d} h(\x) q_{n_1 - n}(\x)$. 

Finally, we conclude that $$V_{n \to n_1}(h) = v = \sum_\x h(\x) q_{n_1 - n}(\x).$$ This verifies the induction step, so the formula is true for all $n \in \{n_0, \ldots, n_1\}$, and in particular it is true for $n_0$, which gives us the result $V_{n_0 \to n_1}(h) = \sum_{k\in \Z^{d}} h(\x) q_{n_1 - n_0}(\x)$.
\end{proof}

\subsection{The proof of Theorem \ref{stepfive}}

In order to prove Theorem \ref{stepfive}, we will require an auxiliary result which gives us a bound on the breakdown time $T(\rho, b)$, as defined in \eqref{defoftrb}. 

\begin{thm}\label{inc}
There are positive constants $\Lambda=\Lambda(d)$ and $\x_2=\x_{2}(d,m)$ such that, if $b\leq \tfrac{1}{8}$, and $q^{(b)}=(q_{n}^{(b)})_{n\geq 0}$ be a $b$-lazy random walk, then for all $|\x| \ge \x_2$, and $n \le \Lambda \tfrac{|\x|^2}{6b}$, 
$$q^{(b)}_n(\x) \le q^{(b)}_{n+1}(\x).$$\end{thm}
The proof of Theorem~\ref{inc} can be found in Section~\ref{app}; it yields the following lower bound on the breakdown time. 
\begin{cor}\label{laplacianpositivityuntilstep}Let $T(\rho,b)$ be defined as in \eqref{defoftrb}. There exist positive constants $c=c(d,m)$ and $\rho_{0}=\rho_{0}(d,m)$ such that for any $b\leq \tfrac{1}{8}$ and $\rho\geq \rho_{0}$, $T(\rho,b) \ge c\rho^{2}b^{-1}$.\end{cor}

\begin{proof}
Let $q_{n}^{(b)}$ be defined as in the text above \eqref{defoftrb}. The discrete Laplacian is $$\Delta q_{n-1}^{(b)}(\x) = \sum_{\y \sim \x} q_{n-1}^{(b)}(\y) - q_{n-1}^{(b)}(\x) = \frac{2d}{b}\left(q^{(b)}_n(\x) - q^{(b)}_{n-1}(\x)\right)$$ and by Theorem \ref{inc}, this is nonnegative if $|\x| \ge \x_2$ and $n-1 \le \Lambda |\x|^2/6b$. This implies the corollary with $c = \Lambda/6 = 1/48d^4$ and $\rho_0 = \x_2$.

\end{proof}

 We may now combine all of the prior results to prove Theorem \ref{stepfive}. 
 
 \begin{proof}[Proof of Theorem \ref{stepfive}]
By the hypotheses, there are positive constants $C, \beta$ and $N$ so that for $n \ge N$,
\begin{equation*}
\tildep(\x, n) \ge Cn^{-d\beta}\quad\text{for all $|\x| \le Cn^\beta$}, 
\end{equation*}
where $\tilde{p}(\cdot, n)$ is volcanic for each $n$. 

Fix $n_{1}\geq (2N)\vee (2C^{m})^{\tfrac{1}{dm\beta}}$ and let $\rho := C(n_{1}/2)^\beta$. By the above, we have
\begin{equation*}
\tildep(\x,j)^m\geq \frac{C^{m}}{n_{1}^{dm\beta}}=:b\quad\text{for $|\x| \le \rho$ and $n_{1}/2 \le j \le n_{1}$}. 
\end{equation*}
Observe that by the choice of $n_{1}$, $b\leq \tfrac{1}{2}$.

Let $n_0:=\tfrac{n_{1}}{2}\vee (n_{1} - T(\rho,b))$, where $T(\rho,b)$ is as in (\ref{defoftrb}), and let $\sigma := \sum_{k\in \Z^{d}} \tildep(\x, n_0)$. Given that $\tilde{p}(\cdot,n)$ is volcanic, all hypotheses of Lemma \ref{l.pupperopt} are satisfied. Thus, 
\begin{equation*}
\tilde{p}(k,n_{1})\leq \sigma V^{(\rho, b)}_{n_0 \to n_{1}}(\sigma^{-1}\tilde{p}(\cdot, n_{0})).
\end{equation*}
 By Theorem~\ref{optim}, this implies that for all $k\in \mathbb{Z}^{d}$, 
\begin{align}\label{e.pqbd}
\tilde{p}(k,n_{1})&\leq \sigma \sum_{k\in \Z^{d}}  \sigma^{-1}\tilde{p}(k,n_{0})q^{(b)}_{n_{1} - n_0}(\x)
\notag \\&\le \sigma \sup_{\x\in \Z^{d}} q_{n_{1} - n_0}^{(b)}(\x)\notag
\\&= \sigma q^{(b)}_{n_{1}-n_0}(0), 
\end{align} 
where in the last line, we used the fact that by Lemma \ref{monotonepreserved} (with the function $A(u) = bu$), $q_n^{(b)}(\x)$ is volcanic, since $b\leq \tfrac{1}{2}$. 

In order to conclude, we simply need to estimate the probability that a $b$-lazy random walk $(Y_{n})_{n\geq 0}$ is at $0$ after some number of steps. We seek an estimate which is uniform in $b$, and consequently, we will consider the asymptotics of the characteristic function associated to the $b$-lazy random walk. After one step, the characteristic function $\phi$ is given by 
$$\phi(\vartheta_1, \ldots, \vartheta_d) := \E{\exp(i Y_1 \cdot \vartheta)} = 1 - b + {b \over 2d} \sum_{j=1}^d \cos \vartheta_j.$$ Due to the independent increments of the lazy random walk, $\phi^n$ will be the characteristic function of $Y_n$, and we can recover the probability that $Y_n = 0$ by using Fourier series inversion and evaluating at 0, 
$$q^{(b)}_n(0)=\p{Y_n = 0} = {1 \over (2\pi)^d} \int_{-\pi}^\pi \cdots \int_{-\pi}^{\pi} \phi^n(\vartheta_1, \ldots, \vartheta_d) \,d\vartheta_1 \cdots d\vartheta_d.$$ Now we can use classic asymptotic analysis arguments. Since $1-x \le e^{-x}$, 
\begin{align*}
    q^{(b)}_n(0) &=
        {1 \over (2\pi)^d} \int_{-\pi}^\pi\cdots\int_{-\pi}^{\pi} \left(1-b+b\left(\frac1{2d}\sum_{i=1}^d \cos \vartheta_i\right)\right)^n \,d\vartheta_1 \cdots \,d\vartheta_d
    \\&\le {1 \over (2\pi)^d} \int_{[-\pi, \pi]^{d}} \exp\left(-bn\left(1 - \frac{1}{2d}\sum_{i=1}^d \cos \vartheta_i\right)\right) \,d\vartheta.
\end{align*}

Let $f(\vartheta_1, \ldots, \vartheta_d) := 1 - {1 \over 2d} \sum_{i=1}^d \cos \vartheta_i$. The only minimum of $f$ in the domain of integration is at $\vartheta = 0$, where $f = 0$. By Laplace's method, we obtain the asymptotic estimate that for $t\gg 1$, 
$$H(t) := \int_{[-\pi, \pi]^{d}} \exp(-tf(\vartheta)) \,d\vartheta \sim \frac{K}{t^{d/2}},$$ 
where $K=K(d)>0$. Since $H(t)$ is bounded for all $t$, we conclude that there exists $K=K(d)$ such that $H(t) \le K/t^{d/2}$ for all $t \ge 0$. This constant depends only on $d$, not on the probability $b$, and hence $q^{(b)}_{n}(0)\leq K(bt)^{-\tfrac{d}{2}}$. Continuing from \eqref{e.pqbd}, this implies that 
\begin{equation*}
\tilde{p}(k,n_{1})\leq \frac{K\sigma}{(b(n_{1}-n_{0}))^{d/2}}. 
\end{equation*}

Finally, we estimate a lower bound on $n_{1} - n_0$.  By Corollary~\ref{laplacianpositivityuntilstep}, we have that $T(\rho, b) \ge c\rho^2/b$, where $c=c(d,m)$. Therefore, by the choices of $n_{1}$ and $n_{0}$, we have
\begin{equation*}
b(n_{1}-n_0) \ge b\min\{n_{1}/2, T(\rho, b)\}\geq \min\{\tfrac{bn_{1}}{2}, c\rho^2\}=\min\{cn_{1}^{1-dm\beta}, cn_{1}^{2\beta}\}.
\end{equation*}
But these are of the same order in $n$, since $1 - dm\beta = (dm+2)\beta - dm\beta = 2\beta$. Therefore, for all $k\in \Z^{d}$, there exists $C'=C'(C, d,m, \sigma)$ such that 
$$\tilde{p}(k,n_{1}) \le {K\sigma \over (b(n_{1}-n_0))^{d/2}} \le C'(n_{1}^{-2\beta})^{d/2} = {C' \over n_{1}^{d\beta}}.$$
This works for all $n_{1}$ sufficiently large, depending on $C, d, m, N$. Finally, by enlarging the constant $C'=C'(C, d,m, \sigma, N)$, we can obtain the desired estimate for all $n_{1}\geq N$. 
\end{proof}

\section{The proof of Theorem \ref{prop.pfacts}(i)}\label{s.p1}
Recall from \eqref{e.pdef} that $p(k,n)$ is the probability mass function of $X^{n}$, where $(X^{n}, n\geq 0)$ is $\CM(m,d)$-cooperative motion started from $0\in \Z^{d}$. As in the prior section, we will frequently write $p_{n}(k)=p(k,n)$ and consider the sequence of functions $p=(p_{n})_{n \geq 0}$. The goal in this section is to prove Theorem \ref{prop.pfacts}(i), the statement that there exists $C=C(d,m)>0$ such that $p(\x, n) \le Cn^{-d\beta}$. 

While Theorem \ref{stepfive} yields a good approach to proving upper bounds for solutions of \eqref{timeevol} which are volcanic in space for each time step (we hereby call such sequences of functions ``volcanos''), the sequence  $p=(p_{n})_{n\geq 0}$ is not a volcano. Thus, we begin by first finding a volcano $\tilde{p}=(\tilde{p}_{n})_{n\geq 0}$ such that $\tilde{p}_{n}$ lies above $p_{n}$ for some $n$ (see Section \ref{psmaller}). In order to verify the hypotheses of Theorem \ref{stepfive} for $\tilde{p}$, namely a sufficient lower bound in a ball centered at the origin, we compare $\tilde{p}$ to a \emph{discretization} of the Barenblatt solution $\bar{u}=\bar{u}^{(R, \Ga)}$, for suitably chosen parameters. The discretization of $\bar{u}$ is \emph{not} a solution of \eqref{timeevol} on $[0, \infty)$, however it turns out that is \emph{approximately} a solution, in a sense which we make precise below. This leads us to develop a theory of comparison between solutions and approximate solutions, which is how we will establish the desired lower bound on $\tilde{p}$. By finally applying Theorem \ref{stepfive}, we will complete the proof of Theorem \ref{prop.pfacts}(i).

\subsection{Preliminaries}

\subsubsection{Finding a volcano above $p_{n}$ for some $n$.}\label{psmaller} Many of the techniques from the prior sections are developed for volcanic functions. The actual probability distribution of cooperative motion at a fixed time is generally not volcanic; for example, $p_{1}(0) = 0$ and $p_1(e_1) = (2d)^{-1}$.  Instead, we use the following lemma to find a solution $\tilde{p}=(\tilde{p}_{n})_{n\geq 0}$ of \eqref{timeevol} with initial condition $\tilde{p}_{0}$ such that $\tilde{p}_{n}$ is volcanic for all $n\geq 0$ and there exists $N_{0}$ such that $\tilde{p}_n(\x)\ge{p}_{n+N_{0}}(\x)$ for all $n \ge 0$.

\begin{lem}
    \label{lemma.gives.tildep} There exists $N_{0}=N_{0}(d,m)\in \mathbb{N}$ and a solution $\tildep_n$ of the scheme (\ref{timeevol}) on $[0, \infty)$ with initial condition $\tilde{p}_{0}$ such that $\tilde{p}_{n}$ is volcanic for all $n\geq 0$,  $|\tildep_0|_1 <\infty$ and $$0 \le p_{n+N_{0}}(\x) \le \tildep_n(\x) \le \frac{1}{3} \qquad \text{for all }\x \in \mathbb Z^d, n \ge 0.$$
\end{lem}

\begin{proof} Since $d\geq 2$, $p_1(\x)=(2d)^{-1}\indc_{\pm e_{i}}(k)$. Set $$\tildep_0(\x) = \begin{cases}\frac{1}{2d}&\text{ if $\x=0, \pm e_{i}$}\\0&\text{ otherwise}.\end{cases}$$ Then $\tilde{p}_{0}$ is symmetric and nondecreasing towards the origin. Moreover, $p_1 \le \tildep_0\le1/3$. Let $\tilde{p}=(\tilde{p}_{n})_{n\geq 0}$ be a solution of \eqref{timeevol} with initial condition $\tilde{p}_{0}$. Then by Corollary \ref{cor.Sell1}, $p_{n+1} = \scheme^n p_1 \le \scheme^n \tildep_0 = \tildep_n$, and by Lemma \ref{monotonepreserved}, $\tildep_n$ is volcanic for every $n$. This implies the result with $N_{0} = 1$. Also by Corollary \ref{cor.Sell1}, for every $n$, 
\begin{equation*}
\sum_{\x\in \Z^{d}} |\tildep_n(\x)| =\sum_{\x\in \Z^{d}} |\tildep_0(\x)|= (2d)^{-1}(2d+1).
\end{equation*}

\end{proof}

\subsubsection{Approximate subsolutions}
We say $q=(q_{n})_{n\geq 0}\subset \mathcal{B}^{+}_{\sfrac{1}{2}}(\Z^{d})\cap L^{1}(\Z^{d})$ is a {\em subsolution} of \eqref{e.operator} on $[n_{0}, \infty)$ with initial condition $f: \Z^{d}\rightarrow \R$ if 
\begin{equation*}
\begin{cases}
q_{n+1} \le \scheme q_n&\text{for $n\geq n_0$},\\
q_{n_{0}}=f.
\end{cases}
\end{equation*}
One can see by induction and Corollary \ref{cor.Sell1} that if $\tilde{p}=(\tildep_n)_{n\geq 0}$ is a solution of \eqref{timeevol} with $\tilde{p}_{0}\in \mathcal{B}^{+}_{\sfrac{1}{2}}(\Z^{d})$, and $q=(q_n)_{n\geq 0}$ is a subsolution of \eqref{timeevol} with initial condition $q_{0}$, and $q_{0}\leq \tilde{p}_{0}$, then $q_{n} \le \tildep_n$ for $n \ge 0$.

This idea is a little too rigid to be useful for this scheme, so we generalize it and introduce the idea of an \emph{approximate} subsolution. We begin by defining a distance.

If $q, r \in  \mathcal{B}^{+}_{\sfrac{1}{2}}(\Z^{d})\cap L^{1}(\Z^{d})$, we define the \emph{asymmetric distance} $$\dist[q,r] := \sum_{\x\in\mathbb{Z}^d}\max\{q(\x)-r(\x),0\}=\sum_{\x\in \mathbb{Z}^{d}} \left(q(\x)-r(\x)\right)_{+}.$$ 
Intuitively, this measures how much $q$ sticks up above $r$. If $q \le r$ pointwise, then the asymmetric distance is zero.

The asymmetric distance satisfies the triangle inequality, $\dist[q,r] \le \dist[q,s] + \dist[s,r]$ and $\dist[q,r] \le |q - r|_1$. As one might expect, it is not symmetric: $\dist[q,r] - \dist[r,q] = |q|_1 - |r|_1$, which is not in general zero. 

We note that the asymmetric distance satisfies basic monotonicity properties on $\mathcal{B}^{+}_{\sfrac{1}{2}}(\Z^{d})$. If $q\leq \tilde{q}$, then for any $r$, $\dist[q,r]\leq \dist[q',r]$. Similarly, if $r\leq r'$, then for any $q$, $\dist[q,r']\leq \dist[q,r]$.

For $q=(q_n)_{n\geq 0}\subset\bdd$, let 
\begin{equation*}
\smallest(q) := \inf_n |q_n|_1.
\end{equation*}
This allows us to define an approximate subsolution. 
\begin{dfn}\label{d.approxsub}
We say sequence $q=(q_n)_{n\geq 0}\subset\bdd$ is an {\em approximate subsolution} to \eqref{e.operator} on $[0, \infty)$ with initial condition $q_{0}$ if $$\sum_{n=0}^\infty \dist[q_{n+1},\scheme q_n] \le \frac12 \smallest(q).$$
\end{dfn}

We now show that we can compare approximate subsolutions with solutions, by considering the asymmetric distance (and not total ordering, as in the case of true subsolutions). 

\begin{lem}\label{errorlemma}Let $\tilde{p}=(\tildep_n)_{n\geq 0}$ and $q=(q_n)_{n\geq 0}$ be sequences in $\bdd$, with $\tilde{p}$ a solution of \eqref{timeevol} with initial condition $\tilde{p}_{0}$ and $q$ an approximate subsolution of \eqref{timeevol} with initial condition $q_{0}$. Then for any $n\in \NZ$, $$\dist[q_n,\tildep_n] \le \dist[q_0,\tildep_0] + \frac12 \smallest(q).$$ \end{lem}

\begin{proof} We proceed by induction to show that for any $n\in \NZ$, \begin{equation} \label{eq:errorlemma}\dist[q_n, \tildep_n] \le \dist[q_0, \tildep_0] + \sum_{j=0}^{n-1} \dist[q_{j+1}, \scheme q_j]. \end{equation} 
We consider the base case when $n = 0$, where the sum is empty and this is the trivial statement that $\dist[q_0, \tildep_0] \le \dist[q_0, \tildep_0]$. 

Suppose (\ref{eq:errorlemma}) holds $n$. Let $r_n(\x) := \min\{\tildep_n(\x), q_n(\x)\}$. By monotonicity, $\scheme r_n \le \scheme \tildep_n$ and $\scheme r_n \le \scheme q_n$.

Now using the monotonicity property of asymmetric distance, we have
\begin{equation*}
 \dist[q_{n+1},\tildep_{n+1}]=\dist[q_{n+1},\scheme \tildep_n]\leq  \dist[q_{n+1},\scheme r_n]\leq \dist[q_{n+1},\scheme q_n] + \dist[\scheme q_n,\scheme r_n].
\end{equation*}
But $\scheme r_n \le \scheme q_n$, so $\dist[\scheme q_n,\scheme r_n] = |\scheme q_n - \scheme r_n |_1\leq |q_n - r_n|_1$ by Corollary \ref{cor.Sell1}(iii). Since $r_n = \min\{\tildep_n, q_n\}$, it follows that $| q_n - r_n |_1 = \dist[q_n,\tildep_n]$. Therefore,
\begin{align*}
    \dist[q_{n+1},\tildep_{n+1}]
    &\le \dist[q_{n+1},\scheme q_n] + \dist[q_n,\tildep_n]
    \\[-.9em]&\le \dist[q_{n+1},\scheme q_n] + \dist[q_0,\tildep_0] + \sum_{j=0}^{n-1} \dist[q_{j+1},\scheme q_j]\\[-.9em]
    &=\dist[q_0,\tildep_0]+\sum_{j=0}^{n} \dist[q_{j+1},\scheme q_j].
\end{align*}
Thus, \eqref{eq:errorlemma} holds for all $n\in \NZ$. Since $q$ is an approximate subsolution to \eqref{timeevol} on $[0, \infty)$ with initial condition $q_{0}$, we conclude that 
\[
\dist[q_n, \tildep_n] \le \dist[q_0, \tildep_0] + \frac12\smallest(q).
\hfill\qedhere
\]
\end{proof}

\gdef\duration{\mathop{\rm duration}}
\subsection{The function $\bar{u}_{n}$ as an approximate subsolution}

We recall the Barenblatt solution introduced in Section \ref{ss.bb}:
\begin{equation}\label{hformula}\bar{u}^{(R,\Gamma)}(x, t) := {R^d \Gamma \over r(t)^d} \left(1 - \left({|x| \over r(t)}\right)^2\right)^{1/m}_+, \quad r(t) := R \left(1 + {t \over \time}\right)^\beta, \quad \time := {2d\gamma R^2 \over \Gamma^m},\end{equation}
where $\beta=\frac{1}{dm+2}$ and $\ga=\frac{m\beta}{2(m+1)}$.

For $n\in \N_{0}$, let $\bar{u}^{(R,\Gamma)}_n: \mathbb Z^d \to \mathbb R$ be the function $$\bar{u}^{(R,\Gamma)}_n(\x) = \bar{u}^{(R,\Gamma)}(\x,n).$$ This is the continuous (Barenblatt) solution with discrete integer coordinates plugged in. We will omit the dependence on $R, \Gamma$ from the notation and simply write $\bar{u}_n$ and $\bar{u}(\x, n)$. We will always have $\Gamma<\sfrac{1}{3}$ and since $\int \bar{u}(x,t)\, dx<\infty$, $\bar{u}_n$ belongs to $\bdd$.

The main goal of this subsection is to show that $\bar{u}=(\bar{u}_{n})_{n\geq 0}$ is in fact an approximate subsolution of \eqref{timeevol} with initial condition $\bar{u}_{0}$, in the sense of Definition \ref{d.approxsub}. 
Throughout this section, we use the operator $\mathcal{S}$ given by \eqref{d.Sdef} with $A(u)=u^{m+1}$. In other words, 
\begin{equation}\label{e.bbSdef} 
\scheme q(\x) := q(\x) +\frac{1}{2d}\sum_{\y \sim \x} \left[q(\y)^{m+1} -q(k)^{m+1}\right]. 
\end{equation}
We will be interested in estimating 
 \begin{equation}\label{discretepmeerror}
       \bar{u}_{n+1}(\x) - \scheme  \bar{u}_n(\x) =  \bar{u}_{n+1}(\x) -  \bar{u}_{n}(\x) -
            \frac{1}{2d}
                \sum_{\y \sim \x} \left[ \bar{u}_n(\y)^{m+1} -  \bar{u}_n(\x)^{m+1}\right].
\end{equation} 

\begin{prop}\label{p.approxsub}
For any $\Ga\in (0, \sfrac{1}{3})$, there exists $R=R(d,m)>0$ such that $\bar{u}=(\bar{u}_{n})_{n\geq 0}$ as defined in \eqref{hformula} is an approximate subsolution of \eqref{timeevol} on $[0, \infty)$ with initial condition $\bar{u}_{0}$,  i.e., for $\mathcal{S}$ as in \eqref{e.bbSdef}, we have
\begin{equation}
\sum_{n=0}^\infty \dist[\bar{u}_{n+1},\scheme \bar{u}_n]=\sum_{n=0}^{\infty} \sum_{\x \in \mathbb Z^d} (\bar{u}_{n+1}(\x) - \scheme \bar{u}_n(\x))_+ \le \frac12 \smallest(\bar{u}).
\end{equation}

\end{prop}

In order to prove Proposition \ref{p.approxsub}, we perform a detailed analysis of the function $\bar{u}$, and its discretization, in three regions of $\Z^{d}\times \N$. We will consider the {\em core}, the {\em gap}, and the {\em outer region}, respectively defined by 
\begin{equation}\label{e.regdefs}
\begin{cases}
\mathcal{C}:=\left\{(k,n)\in \Z^{d}\times \N: |k|<r(n)-\sqrt{d}\right\},\\
\mathcal{G}:=\left\{(\x,n) \in \mathbb Z^d \times \mathbb N: r(n) - \sqrt{d} \le |\x| < r(n) + 1\right\},\\
\mathcal{O}:=\left\{(k,n)\in \Z^{d}\times \N : |\x| \ge r(n) + 1\right\},
\end{cases}
\end{equation}
where we recall that $|k|$ refers to the Euclidean norm of $k\in \Z^{d}$. Notice that $\bar{u}(\cdot, t)$ is supported on $\overline{B_{r(t)}}$. Outside of the support, $\bar{u}$ is identically zero, making the discretization $\bar{u}_{n}$ in the outer region $\mathcal{O}$ identically zero as well. 

In the gap $\mathcal{G}$, the function $\bar{u}$ has singular derivatives. The discrete regions we have created give us ``some room'' which will allow us to control the error in a uniform fashion. For $|x|<r(t)$, $\bar{u}$ is smooth, and hence in the core region $\mathcal{C}$, the analysis of the discretization $\bar{u}_{n}$ will rely heavily on derivative estimates of $\bar{u}$, which, to our knowledge, have not previously appeared  in the literature. 

We estimate the total error in the core, the gap, and the outer region in Proposition \ref{p.out}, Proposition \ref{l.gap} and Proposition \ref{l.core} respectively. 
Combining these will straightforwardly yield Proposition \ref{p.approxsub}. 

We first start with total error in the outer region $\mathcal{O}$, which is the most straightforward region to analyze.
\begin{prop}\label{p.out}
Let $\Ga \in (0, \sfrac{1}{3})$, and any $R\in [(2d\ga)^{-1}, \infty)$ where $\ga=\ga(d,m)$ is defined near \eqref{hformula}. For $\mathcal{O}$ as defined in \eqref{e.regdefs} and $\mathcal{S}$ as defined in \eqref{e.bbSdef},
\begin{equation*}
\sum_{(k,n)\in \mathcal{O}} (\bar{u}_{n+1}(k)-\mathcal{S}\bar{u}_{n}(k))_{+}=0.
\end{equation*}
\end{prop}

\begin{proof}
By definition of $\bar{u}$, we have $\bar{u}_{n}(k)=0$ for all $|k|\geq r(n)$. By the definition of $\mathcal{S}$ in \eqref{e.bbSdef}, this implies that $\mathcal{S}\bar{u}_{n}(k)=0$ for all $(k,n)\in \mathcal{O}$. 

We now claim that $\bar{u}_{n+1}(k)=0$ for all $(k,n)\in \mathcal{O}$. By the form of $\bar{u}$, it suffices to show that 
\begin{equation}\label{steponebound.parta}
r(n+1)\leq r(n)+1,
\end{equation} since in this case, $|k|\geq r(n)+1\geq r(n+1)$ implies $\bar{u}_{n+1}(k)=0$. 

To verify \eqref{steponebound.parta}, we use \eqref{hformula} and the fact that $\beta\in (0,1)$, which implies that  
\begin{equation*}
r(n+1)\leq r(n)+\max_{t\geq 0}\left|\frac{dr}{dt}\right|= r(n)+\max_{t\geq 0}\frac{R\beta}{T_{0}}\bigg(1+\frac{t}{T_{0}}\bigg)^{\beta-1}\leq r(n)+\frac{R\beta}{T_{0}}\leq r(n)+\frac{\Ga^{m}}{2d\ga R}.
\end{equation*}
Since $\Ga\in (0, \sfrac{1}{3})$ and $R \geq (2d\ga)^{-1}$, this implies \eqref{steponebound.parta}, and hence $\bar{u}_{n+1}(k)=0$ for all $(k,n)\in \mathcal{O}$. Combining this with the first observation, we obtain the desired result. 
\end{proof}

The next result pertains to the total error in the gap. 

\begin{prop}\label{l.gap}
There exists $c_{g}=c_{g}(d, m)>0$ such that for any $\Ga \in (0,1]$ and any $R\geq 2\sqrt{d}$, for $\mathcal{G}$ as defined in \eqref{e.regdefs} and $\mathcal{S}$ as defined in \eqref{e.bbSdef}, 
\begin{equation}
\sum_{(k,n)\in \mathcal{G}} (\bar{u}_{n+1}(k)-\mathcal{S}\bar{u}_{n}(k))_{+}\leq c_{g}\Ga R^{d-\frac{1}{m}}.
\end{equation}
\end{prop}

In order to prove Proposition \ref{l.gap}, we require some bounds on ``how long'' points spend in the region $\mathcal{G}$. For $k\in \Z^{d}$ such that $(k,n)\in \mathcal{G}$, we define
\begin{equation}\label{e.quantities}
\begin{cases}
\duration(\x) = \#\{n: (\x, n) \in \mathcal{G}\}=\text{``the number of steps that $\x$ is in the gap''.}\\
U(\x) = \max\{\bar{u}(\x, n) \vee \bar{u}(\x, n+1): (\x, n) \in \mathcal{G}\}.
\end{cases}
\end{equation}
Observe that if $(\x,n)\in \mathcal{G}$, then $|\x| \ge r(n) - \sqrt{d} \ge R - \sqrt{d} \ge R/2$ by the choice of $R$, which implies that
\begin{equation}\label{e.Gbd}
\left\{\x: (\x, n)\in \mathcal{G}\right\}\subset \left\{\x: |\x|>R/2\right\}.
\end{equation} 

We now estimate both of the quantities in \eqref{e.quantities}.
\begin{lem}\label{l.dur}
There exists $c_{u}=c_{u}(d,m)>0$ such that for any $\Ga \in (0,1]$, $R\geq 2\sqrt{d}$, for all $k$ with $(k,n)\in \mathcal{G}$, 
\begin{equation*}
U(k)\leq c_{u}{R^d \,\Gamma \over |\x|^{d+1/m}}\quad\text{and}\quad 
 \duration(\x) \,U^m(\x) \le c.
 \end{equation*}
\end{lem}

\begin{proof}
Fix $k$ such that $(k,n)\in \mathcal{G}$. We first prove that there exists a constant $c=c(d,m)$ such that 
 \begin{equation}\label{durationestimate}
\duration(\x) \le c{|\x|^{dm+1} \over R^{dm} \Gamma^m}.
\end{equation}
In order to estimate $\duration(\x)$, we consider the function $h: [0, \infty)\rightarrow [0, \infty)$ defined by $h(s):=r^{-1}(s)=T_{0}((s/R)^{1/\beta}-1)=\time ((s/R)^{dm+2} - 1)$. 

The function $h$ is convex, and there exists $c=c(d,m)>0$ such that
\begin{equation*}
  h'(s) = {\time \over \beta R^{\frac{1}{\beta}}} s^{dm+1}= {2d\gamma R^{2-\frac{1}{\beta}} s^{dm+1} \over \beta \Gamma^m}= c{s^{dm+1} \over R^{dm} \Gamma^m}>0.
\end{equation*}
Set $t_{1}:=h(|k|-1)$ and $t_{2}:=h(|k|+\sqrt{d})$, 
and note that $\left\{n: (k,n)\in \mathcal{G}\right\}\subset \left\{n: t_1<n\leq t_{2}\right\}$. By convexity and the derivative estimate above, we estimate
\begin{align*}
t_{2}-t_{1}\leq h'(|k|+\sqrt{d})\Big[(|k|+\sqrt{d})-(|k|-1)\Big]\leq  {c \over R^{dm} \Gamma^m} (|\x| + \sqrt{d})^{dm+1}.
\end{align*}
This implies that 
\begin{equation*}
\duration(\x)\leq {c \over R^{dm} \Gamma^m} (|\x| + \sqrt{d})^{dm+1}+1.
\end{equation*}
We have assumed that $R\geq 2\sqrt{d}$, so by \eqref{e.Gbd}, $|\x|\geq \sqrt{d}$. Moreover, since $\Gamma \le 1$, we have $|\x|^{dm+1} (R^{dm} \Gamma^m)^{-1}\ge 2^{-dm}|k|\Ga^{-m}\geq \sqrt{d}2^{-dm}$, and hence, by increasing $c$, we obtain \eqref{durationestimate}. 

We now fix $n\in V(\x) := \{n : (\x,n) \in \mathcal{G} \text{ or } (\x,n-1) \in \mathcal{G}\}$. As we show in \eqref{estimateforu} below, $V(\x)$ can be used to control $U(k)$.

We introduce some more notation to simplify (\ref{hformula}). Set $\vartheta=\vartheta(k) := |\x| / r(n)$, so that $$\bar{u}_n(\x) = {R^d \, \Gamma \over r(n)^d} \left(1 - \vartheta^2\right)_+^{1/m},$$
and $\mathcal{G}=\left\{(\x,n): \vartheta \in \big[1- \tfrac{\sqrt{d}}{r(n)}, 1+\tfrac{1}{r(n)}\big)\right\}$. For $(\x, n)\in \mathcal{G}$ such that $|\x|\leq r(n)$,
\begin{equation*}
1-\vartheta^2 = (1+\vartheta)(1-\vartheta) \le 2(1-\vartheta)\leq 2 \frac{\sqrt{d}}{r(n)}=:\frac{c'}{r(n)}.
\end{equation*}

If $(k,n-1)\in \mathcal{G}$, then by (\ref{steponebound.parta}) and the fact that $d> 1$, we have
$$|\x| \ge r(n-1) - \sqrt{d} \ge r(n) - 1 - \sqrt{d}\geq r(n)-2\sqrt{d}.$$
In this case, we obtain the identical estimate above that if $|k|\leq r(n)$, 
\begin{equation*}
1-\vartheta^2\leq 2(1-\vartheta)\leq 2\frac{2\sqrt{d}}{r(n)}\leq \frac{2c'}{r(n)}. 
\end{equation*}

Therefore, for $(\x, n)\in \mathcal{G}$ or $(k, n-1)\in \mathcal{G}$ such that $|\x|\leq r(n)$, we have
\[\bar{u}_n(\x)={R^d\,\Gamma\over r(n)^d}\left(1-\vartheta^2\right)^{1/m} \le 2c'{R^d\,\Gamma \over r(n)^{d+1/m}} \le 2c'{R^d \, \Gamma \over |\x|^{d+1/m}}.\]
If $|\x|>r(n)$, then $\bar{u}_n(\x) = 0$ and the same inequality holds. Therefore, for $\x$ such that $(\x, n)\in \mathcal{G}$, 
\begin{equation}\label{estimateforu} U(\x) = \max\{\bar{u}_{n}(\x) \vee \bar{u}_{n+1}(\x): (\x, n) \in \mathcal{G}\}=\max_{n\in V(k)} \bar{u}_{n}(k)\le 2c'{R^d \,\Gamma \over |\x|^{d+1/m}}. \end{equation}
Combining this with (\ref{durationestimate}), we obtain the existence of $c_{u}=c_{u}(d,m)>0$ such that \begin{equation}\label{durationanduestimate}\duration(\x) \,U^m(\x) \le c_{u}.\end{equation}

\end{proof}

Equipped with Lemma \ref{durationestimate}, we may now prove Proposition \ref{l.gap}. 
\begin{proof}[Proof of Proposition \ref{l.gap}]
Fix $k$ such that $(k,n)\in \mathcal{G}$. Then from (\ref{discretepmeerror}), dropping negative terms from the right hand side, and then taking the positive part, we obtain
\[
(\bar{u}_{n+1}(\x) - \scheme \bar{u}_n(\x))_{+} \le (\bar{u}_{n+1}(\x) - \bar{u}_n(\x))_{+} + \bar{u}_{n}(\x)^{m+1}.
\]
Since $r(t)$ is monotone, by the definition of $\mathcal{G}$, $\left\{n: (k,n)\in \mathcal{G}\right\}$ is a discrete interval. Moreover, since $t\mapsto \bar{u}(k,t)$ is unimodal in $t$ by Lemma \ref{l.baruprop}(i), it follows that $E:=\left\{n: \bar{u}_{n+1}(k)-\bar{u}_{n}(k)>0\right\}$ is also a discrete interval. Therefore, 
\begin{align*}
    \sum_{n: (k,n) \in \mathcal{G}} (\bar{u}_{n+1}(\x) - \bar{u}_{n}(\x))_{+}\leq \sum_{n\in E: (k,n)\in \mathcal{G}} (\bar{u}_{n+1}(\x) - \bar{u}_{n}(\x))
        \le \max_{n: (k,n) \in \mathcal{G}}\bar{u} (k, n+1)\le U(k), 
\end{align*}
where we used that the sum telescopes, and that $\bar{u}\geq 0$ everywhere. By Lemma \ref{l.dur}, this implies that there exists $c=c(d,m)$ such that
\begin{align*}
\sum_{n: (k,n) \in \mathcal{G}} (\bar{u}_{n+1}(\x) - \scheme \bar{u}_n(\x))_{+} &\le U(k)+\sum_{n: (k,n) \in \mathcal{G}} \bar{u}_{n}(\x)^{m+1}\\
&\leq U(k)+\duration(k)U(k)^{m+1}\\
&\leq cU(k).
\end{align*}
Combining this with Lemma \ref{l.dur} and \eqref{e.Gbd}, we obtain that
\begin{equation}\label{gaperrorfinal}
\begin{aligned}
\sum_{(k,n)\in \mathcal{G}} (\bar{u}_{n+1}(\x) - \scheme \bar{u}_n(\x))_{+}
&\leq \sum_{\x \in \mathbb Z^d:|\x| \ge R/2}
       cU(\x)\\
    &\le cR^d\,\Gamma
        \sum_{\x \in \mathbb Z^d: |\x| \ge R/2} \frac1{|\x|^{d+1/m}}
    \\&\le cR^d \,\Gamma \int_{R/2}^{\infty} \frac{r^{d-1}}{r^{d+1/m}}\, dr=:c_{g}\Ga R^{d-\frac{1}{m}},
\end{aligned}
\end{equation}
where $c_{g}=c_{g}(d,m)$. 

\end{proof}

The final preliminary result concerns the core region. 
\begin{prop}\label{l.core}
There exists $c_{c}=c_{c}(d,m)>0$ such that for any $\Ga>0$, $R>0$, for $\mathcal{C}$ as defined in \eqref{e.regdefs} and $\mathcal{S}$ as defined in \eqref{e.bbSdef},  
\begin{equation*}
\sum_{(k,n)\in \mathcal{C}} (\bar{u}_{n+1}(k)-\mathcal{S}\bar{u}_{n}(k))_{+}\leq c_{c}R^{d-2} \,\Gamma^{m+1}.
\end{equation*}
\end{prop}

In the core region, we know that the continuous function $\bar{u}$ solves the porous medium equation $\partial_t \bar{u} = {1 \over 2d} \Delta(\bar{u}^{m+1})$. By (\ref{discretepmeerror}) this implies 
\begin{align*}
    \bar{u}_{n+1}(\x) - \scheme \bar{u}(\x, t)&= \underbrace{\bar{u}(\x, n+1) - \bar{u}(\x, n) - {\partial \bar{u} \over \partial t}(\x,n)}_{\text{(time error)}}\\
        &\quad+\underbrace{{1 \over 2d} \Delta (\bar{u}^{m+1}(k,n)) - {1 \over 2d} \sum_{\y \sim \x} [\bar{u}(\y, n)^{m+1} - \bar{u}(\x, n)^{m+1}]}_{\text{(space error)}}.
\end{align*}

We now estimate the contributions of these two error terms individually. Proposition \ref{l.core} is an immediate consequence of the following two lemmas. 
\begin{lem}\label{l.coret}
There exists $c_{c}=c_{c}(d,m)>0$ such that for any $\Ga>0$ and $R>0$, for $\mathcal{C}$ as defined in \eqref{e.regdefs}, \begin{equation*}
\sum_{(k,n)\in \mathcal{C}}\left(\bar{u}(\x, n+1) - \bar{u}(\x, n) - {\partial \bar{u} \over \partial t}(\x,n)\right)_{+}\leq c_{c}R^{d-2} \,\Gamma^{m+1}.
\end{equation*}

\end{lem}

\begin{lem}\label{l.corex}
For $\mathcal{C}$ as defined in \eqref{e.regdefs}, 
\begin{equation*}
\sum_{(k,n)\in \mathcal{C}}\left({1 \over 2d} \Delta (\bar{u}^{m+1}(k,n)) - {1 \over 2d} \sum_{\y \sim \x} [\bar{u}(\y, n)^{m+1} - \bar{u}(\x, n)^{m+1}]\right)_{+}=0. 
\end{equation*}
\end{lem}

\begin{proof}[Proof of Lemma \ref{l.coret}]
As a consequence of Taylor's theorem, there exists $s\in [n, n+1]$ such that 
\begin{equation*}
\bar{u}(\x, n+1) - \bar{u}(\x, n) - \partial_t \bar{u}(\x, n) = \frac12 \partial^2_t \bar{u}(\x, s), 
\end{equation*}
and thus, we will be concerned with estimating $\left(\partial^2_t \bar{u}(\x, t)\right)_{+}$. 
Recall the parameterization $\vartheta=\vartheta(x,t)=|x| / r(t)$; that in Lemma \ref{l.baruprop}(ii), we computed $\partial_{t}\bar{u}$ in \eqref{first.time.derivative}; and that $\tfrac{\partial \vartheta}{\partial t} =- \tfrac{\vartheta\beta}{(t + \time)}$. Thus, by a direct computation, we have that
\begin{equation}\label{eq:partial2t}
\partial^{2}_{t}\bar{u}(x, t) = {R^d\Gamma(1 - \vartheta^2)^{\tfrac{1}{m}}\over r(t)^d m(t+T_0)^2} \left[-\left({2\beta \over 1 - \vartheta^2} - 1\right) + {1 \over m} \left({2\beta \over 1 - \vartheta^2} - 1\right)^2 - {4\beta^2 \vartheta^2 \over (1 - \vartheta^2)^2}\right], 
\end{equation}
and combining terms, it follows that there exist positive constants $K_{1}, K_{2}$, and $K$, depending only on $d,m$, such that 
\begin{align}\label{gbound}
\partial^{2}_{t} \bar{u}(x, t)&\leq {R^d\Gamma\over r(t)^d m(t+T_0)^2}\left[K_{1}
(1 - \vartheta^2)^{\tfrac{1}{m}} + K_2(1 - \vartheta^2)^{\tfrac{1}{m} - 1} + 4\beta^2 (\tfrac{1}{m} - \vartheta^2) (1 - \vartheta^2)^{\tfrac{1}{m}-2}\right]\notag\\
&\leq {KR^d\Gamma\over r(t)^d m(t+T_0)^2}(1 - \vartheta^2)^{\tfrac{1}{m} - 1}\notag\\&={KR^d\Gamma\over r(t)^d m(t+T_0)^2}\bigg(1 - \Big(\frac{|x|}{r(t)}\Big)^{2}\bigg)^{\tfrac{1}{m} - 1}=: g(x,t),
\end{align}
where for the last inequality, we used that $m\geq 1$. We now simply need to obtain an upper bound on the function $g$. First, since $m\geq 1$, $t\mapsto g(k,t)$ is a decreasing function, and hence by the first observation of the proof, 
\begin{equation}\label{e.taylors}
\bar{u}(\x, n+1) - \bar{u}(\x, n) - \partial_t \bar{u}(\x, n)\leq \frac{1}{2}g(k,n). 
\end{equation}
Furthermore, by definition of $g$, it is clear that $g\geq 0$ on $\mathcal{C}$ and that for any $t$ fixed, $g(k,t)$ is monotone increasing with respect to $|k|$ in $\mathcal{C}$. 

Fix $n$ such that there exists $k$ with $(k,n)\in \mathcal{C}$. We claim that 
\begin{equation}\label{e.core1}
\sum_{k: (k,n)\in \mathcal{C}} g(\x,n)
    \le \int\limits_{B_{r(n)}}
        g(y, n) \,dy,
        \end{equation}
        where $B_{r(n)}$ is the Euclidean ball in $\R^{d}$ of radius $r(n)$. In order to prove \eqref{e.core1}, due to the symmetry of $g$, it is enough to prove an analogous bound in the positive orthant $E_{n}^{+}:=\mathbb{Z}^{d}_{+}\cap \left\{k: (k,n)\in \mathcal{C}\right\}$. More precisely, since $k\in E_{n}^{+}$, we know $|\x| < r(n) - \sqrt{d}$, and hence $[k_{1}, k_{1}+1]\times \ldots [k_{d}, k_{d}+1]\subseteq \left\{|\x|<r(n)\right\}=B_{r(n)}\cap \mathbb{Z}^{d}$. Owing to the monotonicity property of $g$, we conclude
        \[
\sum_{k\in E_{n}^{+}} g(\x,n) \le \int_{\bigcup_{k\in E_{n}^{+}} [k_{1}, k_{1}+1]\times \ldots [k_{d}, k_{d+1}]}g(y,n) \,dy \le \int\limits_{\R^{d}_{+}\cap B_{r(n)}} g(y,n) \,dy,  
\]
where $\mathbb R^d_+ = [0, \infty)^d$. Summing over all orthants, we obtain \eqref{e.core1}. 

We now conclude by \eqref{e.core1} that 
\begin{align*}
\sum_{k: (k,n)\in \mathcal{C}} g(\x,n)
    &\leq {K R^d \Gamma \over r(n)^d(n+T_0)^2}
        \int_{B_{r(n)}}
            \Big(1 - \frac{|y|^{2}}{r(n)^{2}}\Big)^{\tfrac{1}{m}-1} \,dy
    \\&= {K R^d \Gamma \over (n+T_0)^2}
        \int_{B_{1}}
            (1 - z^2)^{\tfrac{1}{m}-1} \,dz
    \\&= {KR^d \,\Gamma \over (n+T_0)^2},
    \end{align*}
    by increasing the value of $K=K(d,m)$ if necessary. 
    
    Combining this with \eqref{e.taylors} and repeatedly adjusting the constant $K$ as needed, we obtain \begin{equation*}
\sum_{(k,n)\in \mathcal{C}}\left(\bar{u}(\x, n+1) - \bar{u}(\x, n) - {\partial \bar{u} \over \partial t}(\x,n)\right)_{+}\le \sum_{n=0}^\infty {KR^d \Gamma \over (n+T_0)^2} \le {KR^d \Gamma \over T_0}=KR^{d-2} \,\Gamma^{m+1}, 
\end{equation*}
where we used the definition of $T_{0}$ from \eqref{hformula}. 
The result follows.
\end{proof}

\begin{proof}[Proof of Lemma \ref{l.corex}]
For fixed $(k,n)\in \mathcal{C}$, we can re-express the space error as 
\[
{1 \over 2d} \Delta(\bar{u}(k,n))^{m+1}- {1 \over 2d} \sum_{\y \sim \x} [\bar{u}(\y, n)^{m+1} - \bar{u}(\x,n)^{m+1}] = \sum_{i=1}^d \mathcal{E}_i(\x, n)
\] 
where $\mathcal{E}_i$ is the error along a single coordinate:
\[
\mathcal{E}_i(\x,n) := {1 \over 2d} \left[{\partial^2 \over \partial^2 x_i} \bar{u}(\x, n)^{m+1} - \left(\bar{u}(\x+e_i, n)^{m+1} - 2\bar{u}(\x, n)^{m+1} + \bar{u}(\x-e_i, n)^{m+1}\right)\right] .
\]

To estimate the error, we observe the following: if $f \in C^4(\R)$, then 
\begin{equation}\label{e.ibp}
f''(x) - (f(x+1) - 2f(x) + f(x-1)) = -\int_{x-1}^{x+1} {(1 - |s-x|)^3 \over 6} f^{(4)}(s) \,ds.
\end{equation}
The proof of \eqref{e.ibp} relies on the integral version of Taylor's theorem, and appears immediately after this argument. 

Using \eqref{e.ibp}, we obtain the alternative formula 
\begin{equation*}
\mathcal{E}_i(k,n) = -{1\over2d} \int_{k_i - 1}^{k_i + 1} {(1 - |s - k_i|)^3 \over 6} \partial^4_i(\bar{u}^{m+1}(k_1, \ldots, k_{i-1}, s, k_{i+1}, \ldots, k_d)) \,ds. \label{errori}
\end{equation*}
By Lemma \ref{l.baruprop}(ii), $\bar{u}^{(4)}(k,n)\geq 0$ for all $(k,n)\in \mathcal{C}$. This completes the proof. 
\end{proof}

\begin{proof}[Proof of \eqref{e.ibp}]
By Taylor's theorem in integral form, we have 
\begin{align*}
f(x+1)&=f(x)+f'(x)+\frac{1}{2}f''(x)+\frac{1}{6}f^{(3)}(x)+\int_{x}^{x+1}f^{(4)}(s)\frac{(1-|s-x|)^{3}}{6}\, ds,\\
f(x-1)&=f(x)-f'(x)+\frac{1}{2}f''(x)-\frac{1}{6}f^{(3)}(x)+\int_{x-1}^{x}f^{(4)}(s)\frac{(1-|s-x|)^{3}}{6}\, ds.
\end{align*}
Adding the two expressions and rearranging, we arrive at \eqref{e.ibp}. 
\end{proof}

In order to conclude, we require one additional result to control $\smallest(\bar{u})$. 

\begin{lem}\label{small} There is a constant $c_0 = c_0(d,m)>0$ such that for any $\Ga\in (0,1]$ and $R\geq 2\sqrt{d}$, $\smallest(\bar{u}) \ge c_0R^d \Gamma$.\end{lem}

\emph{Proof.} According to the form of the Barenblatt solution, as previously noted, at a fixed time $n$, $\bar{u}_n(\x)$ decreases as $|\x|$ increases. Consequently, if $\x\in\mathbb Z^{d}_{+}$,
\[
\bar{u}_n(\x) \ge \int_{[\x_1,\x_1+1] \times \cdots[\x_d,\x_d+1]} \bar{u}_{n}(x) \,dx.
\]
 This implies that 
 \[
 |\bar{u}_n |_1 = \sum_{\x \in \mathbb Z^d} \bar{u}_n(\x) \ge \sum_{k\in \mathbb{Z}^{d}_{+}} \bar{u}_{n}(k)\geq  \int_{[1, \infty)^d} \bar{u}_{n}(x)\, dx. 
 \]

Using \eqref{hformula} and a change of variables, with $B_{r(n)}$ denoting the Euclidean ball of radius $r(n)$, we have
\begin{align*}
    \int_{[1, \infty)^{d}} \bar{u}_n(x) \,dx
        &= {R^d \Gamma \over r(n)^d} \int\limits_{[1, \infty)^{d}\cap B_{r(n)}} \left(1 - \left({|x| \over r(n)}\right)^2\right)^{1/m} \,dx
        \\&= R^d \Gamma \int\limits_{[r(n)^{-1}, \infty)^{d}\cap B_{1}} (1-|y|^2)^{1/m} \,dy.
\end{align*}

Since $R = r(0) \geq 2 \sqrt{d}$, it follows that for every $n\geq 0$, $[r(n)^{-1}, \infty)^{d}\supseteq [(2\sqrt{d})^{-1}, \infty)^{d}$. Thus 
$$\int_{[1, \infty)^{d}} \bar{u}_n(x) \ge R^d \Gamma \int\limits_{[(2\sqrt{d})^{-1}, \infty)^{d}\cap B_{1}} (1-|y|^2)^{1/m} \,dy=: c_{0}R^{d} \Gamma.$$ 
It is clear that $c_{0}=c_{0}(d,m)>0$, and hence $|\bar{u}_n|_1=\sum_{k} \bar{u}_n(\x) \ge c_0 R^d \Gamma$. Since $n$ was arbitrary, the result follows. \qed

\begin{proof}[Proof of Proposition \ref{p.approxsub}]
We now make explicit choices of $R$ and $\Gamma$. We allow arbitrarily small $\Gamma\in (0,\sfrac{1}{3})$, while $R\in [1, \infty)$ will be moderately large and fixed, given by
\begin{equation}\label{e.Rdef}
R = \max\left\{2\sqrt{d}, (2d\ga)^{-1}, \left( 4c_g\over c_0\right)^{m}, \sqrt{\frac{4c_c}{c_{0}}}\right\},
\end{equation} 
where the constant $c_0$ is from Lemma \ref{small}, $c_g$ is from Proposition \ref{l.gap}, and $c_c$ is from Proposition \ref{l.core}. All these constants (and hence $R$ itself) depend only on $d$ and $m$.

Choosing $R$ as in \eqref{e.Rdef} and $\Ga<1/3$, the hypotheses of Proposition \ref{p.out}, Proposition \ref{l.gap} and Proposition \ref{l.core} hold simultaneously. It follows that 
\begin{align*}
\sum_{(k,n)\in \Z^{d}\times \N}& (\bar{u}_{n+1}(\x) - \scheme \bar{u}_n(\x))_+\\
&\leq \sum_{(k,n)\in \mathcal{O}}(\bar{u}_{n+1}(\x) - \scheme \bar{u}_n(\x))_+ +\sum_{(k,n)\in \mathcal{G}}(\bar{u}_{n+1}(\x) - \scheme \bar{u}_n(\x))_+\\
&\quad +\sum_{(k,n)\in \mathcal{C}}(\bar{u}_{n+1}(\x) - \scheme \bar{u}_n(\x))_+\\
&\leq c_{g}\Ga R^{d-\frac{1}{m}}+c_{c}R^{d-2} \,\Gamma^{m+1}. 
\end{align*}
To bound each of these terms, we will use Lemma \ref{small} which says $\smallest(\bar{u})\geq c_0 R^d \Gamma$. 

For the first term, by the choice of $R$, $R \ge \left(\tfrac{4c_g}{c_{0}}\right)^{m}$, so 
\begin{equation*}
c_{g}\Ga R^{d-\frac{1}{m}}\leq c_{g}\Ga R^{d}\frac{c_{0}}{4c_{g}}\leq \frac{1}{4}\smallest(\bar{u}). 
\end{equation*}

Similarly, for the second term, by choice of $R$, $R^{2}\geq 4c_{c}c_{0}^{-1}$. Hence
\begin{equation*}
c_{c}R^{d-2} \,\Gamma^{m+1}\leq \frac{c_{0}}{4}R^{d}\Ga^{m+1}\leq \frac{1}{4}\Ga^{m}\smallest(\bar{u})\leq \frac{1}{4} \smallest(\bar{u}).
\end{equation*}
This yields the desired claim. 
\end{proof}

\subsection{The proof of Theorem \ref{prop.pfacts}(i).}

We now recall the function $(\tilde{p}_{n})_{n\geq 0}$ defined in Lemma \ref{lemma.gives.tildep}, which is a solution of \eqref{timeevol} on $[0, \infty)$ with initial condition $\tilde{p}_{0}$, where $\tildep_n(\x)\in [0, \sfrac{1}{3}]$ is volcanic for all $n$. Moreover, by definition of the scheme, $\tildep_n(\x)$ is strictly positive whenever $|\x|_1 \le n$.

We aim to control $\dist[\bar{u}_{n}, \tilde{p}_{n+n_{0}}]$ for some $n_{0}\in \mathbb{N}$ to be chosen. Since $(\tilde{p}_{n+n_{0}})_{n\geq 0}$ is still a solution of the scheme \eqref{timeevol} on $[0, \infty)$ with initial condition $\tilde{p}_{n_{0}}$, as a consequence of Proposition \ref{p.approxsub} and Lemma \ref{errorlemma}, for any $\Ga\in (0, \sfrac{1}{3})$ and for a fixed $R=R(d,m)$, for $\bar{u}=\bar{u}^{(R, \Ga)}$, we have
\begin{equation*}
\dist[\bar{u}_{n}, \tilde{p}_{n+n_{0}}]\leq \dist[\bar{u}_{0}, \tilde{p}_{n_{0}}]+\frac{1}{2}\smallest(\bar{u}).
\end{equation*}

We now choose $n_{0}$ (and $\Ga$) so that $\dist[\bar{u}_{0}, \tilde{p}_{n_{0}}]=0$. Set 
\begin{equation}\label{e.n0ga}
n_0 = \lceil R \sqrt{d} \rceil\quad\text{and}\quad \Gamma = (\min_{|\x| \le R} \tildep_{n_0}(\x))\wedge \sfrac{1}{3}.
\end{equation}
Observe that both $n_{0}$ and $\Ga$ depend only on $d,m$, since $R=R(d,m)$. Recall that $\norm{\bar{u}}_{L^{\infty}(\R^{d}\times [0, \infty))}\leq \Ga$, and  
$R<n_{0}$, we know $\Ga>0$. Letting $\bar{u}_{n}=\bar{u}^{(R,\Ga)}(\cdot, n)$ for these choices of $R, \Ga$, we have
\begin{equation*}
\dist[\bar{u}_{0}, \tilde{p}_{n_{0}}]=\sum_{\x\in \mathbb{Z}^{d}} \left(\bar{u}_{0}(\x)-\tilde{p}_{n_{0}}(\x)\right)_{+}\leq \sum_{|\x| \le R} (\Gamma - \tildep_{n_0}(\x))_+ = 0.
\end{equation*}
Combining this with the prior display, we have that for these choices of $n_{0}, \Ga, R$,  
\begin{equation}\label{e.intermed}
\dist[\bar{u}_{n}, \tilde{p}_{n+n_{0}}]=\sum_{\x\in \mathbb{Z}^{d}} \left(\bar{u}_{n}(\x)-\tilde{p}_{n+n_{0}}(\x)\right)_{+}\leq \frac{1}{2}\smallest(\bar{u}).
\end{equation}

This observation now allows us to prove the following lower bound on $\tilde{p}$ inside the support of the associated Barenblatt solution $\bar{u}$. 
\begin{lem}\label{lowerboundonp}
Let $R$ be defined as in \eqref{e.Rdef} and $n_{0}, \Ga$ be defined by \eqref{e.n0ga}, which only depend on $d,m$. There are positive constants $c=c(d, m)<1$ and $N=N(d,m)$ such that for all $n\geq N$, $$\tildep_{n+n_0}(\x) \ge \frac{c}{r(n)^d} \quad\text{for all $|\x|\leq cr(n)$.}$$
\end{lem}

\begin{proof}
 Let $g(z) := R^d \Gamma (1 - |z|^2)^{1/m}_+$, and note that $\int_{\R^{d}} g(z)\, dz>0$. We consider the set $A(c) = \{z \in \mathbb R^d: |z_1|, \ldots, |z_d| \ge c\}$ which is simply the complement of the ball of radius $c$ in the $L^{\infty}$-norm. As $c \to 0+$, as a consequence of the monotone convergence theorem, we have $$\lim_{c \to 0+} \int_{A(c)} (g(z) - c)_+ \,dz = \int_{\mathbb R^d} g(z) \,dz.$$ 
Consequently, we may fix a constant $c=c(d,m)\in (0,1)$ such that 
\begin{equation}\label{e.c3choice}
\int_{A(c)} (g(z) - c)_+ \,dz \ge \frac23 \int_{\mathbb R^d} g(z) \,dz.
\end{equation}

Let $A_n(c) := \{\x \in \mathbb Z^d: |\x_1|, \ldots, |\x_d| \ge cr(n)\}$. Recall that by Lemma \ref{lemma.gives.tildep}, $\tilde{p}_{n+n_{0}}$ is volcanic for every $n$. It follows (by applying the volcanic property iteratively over neighbors) that for any $\x\in A_{n}(c)$ and $\ell\in \Z^{d}$ with $|\y| \le cr(n)$, $\tilde{p}_{n+n_{0}}(\x) \le \tilde{p}_{n+n_{0}}(\y)$. Therefore, by \eqref{e.intermed}, for any fixed $|\y| \le cr(n)$,
\begin{equation} \label{lowerboundonp.one}
\sum_{\x\in A_{n}(c)} \left(\bar{u}_{n}(\x)-\tilde{p}_{n+n_{0}}(\y)\right)_{+}\leq  \sum_{\x\in \mathbb{Z}^{d}} \left(\bar{u}_{n}(\x)-\tilde{p}_{n+n_{0}}(\x)\right)_{+}\leq \frac{1}{2}\smallest(\bar{u})\leq \frac{1}{2}|\bar{u}_{n}|_{1}. 
\end{equation}

By definition of $\bar{u}_{n}$ in \eqref{hformula}
\begin{equation}\label{e.grel}
\bar{u}_n(r(n)z) = {R^d \Gamma \over r(n)^d} \left(1 - |z|^2\right)^{1/m}_+ = \frac{g(z)}{r(n)^{d}}.
\end{equation}
Since $r(n)\to \infty$ as $n\to \infty$, we may apply a Riemann approximation to conclude that 
\[
\lim_{n \to \infty} |\bar{u}_n |_1= \lim_{n \to \infty} \frac{1}{r(n)^{d}} \sum_{ (r(n)z)\in \mathbb{Z}^{d}} g(z) = \int_{\mathbb R^d} g(z) \,dz.\]
As another consequence of \eqref{e.grel}, we have
\begin{align*}\lim_{n \to \infty} \sum_{\x \in A_n(c)} \left(\bar{u}_n(\x) - \frac{c}{r(n)^d}\right)_+ = \lim_{n \to \infty} \frac{1}{r(n)^{d}} \sum_{r(n)z \in A_n(c)} (g(z) - c)_+ = \int_{A(c)} (g(z) - c)_+ \,dz.\end{align*}
By \eqref{e.c3choice} and the prior display, we have
\begin{equation*}
\lim_{n \to \infty} \sum_{\x \in A_n(c)} \left(\bar{u}_n(\x) - \frac{c}{r(n)^d}\right)_+\geq \lim_{n \to \infty} \frac{2}{3}|\bar{u}_n |_1>0. 
\end{equation*}
Consequently, there is $N>0$ so that for all $n \ge N$, $$\sum_{\x \in A_n(c)} \left(\bar{u}_n(\x) - \frac{c}{r(n)^d}\right)_+\geq  \frac{5}{9}|\bar{u}_n |_1>0. $$ 
Now for the purposes of contradiction, fix $n \ge N$, and suppose that there is a point $\y$ such that $|\y| \le cr(n)$ and $\tilde{p}_{n+n_{0}}(\y) < \tfrac{c}{r(n)^d}$. Then by the prior display, $$\sum_{\x \in A_n(c)} (\bar{u}_n(\x) - \tilde{p}_{n+n_{0}}(\y))_+ \ge \sum_{\x \in A_n(c)} \left(\bar{u}_n(\x) - \frac{c}{r(n)^d}\right)_+ \ge \frac59 |\bar{u}_n |_1,$$ which contradicts the inequality (\ref{lowerboundonp.one}). This implies that $\tilde{p}_{n+n_{0}}(\x) \ge \tfrac{c}{r(n)^d}$ if $|\x| \le cr(n)$ and $n \ge N$. 
\end{proof}

Equipped with Lemma \ref{lowerboundonp}, we are now ready to finish the proof of Theorem \ref{prop.pfacts}(i). 
\begin{proof}[Proof of Theorem \ref{prop.pfacts}(i)]
By Lemma \ref{lowerboundonp}, for $n_{0}$ as in \eqref{e.n0ga}, and by the definition of $r$ in \eqref{hformula}, by adjusting $c=c(d,m)$, we have that for all $n\geq N$, 
\begin{equation*}
\tildep_{n+n_0}(\x) \ge \frac{c}{r(n)^d}=\frac{c}{(T_{0}+n)^{d\beta}} \quad\text{for all $|\x|\leq cr(n)$.}
\end{equation*}
Consequently, for all $n\geq N\vee T_{0}$, by adjusting $c=c(d,m)$ and since $n_{0}>0$, we have
\begin{equation*}
\tildep_{n+n_0}(\x) \ge \frac{c}{(2n)^{d\beta}}\geq \frac{c}{(n+n_{0})^{d\beta}}.
\end{equation*}
By Theorem \ref{stepfive}, it follows that for all $n\geq n_{0}+ (N\vee T_{0})$, $\sup_{k} \tilde{p}(k,n)\leq C'n^{-d\beta}$. For $N_{0}$ as in Lemma \ref{lemma.gives.tildep}, this implies that there is $C''=C''(d,m)$ such that for all $n\geq n_{0}+ N\vee T_{0}$, we have 
\begin{equation*}
p(k,n+N_{0})\leq \tilde{p}(k,n)\leq C'n^{-d\beta}\leq C''(n+N_{0})^{-d\beta}\quad\text{for all $k\in \mathbb{Z}^{d}$.}
\end{equation*}
Since $n_{0}, N_{0}, N, T_{0}$ all depend only on $d,m$ and are all finite, this statement can then be upgraded to the statement that there exists $C=C(d,m)>0$ such that 
\begin{equation*}
p(k,n)\leq Cn^{-d\beta}\quad\text{for all $n\in \mathbb{N}$ and all $k\in \mathbb{Z}^{d}$.} \qedhere
\end{equation*}
\end{proof}

\section{The proof of Theorem \ref{prop.pfacts}(ii) and Theorem \ref{prop.pfacts}(iii)}\label{s.p23}
In this section, we prove Theorem \ref{prop.pfacts}(ii) and Theorem \ref{prop.pfacts}(iii). 

\subsection{The proof of Theorem \ref{prop.pfacts}(ii)}
We prove that there is a constant $C=C(d,m)>0$ such that for all $r > 0$ and $n \in \N$, 
\begin{equation*}
\sum_{k\in \Z^{d}\cap B^{c}_{rn^{\beta}}} p(k,n)= \sum_{|k|>rn^{\beta}}p(k,n)\leq C \exp(-r/C).
\end{equation*}
We begin by recalling Freedman's inequality, which is a type of concentration inequality. Let $(\mathscr F_n)_{n\geq 0}$ be a filtration and $Y_n$ be a real-valued martingale with $Y_0 = 0$. Let $V_n:=\text{Var}[Y_n \mid \mathscr F_{n-1}]$ be the conditional variance. 
\begin{thm}\label{t.freedman}\cite[Freedman, Theorem 1.6]{freedman}
If $|Y_{n+1} - Y_n| \le 1$ for all $n$, then for any $a, b>0$,
$$\mathbb P\left(\exists n: Y_n \ge a \text{ and } \sum_{j=1}^n V_j \le b\right) \le \exp\left(-{a^2 \over 2(a+b)}\right).$$
\end{thm}

Equipped with this result, we can now proceed with the proof of Theorem \ref{prop.pfacts}(ii).

\begin{proof}[Proof of Theorem \ref{prop.pfacts}(ii)]

If $r\in (0,1]$, we obtain the desired bound for all $n\in \mathbb{N}$, so long as $C>1.$ Thus, let $r > 1$, and let $(X_j)_{j\in \mathbb{N}}$ be $\text{CM}(m)$-distributed. Fix $i \in [d]$, and let $Y_j$ be the $i$-th coordinate of $X_j$.

At any step of the $\text{CM}(m)$, with probability $p^m / d$, we move in the $i$-th coordinate, and when we do, we take a Bernoulli step that is conditionally independent of the past. Thus, $(Y_j)_{j\geq 0}$ is a martingale, and its conditional variance is given by 
\begin{equation*}
V_{j}=\text{Var}[Y_j \mid \mathscr F_{j-1}] = \mathbb E[(Y_j - Y_{j-1})^2 \mid \mathscr F_{j-1}] = \frac{1}{d}p^m(X_{j-1}, j-1).
\end{equation*}

By Theorem \ref{prop.pfacts}(i), this has a deterministic upper bound, given by $$V_j = c p^m(X_{j-1}, j-1) \le {c \over j^{dm\beta}}.$$
Since $0 < dm\beta < 1$, we have 
\begin{equation*}
\sum_{j=1}^{n}V_{j}=c\sum_{j=1}^n j^{-dm\beta} = cn^{1-dm\beta} + O(1),
\end{equation*}
which implies $\sum_{j=1}^n V_j \le cn^{1-dm\beta} = Cn^{2\beta}$, where $C=C(d,m)$. Let $b:=Cn^{2\beta}$, so that $\sum_{j=1}^n V_j \le b$.

Let $a:=rn^\beta/\sqrt{d}$. Then by properties of norm-equivalences on $\R^{d}$, we have
$$\{|X_n| \ge rn^\beta\} \subseteq \bigcup_{i=1}^d\{(X_n)_i \ge a\} \cup \bigcup_{i=1}^d \{-(X_n)_i \ge a\}.$$ 
By a union bound, symmetry of the process, and the prior claim, this implies 
$$\mathbb P\left(|X_n| \ge rn^\beta\right) \le 2d\mathbb P\left((X_n)_i \ge a\right) = 2d \mathbb P\left(Y_n \ge a\right) = 2d \mathbb P\left(Y_n \ge a \text{ and }\sum_{j=1}^n V_j \le b\right).$$
We now conclude by using Freedman's inequality (Theorem \ref{t.freedman}). In particular, we have
$$\mathbb P\left(|X_n| \ge rn^\beta\right) \le 2d\mathbb P\left(Y_n \ge a \text{ and } \sum_{j=1}^n V_j \le b\right) \le 2d\exp\left(-{r^2 n^{2\beta}/d \over 2(rn^{\beta}/\sqrt{d} + Cn^{2\beta})}\right).$$
Using that $r > 1$, we conclude with the lower bound
\begin{equation*}
{r^2 n^{2\beta}/d \over 2(rn^{\beta}/\sqrt{d} + Cn^{2\beta})} = \frac{r^2}{2d(rn^{-\beta}/\sqrt{d} + C)} \geq \frac{r^2}{2d(r/\sqrt{d} + C)} \geq \frac{r^2}{C(r+1)} \geq \frac{r}{C}.
\end{equation*}

\end{proof}

\subsection{The proof of Theorem \ref{prop.pfacts}(iii)}
The goal is to obtain an upper bound on the total variation of $p(\cdot, n)$ for a domain $D:=B_{rn^{\beta}}\cap\Z^{d}$ of the form
$$[p(\cdot, n)]_{\TV(D)} := 
\frac{1}{2}\sum_{\{u,v\in D:u\sim v\}} |p(v,n)-p(u,n)|\leq Cr^{d-1}n^{-\beta}.$$

 In order to prove this, we begin by introducing a weaker version of spatial monotonicity. We say that a function $q: \mathbb Z^d \to \mathbb R$ is \emph{almost-volcanic} if, $q$ is symmetric, as in \eqref{e.symmdef}, and for any two neighbours $\x \sim \y$,
\begin{itemize}
\item if $|\x|_1 > |\y|_1 \ge 1$, then $q(\x) \le q(\y)$, and if 
\item if $|\x|_1 \ge 2$, then $q(\x) \le q(0)$.
\end{itemize}
Observe that this definition differs from the definition of volcanic by only requiring that $q(\ell)$ is nondecreasing towards the origin for points $|\y|_{1}\geq 1$, and furthermore that $q(0)\geq q(\x)$ for $|\x|_{1}\geq 2$. We first show that being almost-volcanic is preserved under certain monotonicity conditions on the scheme. Unlike Lemma \ref{monotonepreserved}, the analogous statement for volcanic functions, we state this result directly for the specific function $p_{n}$, the distribution of a \text{CM}($m$)-distributed process at time $n$. 

\begin{lem} For any $n\in \mathbb{N}_{0}$, if $0\leq p_{n}\leq 1/3$ is almost-volcanic, then $p_{n+1}$ is almost-volcanic.
\end{lem}

\begin{proof}
Recall that $p=(p_{n})_{n\geq 0}$ solves the scheme \eqref{timeevol} on $[0, \infty)$ with initial condition $p_{0}=\mathbf{1}_{0}$, and by hypothesis, let us assume that $0\leq p_{n}\leq 1/3$. Since $A(u)=u^{m+1}$ in the scheme \eqref{timeevol}, we have $0\leq A'(\cdot)\leq \sfrac{2}{3}$ for $0\leq u\leq \sfrac{1}{3}$. 

By the same reasoning as in Lemma~\ref{monotonepreserved}, if $p_{n}$ is symmetric, then $p_{n+1}$ is also symmetric. Therefore, it remains to verify that $p_{n+1}$ satisfies the above two additional conditions necessary to be almost-volcanic.

We begin with the first claim, which is that $p_{n+1}(\x) \le p_{n+1}(\y)$ for $|\x|_1 > |\y|_1 \ge 1$. In the case of $\x, \y$ with $|\y|_1 \ge 2$, this follows immediately from Case 1 in the proof of Lemma~\ref{monotonepreserved}. Therefore, we simply need to check this inequality in the case when $|\x|_1 = 2, |\y|_1 = 1$.

Suppose $\x$ is a point with $|\x|_1 = 2$. By symmetry, it is enough to consider the cases $\x = 2e_1$ or $\x = e_1 + e_2$. 

\medskip
\emph{Case 1. $\x = 2e_1$.} 
If we look at all neighbors of $e_1$, there are $2d-2$ neighbours the form $e_1 \pm e_j, j \ne 1$, one neighbor which is $2e_{1}$ and one neighbor which is the origin. Thus, for $\ell\sim e_{1}$, since $p_{n}$ is almost-volcanic, the symmetry and monotonicity properties imply that
\begin{equation*}
p_{n}(\ell)\begin{cases}
=p_n(e_1+e_2)&\text{if $\ell=e_1 \pm e_j, j \ne 1$},\\
=p_{n}(2e_{1})&\text{if $\ell=2e_{1}$},\\
\geq p_{n}(2e_{1})&\text{if $\ell=0$}.
\end{cases}
\end{equation*}
Therefore, since $A$ is nondecreasing, we have
\begin{align*}
    p_{n+1}(e_1)
        &= p_n(e_1) + {1 \over 2d} \sum_{\y \sim e_1}
                (A(p_n(\y)) - A(p_n(e_1)))
        \\&\ge p_{n}(e_{1}) + {2d-2 \over 2d} (A(p_n(e_1+e_2)) - A(p_{n}(e_{1}))) +
            {2 \over 2d} (A(p_{n}(2e_{1})) - A(p_{n}(e_{1}))).
\end{align*}

On the other hand, $2e_1$ has $2d-2$ neighbours $2e_1 \pm e_j, j \ne 1$, one neighbour that is $e_1$, and one neighbour that is $3e_1$. Using again the almost-volcanicity, we conclude
\begin{equation*}
p_{n}(\ell)\begin{cases}
\leq p_n(e_1+e_2)&\text{if $\ell=2e_1 \pm e_j, j \ne 1$},\\
\leq p_{n}(2e_{1})&\text{if $\ell=3e_{1}$},\\
=p_{n}(e_{1})&\text{if $\ell=e_{1}$}.
\end{cases}
\end{equation*}

Using the above and the scheme, we have that
 $$p_{n+1}(2e_1) \le p_{n}(2e_{1}) + {2d-2 \over 2d} (A(p_{n}(e_{1}+e_{2})) - A(p_{n}(2e_{1}))) + {1 \over 2d} (A(p_{n}(e_{1})) - A(p_{n}(2e_{1}))) + 0.$$
 
Subtracting these two equations, we get $$p_{n+1}(e_1) - p_{n+1}(2e_1) \ge p_{n}(e_{1})- p_{n}(2e_{1}) + {2d+1 \over 2d} (A(p_{n}(2e_{1})) - A(p_{n}(e_{1}))).$$ Now, arguing as in the proof of Lemma~\ref{monotonepreserved}, by the mean-value theorem, 
\begin{multline*}
(p_{n}(e_{1})- p_{n}(2e_{1}))\left[1 - {2d + 1 \over 2d} \left({A(p_{n}(2e_{1})) - A(p_{n}(e_{1})) \over (p_{n}(e_{1})- p_{n}(2e_{1}))}\right)\right]\\
 = (p_{n}(e_{1})- p_{n}(2e_{1}))\left[1 - {2d+1 \over 2d} A'(t)\right],
 \end{multline*} for some $t \in [p_{n}(2e_{1}), p_{n}(e_{1})]$. Since $A'(t) \le 2/3$ whenever $0 \le t \le 1/3$ and $(2d+1)/2d \le 3/2$, the above expression is at least $(p_{n}(e_{1})- p_{n}(2e_{1})) [1 - (3/2) (2/3)] = 0$.

\medskip
\emph{Case 2. $\x = e_1 + e_2$.} 

We perform a similar analysis as above. We have that for $\ell\sim e_{1}+e_{2}$, by symmetry,
\begin{equation*}
p_{n}(\ell)\begin{cases}
\leq p_n(2e_{1})&\text{if $\ell=2e_1+e_{2}, e_{1}+2e_{2}$},\\
=p_{n}(e_{1})&\text{if $\ell=e_{1}, e_{2}$},\\
\leq p_{n}(e_{1}+e_{2})&\text{otherwise}.
\end{cases}
\end{equation*}
This implies
\begin{equation*}
p_{n+1}(e_1 + e_2) \le p_{n}(e_{1}+e_{2}) + {2 \over 2d} (A(p_n(e_1)) - A(p_{n}(e_{1}+e_{2}))) + {2 \over 2d} (A(p_n(2e_1)) - A(p_{n}(e_{1}+e_{2}))).
\end{equation*}
Subtracting this from the inequality for $p_{n+1}(e_1)$ from the first case, we obtain the bound
$$p_{n+1}(e_1) - p_{n+1}(e_1 + e_2) \ge p_{n}(e_{1}) -  p_{n}(e_{1}+e_{2}) + {2d+2 \over 2d} (A( p_{n}(e_{1}+e_{2})) - A(p_{n}(e_{1}))).$$
Since again, $A'(t) \le 2/3$ whenever $0\leq t \le 1/3$, and $(2d+2)/2d \le 3/2$ because $d \ge 2$, this implies the right side is at least 0.

\medskip 

Finally, we must also show that $p_{n+1}(\x) \le p_{n+1}(0)$ if $|\x|_1 \ge 2$. In this case, because $p_n$ is almost-volcanic, $p_n(\x) \le p_n(0)$ and all the neighbours $\y \sim \x$ have $p_n(\y) \le p_n(e_1)$, so since $A$ is increasing, 
\begin{align*}
    p_{n+1}(\x) &= p_n(\x) + {1 \over 2d} \sum_{\ell\sim k} (A(p_n(\y)) - A(p_n(\x)))
    \\&\le p_n(\x) + {1 \over 2d} \sum_{\ell\sim k} (A(p_n(e_1)) - A(p_n(\x))).
\end{align*}
If we  remove all $A(p_n(\x))$ terms out of the sum, and using that $t\mapsto t-A(t)$ is nondecreasing for $0\leq t\leq 1/3$, we get
\begin{align*} p_{n+1}(\x) &\le p_n(\x) - A(p_n(\x)) + {1 \over 2d} \sum_{\ell\sim k} A(p_n(e_1))
    \\&\le p_n(0) - A(p_n(0)) + {1 \over 2d} \sum_{\ell\sim k} A(p_n(e_1))
    \\&=p_{n+1}(0).
\end{align*}
This proves the second condition, so $p_{n+1}$ is indeed almost-volcanic.\end{proof}

Since $d \ge 2$,  
\begin{equation*}
p_1(\x)=\begin{cases}1/2d&\text{when $\x \sim 0$},\\
0&\text{otherwise,}
\end{cases}
\end{equation*}
which is clearly almost-volcanic, so by the prior result, $p_n$ is also almost-volcanic for any $n \ge 1$. 
In the rest of this section we suppose that $p_n$ is almost-volcanic for all $n\geq 1$. We can use this with the bound in Theorem \ref{prop.pfacts}(i) to get a bound on the total variation of $p_n$, for all $n\geq 1$.
\begin{proof}[Proof of Theorem \ref{prop.pfacts}(iii)]
In order to estimate $[p_{n}]_{\TV(D)} := \tfrac{1}{2}\sum_{\{u,v\in D: u\sim v\}} |p_{n}(v)-p_{n}(u)|$ on the set $D := B(rn^\beta) \cap \mathbb Z^d$, we observe that since both the set $D$ and the almost-volcanic function $p_{n}$ are invariant under reflection and coordinate transposition, we may reduce this expression to studying 
\begin{equation*}
[ p_n ]_{\TV(D)} = d \sum_{\x \in \mathcal{D}}\ |p_n(\x + e_1) - p_n(\x)|, 
\end{equation*} 
where 
\begin{equation*}
\mathcal{D}:=\left\{k\in \mathbb Z^d: \x, \x+e_1 \in D\right\}.
\end{equation*}

Let 
\begin{equation*}
\overline{p}_{n}(\x) = \begin{cases} p_n(e_{1})&\text{if $\x=0$},\\
p_{n}(\x)&\text{otherwise.}
\end{cases}
\end{equation*} 
This makes $\overline{p}_{n}$ a volcanic function, and this implies that 
\begin{equation*}
\overline{p}_{n}(\x + e_1) - \overline{p}_{n}(\x)=\begin{cases}\geq 0&\text{if $\x_{1}<0$},\\
\leq 0&\text{if $\x_{1}\geq 0$.} 
\end{cases}
\end{equation*}
For each $k\in \mathcal{D}$, we rewrite $k=(k_{1}, k')$ where $k'\in \Z^{d-1}$, and for such $k'$, we denote that $k'\in \mathcal{D}'$. We now estimate the total variation of $\overline{p}_{n}$, which by the prior observation and symmetry is given by,
\begin{align*}
d\sum_{\x \in \mathcal{D}}& |\overline{p}_{n}(\x + e_1) - \overline{p}_{n}(\x)| \\
&= d \sum_{k'\in \mathcal{D}'} \Bigg[\sum_{k_{1}<0: (k_{1}, k')\in \mathcal{D}} \overline{p}_{n}(\x + e_1) - \overline{p}_{n}(\x)+\sum_{k_{1}\geq 0: (k_{1}, k')\in \mathcal{D}} \overline{p}_{n}(\x) - \overline{p}_{n}(\x + e_1)\Bigg]\\
&=d \sum_{k'\in \mathcal{D}'} [\overline{p}_{n}(0, k')-\overline{p}_{n}(-B+1, k')+\overline{p}_{n}(0, k')-\overline{p}_{n}(B+1, k')]\\
&\leq 2d  \sum_{k'\in \mathcal{D}'} \overline{p}_{n}(0, k'),
\end{align*}
where the point $B$ is such that $(B, k')\in \mathcal{D}$ and $(B+1, k')\notin \mathcal{D}$. 

Since $\mathcal{D}$ (and hence $\mathcal{D}'$) has radius at most $rn^\beta$, in order for $k'\in \mathcal{D}'$, we must have $|\x_i| \le rn^\beta$ for $i = 2, \ldots, d$. Therefore the sum has at most $(2rn^\beta)^{d-1}$ terms. By Theorem \ref{prop.pfacts}(i), each term is at most $c/n^{d\beta}$, and thus  
\begin{equation*}
[ \overline{p}_n ]_{\TV(\mathcal{D})}\leq cr^{d-1}n^{-\beta}.
\end{equation*}

We now compare the total variation of $p_n$ to the total variation of $\overline{p}_{n}$. Since only the origin of $p_{n}$ differs from the origin of $\overline{p}_{n}$, we have 
\begin{equation*}
[ p_n ]_{\TV(D)}\leq [ \overline{p}_n ]_{\TV(D)}+2d|p_n(e_1) - p_n(0)|.
\end{equation*}
Again by Theorem \ref{prop.pfacts}(i), the second term is bounded above by $cn^{-d\beta}$, and hence
\begin{equation*}
[ p_n ]_{\TV(D)}\leq cr^{d-1}n^{-\beta}+cn^{-d\beta}. 
\end{equation*}

If $r < n^{-\beta}$, then $D=\left\{0\right\}$, and by definition $[p_{n}]_{\TV(D)}=0$. Otherwise $cn^{-d\beta} = c(n^{-\beta})^{d-1}n^{-\beta} \le cr^{d-1}n^{-\beta}$ and the total variation of $p_{n}$ is at most $cr^{d-1}n^{-\beta}$, which implies the claimed bound. 

\end{proof}

\section{The proof of Theorem \ref{t.main} Using the Framework of Finite Difference Schemes}\label{s.scheme}

In this section, we complete the proof of our main result, Theorem \ref{t.main}, by situating the central limit theorem as a consequence of the convergence of a finite difference scheme. As priorly mentioned, to our knowledge, there is no general theory of convergence for finite difference schemes approximating the porous medium equation with initial data which do not belong to $L^\infty(\R^{d})$. Theorem \ref{prop.pfacts} is used crucially throughout our argument. 

\subsection{Notation}
For an interval $I \subset [0,\infty)$, we let
\begin{multline*} 
	C(I;L^{1}(\R^{d})) = \{u: \R^{d} \times I \to \R \mid u(\cdot,t) \in L^{1}(\R^{d}) \text{ for all }t \in I, \text{ and} \\ 
	\lim_{t \to t_{0}}\norm{u(\cdot,t)-u(\cdot,t_{0})}_{L^{1}(\R^d)} = 0 \text{ for all }t_{0} \in I \}. 
\end{multline*}
We rely on continuity properties of functions with respect to the $L^{1}(\R^{d})$ norm. We refer to a \emph{modulus of continuity} as a function $\nu:[0,\infty)\to [0,\infty)$ which is nondecreasing, continuous at $0$, and for which $\nu(0) = 0$. In the rest of this section, we assume that $\nu$ is a modulus of continuity. We say that a function $w \in L^{1}(\R^{d})$ has \emph{spatial modulus of continuity $\nu$} if 
\begin{equation*}
	\norm{w(\cdot + y) - w(\cdot)}_{L^{1}(\R^{d})} \leq \nu(|y|)\quad\text{for all $y\in \R^{d}$}.
\end{equation*}
For an interval $I\subset [0,\infty)$, we say that $v \in L^{1}(\R^{d} \times I)$ has {\em spatial modulus of continuity $\nu$} if 
\[
\norm{v(\cdot+y,t)-v(\cdot,t)}_{L^1(\R^d)} \le \nu(|y|)\quad\text{for all $t\in I$ and $y\in \R^{d}$}\, ,
\]
and that $v$ has {\em temporal} modulus of continuity $\nu$ if, 
\begin{equation*}
	 \norm{v(\cdot,t) - v(\cdot,s)}_{L^{1}(\R^{d})} \leq \nu(|t-s|)\quad\text{for all $s,t \in I$.}
\end{equation*}
We say a collection $(v_a)_{a \in A}\subset L^1(\R^d\times I)$ has spatial (resp.\ temporal) modulus of continuity $\nu$ if for every $a\in A$, $v_a$ has spatial (resp.\ temporal) modulus of continuity $\nu$. Note that if $v \in L^{1}(\R^{d} \times I)$ has a temporal modulus of continuity on $\R^{d} \times I$, then $v \in C(I;L^{1}(\R^{d}))$. Conversely, if $I$ is a bounded interval and $v \in C(I; L^{1}(\R^{d}))$ then $v$ has a temporal modulus of continuity on $\R^{d}\times I$. Lastly, for an arbitrary interval $I \subset [0,\infty)$, we say that a sequence of functions $(v_{N})_{N\in\N}$ defined on $\R^{d} \times I$ has \emph{approximate} temporal modulus of continuity $\nu$ on $\R^{d}\times I$ if for all $N \in \N$ and all $s,t\in I$, 
\begin{equation*}
	\norm{v_{N}(\cdot,t) - v_{N}(\cdot,s)}_{L^{1}(\R^{d})} \leq \nu(|s-t|+\tfrac{1}{N}).
\end{equation*}
\begin{remark} \label{rmk.WLOGmod}
	We remark that a modulus of continuity (be it spatial, temporal, or approximate temporal) is not unique; a modulus of continuity can always be replaced by any other modulus of continuity which dominates it. Hence, we will often assume without loss of generality that any two {\em a priori} distinct moduli of continuity are in fact the same.\end{remark}
	
	\subsection{Entropy and Distributional Solutions}\label{ss.sol}
We define two notions of solution for the Cauchy problem 
\begin{equation} \label{e.pme}
	\begin{cases}
		u_{t} - \frac{1}{2d}\Delta(u^{m+1}) = 0 &\text{in } Q_{T}=\R^{d}\times (0,T], \\
		u(x,0) = u_{0}(x) &\text{in $\R^{d}$}.
	\end{cases}
\end{equation}
Our work builds upon that of Karlsen and Risebro \cite{karlsenrisebro01}. In that article, they work with a \emph{nonnegative entropy solution}, which is defined as follows. 
\begin{definition}
	For $T > 0$, we say that a measurable function $u: Q_{T}\to \R$ is a \emph{nonnegative entropy solution} to the equation $u_{t} - \frac{1}{2d}\Delta(u^{m+1}) = 0$ in $Q_{T}$ if the following holds:
	\begin{enumerate}
		\item $u \in  L^{1}(Q_{T}) \cap L^{\infty}(Q_{T}) \cap C((0,T);L^{1}(\R^{d}))$.
		\item $u \geq 0$ a.e.\@ in $Q_{T}$.
		\item For all $c \in \R$ and all non-negative $\vp \in C^{\infty}_{c}(\R^{d} \times [0,T])$,
		\begin{equation} \label{e.etrpycond}
			\int_{\R^{d}}\int_{0}^{T}\left(|u-c|\vp_{t} +\left|u^{m+1}-c^{m+1}\right|\frac{1}{2d}\Delta \vp\right)\,dt\,dx \geq 0.
		\end{equation}
		\item $u^{m+1} \in L^{2}((0,T) ;H^{1}(\R^{d})),$ where $H^{1}(\R^{d})$ is the Sobolev space $W^{1,2}(\R^{d})$.
	\end{enumerate}

\end{definition}

\begin{definition}
	Let $u_{0} \in L^{1}(\R^{d}) \cap L^{\infty}(\R^{d})$. For $T > 0$, we say $u: Q_{T}\to \R$ is a \emph{nonnegative entropy solution} to \eqref{e.pme} (the initial value problem) if it is a nonnegative entropy solution in the above sense and as $t \to 0^{+}$, $\norm{u(\cdot,t) - u_{0}}_{L^{1}(\R^{d})} \to 0.$
\end{definition}

\begin{remark}
	We remark that the definition in \cite{karlsenrisebro01} is for a more general PDE, and we tailor the definition presented here to reflect the special case we are interested in. Furthermore, we require that $u \geq 0$ a.e.\@ in order to ensure that $u \mapsto u^{m+1}$ is a nondecreasing function. In Remark  \ref{rmk.KRdiffs}, we further outline the details of how our setting is a specific case of the one considered in \cite{karlsenrisebro01}.
\end{remark}

\begin{remark}
	If $u\in  L^{1}(Q_{T}) \cap L^{\infty}(Q_{T}) \cap C((0,T);L^{1}(\R^{d}))$ is a nonnegative entropy solution to the initial value problem \eqref{e.pme}, we can take the convention of defining $u(x,0) := u_{0}(x)$. Hence, the convergence $\norm{u(\cdot,t)-u_{0}}_{L^{1}(\R^{d})} \to 0$ allows us to conclude that in fact $u \in C([0,T); L^{1}(\R^{d}))$ instead of merely $u \in C((0,T);L^{1}(\R^{d}))$.
\end{remark}

There is another kind of solution to \eqref{e.pme} which allows for distributional initial conditions, i.e. the case where $u_{0}$ is a Radon measure. 

\begin{definition} \label{defn.distsol}
	For $T > 0$, we say that a measurable function $u: Q_{T}\to \R$ is a \emph{distributional solution} of $u_{t} - \frac{1}{2d}\Delta(u^{m+1}) = 0$ in $Q_{T}$ if 
	\begin{enumerate}
		\item $u \in L^{1}_{\loc}(Q_{T})$, the space of locally integrable functions on $Q_{T}$.
		\item $u^{m+1} \in L^{1}_{\loc}(Q_{T})$.
		\item $u \in C((0,T);L^{1}_{\loc}(\R^{d}))$.
		\item For all $\vp \in C^{\infty}_{c}(Q_{T})$, we have
		\begin{equation} \label{e.distIBP}
			\int_{\R^{d}} \int_{0}^{T} \left(u \vp_{t} + (u^{m+1})\frac{1}{2d}\Delta\vp\right)\,dt\,dx = 0.
		\end{equation}
	\end{enumerate}
\end{definition}

We remark that in the above definition, condition (3) is immediate from the other conditions by the following theorem.

\begin{thm}[\cite{vazquez2007porous}, Theorem 13.16] \label{thm.distribcont}
	For $T > 0$, let $u: Q_{T}\to \R$ be a measurable function satisfying conditions (1), (2), and (4) of Definition \ref{defn.distsol}. Then, $u$ is locally H\"older continuous, in particular $u \in C(Q_{T})$ and $u \in C((0,T); L^{1}_{\loc}(\R^{d}))$. 
\end{thm}

As for the initial conditions of a distributional solution, we have the following definition.

\begin{definition}
	Let $u_{0}$ be a Radon measure on $\R^{d}$. We say that $u$ is a distributional solution of \eqref{e.pme} (the initial value problem) if $u$ is a distributional solution in the above sense, and if for all $\vp \in C_{c}(\R^{d})$, 
	\begin{equation*}
		\lim_{t \to 0^{+}} \int_{\R^{d}}u(x,t)\vp(x) \,dx = \int_{\R^{d}} \vp(x) u_{0}(dx).
	\end{equation*}
	In such a case we call $u_{0}$ the \emph{initial trace} of $u$. For $u_{0} \in L^{1}_{\loc}(\R^{d})$, we use the embedding of $u_{0}$ as a Radon measure, namely $u_{0}(dx) = u_{0}\,dx$, in the above definition. 
\end{definition}

We remark that if $u_{0} \in L^{1}_{\loc}(\R^{d})$, the above definition is equivalent to saying $u(\cdot,t) \xrightarrow{t\to 0^{+} }u_{0}$, in $L^{1}_{\loc}(\R^{d})$, and is thus consistent with the way we interpret initial conditions for nonnegative entropy solutions, albeit for convergence in $L^1_{\loc}(\R^{d})$ instead of $L^1(\R^{d})$.

The reason for introducing distributional solutions is that they allow for a theory of uniqueness, even when the initial condition $u_{0}$ is a Radon measure.

\begin{thm}[\cite{vazquez2007porous}, Theorem 13.8] \label{thm.distunique}
	For $T > 0$, let $u_{1}$ and $u_{2}$ be two nonnegative distributional solutions of $u_{t} - \frac{1}{2d}\Delta(u^{m+1}) = 0$ in $Q_{T}$. If the initial traces of both $u_{1}$ and $u_{2}$ coincide, then $u_{1} = u_{2}$ a.e.\@ in $Q_{T}$.
\end{thm}

\begin{remark}
	In \cite[Theorem 13.8]{vazquez2007porous}, the result is for uniqueness of continuous nonnegative distributional solutions. However, by Theorem \ref{thm.distribcont}, all distributional solutions have an a.e.\@ continuous representative. Hence we can state Theorem \ref{thm.distunique} in the form given.
\end{remark}

Lastly, we show that all nonnegative entropy solutions are in fact distributional solutions.

\begin{prop} \label{prop.entropdist}
	Let $u$ be a nonnegative entropy solution to $u_{t} - \frac{1}{2d}\Delta(u^{m+1}) = 0$ in $Q_{T}$. Then $u$ is a distributional solution to $u_{t} - \frac{1}{2d}\Delta(u^{m+1}) = 0$ in $Q_{T}$.
\end{prop}
\begin{proof}
	We remark that $u \in L^{\infty}(Q_{T})$ implies that $u, u^{m+1} \in L^{1}_{\loc}(Q_{T})$. Therefore all that remains to be shown is the integration by parts formula \eqref{e.distIBP}. Let $\vp \in C^{\infty}_{c}(Q_{T})$. First, we assume $\vp \geq 0$. By choosing $c = 0$ in \eqref{e.etrpycond}, we have
	\begin{equation*}
		\int_{\R^{d}}\int_{0}^{T}\left(u\vp_{t} +u^{m+1}\frac{1}{2d}\Delta \vp\right)\,dt\,dx \geq 0.
	\end{equation*}
	Next, by choosing $c = \norm{u}_{L^{\infty}(Q_{T})}$ in \eqref{e.etrpycond}, we have
	\begin{equation*}
		\int_{\R^{d}}\int_{0}^{T}\left((\norm{u}_{L^{\infty}(Q_{T})}-u)\vp_{t} +\left(\norm{u}_{L^{\infty}(Q_{T})}^{m+1}-u^{m+1}\right)\frac{1}{2d}\Delta \vp\right)\,dt\,dx \geq 0,
	\end{equation*}
which implies
	\begin{align*}
		\int_{\R^{d}}\int_{0}^{T}\left(u\vp_{t} +u^{m+1}\frac{1}{2d}\Delta \vp\right)\,dt\,dx &\leq  \int_{\R^{d}}\int_{0}^{T}\left(\norm{u}_{L^{\infty}(Q_{T})}\vp_{t} +\norm{u}_{L^{\infty}(Q_{T})}^{m+1}\frac{1}{2d}\Delta \vp\right)\,dt\,dx\\
		& = 0,
	\end{align*}
	where we used the fundamental theorem of calculus and the compact support of $\vp$ in the last step. Hence, we conclude that \eqref{e.distIBP} holds for for all $\vp \in C^{\infty}_{c}(Q_{T})$ with $\vp \geq 0$. To account for general $ \vp \in C^{\infty}_{c}(Q_{T})$, we decompose $\vp$ into $\vp = \vp^{+} - \vp^{-}$ for $\vp^{+}, \vp^{-} \geq 0$ and $\vp^{+}\vp^{-} \equiv 0$. We note that $\vp^{+}$ and $\vp^{-}$ may not be smooth at their $0$-level sets. However, by approximating each of $\vp^{+}$ and $\vp^{-}$ by smooth functions, we can still conclude \eqref{e.distIBP} holds for both $\vp$ replaced by $\vp^{+}$ and for $\vp$ replaced by $\vp^{-}$. Finally, recalling $\vp = \vp^{+} - \vp^{-}$, we obtain the desired result for $\vp$. 
\end{proof}

By combining Theorems \ref{thm.distunique} and Proposition \ref{prop.entropdist} we obtain the following result. 

\begin{cor}\label{cor:uniqueness}
	Let $T > 0$ and $u_{0} \in L^{1}(\R^{d}) \cap L^{\infty}(\R^{d})$. If $u_{1}$ and $u_{2}$ are two nonnegative entropy solutions to \eqref{e.pme}, then $u_{1} \overset{a.e.}{=} u_{2}$.
\end{cor}

\begin{remark} \label{rmk.solshift}
	We remark that by considering the function $\tilde{u}(x,t) := u(x,t+a)$, one can adapt the two definitions of solutions presented above to domains of the form $\R^{d}\times (a,b)$.
\end{remark}

\subsection{Tools and technical results}
In this section, we outline some of the tools and technical lemmas that we will use throughout the rest of this section.

We begin with a compactness result for $L^{1}_{\loc}(\R^{d}\times[a,b])$.
\begin{lem} [\cite{splitting}, Lemma 3.3] \label{lem.L1loccmpct}
	Let $(v_{N})_{N \in \N}$ be a sequence of functions defined on $\R^{d} \times [a,b]$ which satisfies the following conditions:
	\begin{enumerate}
		\item There exists $C> 0$ such that for all $N \in \N$ and $t \in [a,b]$,
		\begin{equation*}
			\norm{v_{N}(\cdot,t)}_{L^{1}(\R^{d})} \leq C \text{ and } \norm{v_{N}(\cdot,t)}_{L^{\infty}(\R^{d})} \leq C.
		\end{equation*}
		\item There exists a spatial modulus of continuity $\nu$ such that for all $N\in \N$ and $t \in [a,b]$,
		\begin{equation*}
			\norm{v_{N}(\cdot + y,t) - v_{N}(\cdot,t)}_{L^{1}(\R^{d})} \leq \nu(|y|).
		\end{equation*}
		\item There exists an approximate temporal modulus of continuity $\om$ such that for all $N \in \N$ and $s, t \in [a,b]$,
		\begin{equation*}
			\norm{v_{N}(\cdot,t) - v_{N}(\cdot,s)}_{L^{1}(\R^{d})} \leq \om(|t-s| + \tfrac{1}{N}).
		\end{equation*}
	\end{enumerate}
	Then $(v_{N})_{N\in \N}$ is compact in the strong topology of $L^{1}_{\loc}(\R^{d}\times[a,b])$. Moreover, if $v$ is an $L^{1}_{\loc}(\R^{d}\times[a,b])$-limit point of $(v_{N})_{N \in \N}$, then $v \in L^{1}(\R^{d}\times[a,b])\cap L^{\infty}(\R^{d}\times[a,b]) \cap C([a,b];L^{1}(\R^{d}))$, and $v$ has a spatial modulus of continuity $\nu$ on $\R^{d}\times[a,b]$ and $v$ has a temporal modulus of continuity $\om$ on $\R^{d}\times[a,b]$. 
\end{lem}
\label{rmk.endpoints}
\begin{remark}
	We note that the reference \cite{splitting} only concludes that $v \in C((a,b);L^{1}(\R^{d}))$ and not that $v \in C([a,b];L^{1}(\R^{d}))$. Because $v$ admits $\om$ as a temporal modulus of continuity, $t \mapsto v(\cdot,t)$ is uniformly continuous as a function from $[a,b]$ to $L^{1}(\R^{d})$. Hence, because $L^{1}(\R^{d})$ is complete, we can extend this mapping uniquely to define $v(\cdot,a)$ and $v(\cdot,b)$ as functions in $L^{1}(\R^{d})$, and then $v \in C([a,b];L^{1}(\R^{d}))$. The temporal continuity also implies that $\nu$ is a spatial modulus of continuity on all of $\R^{d}\times [a,b]$, and not merely $\R^{d} \times (a,b)$.
\end{remark}

Next, we state a technical lemma on the convergence of sequences of functions which admit a temporal modulus of continuity.

\begin{lem} \label{lem.tslice}
	Let $(w_{N})_{N \in \N}$ and $(z_{N})_{N \in \N}$ be two sequences in $L^{1}(\R^{d}\times[a,b])$ which have an approximate temporal modulus of continuity $\om$. Then for all compact $D \subset \R^{d}$, 
	\begin{equation*}
		\norm{w_{N} - z_{N}}_{L^{1}(D\times (a,b))} \xrightarrow{} 0 \text{ if and only if, for all } t \in [a,b], \norm{w_{N}(\cdot,t) - z_{N}(\cdot,t)}_{L^{1}(D)} \xrightarrow{} 0.
	\end{equation*}
\end{lem}
\begin{proof}
	First, assume that 
	\begin{equation} \label{e.sliceloc}
		\norm{w_{N} - z_{N}}_{L^{1}(D\times (a,b))} \xrightarrow{N \to \infty} 0.
	\end{equation}
	Fix $t \in [a,b)$ and $\ve \in (0,b-t)$. We calculate
	\begin{align*}
		&\norm{w_{N}(\cdot,t) - z_{N}(\cdot,t)}_{L^{1}(D)} \\
		&= \frac{1}{\ve}\int_{t}^{t+\ve}\norm{w_{N}(\cdot,t) - z_{N}(\cdot,t)}_{L^{1}(D)}\,ds \\
		&\leq  \frac{1}{\ve}\int_{t}^{t+\ve}\norm{w_{N}(\cdot,t) - w_{N}(\cdot,s)}_{L^{1}(D)}\,ds  +
		\frac{1}{\ve}\norm{w_{N} - z_{N}}_{L^{1}(D\times(t,t+\ve))} \\
		&\quad\quad\quad\quad\quad\quad\quad+\frac{1}{\ve}\int_{t}^{t+\ve}\norm{z_{N}(\cdot,s) - z_{N}(\cdot,t)}_{L^{1}(D)}\,ds.
	\end{align*}
	We bound the final term in the upper bound by observing that
	\begin{equation*}
		\frac{1}{\ve}\int_{t}^{t+\ve}\norm{z_{N}(\cdot,s) - z_{N}(\cdot,t)}_{L^{1}(D)}\,ds \leq \frac{1}{\ve}\int_{t}^{t+\ve}\om(s-t+\tfrac{1}{N})\,ds \leq \om(\ve+\tfrac{1}{N}).
	\end{equation*}
	A similar calculation holds for the first term in the upper bound, so we have 
	\begin{equation*}
		\norm{w_{N}(\cdot,t) - z_{N}(\cdot,t)}_{L^{1}(D)}  \leq 2 \om(\ve+\tfrac{1}{N}) + \frac{1}{\ve}\norm{w_{N} - z_{N}}_{L^{1}(D\times(t,t+\ve))}.
	\end{equation*}
	Taking $N \to \infty$ in the above, we apply \eqref{e.sliceloc} to obtain that for every $t\in [a,b)$,
	\begin{equation*}
		\limsup_{N \to \infty}\norm{w_{N}(\cdot,t) - z_{N}(\cdot,t)}_{L^{1}(D)}  \leq 2 \om(\ve).
	\end{equation*}
	Lastly, sending $\ve \to 0$ allows us to conclude that $\norm{w_{N}(\cdot,t) - z_{N}(\cdot,t)}_{L^{1}(D)} \to 0$. In the case $t = b$, we can repeat a similar argument by instead integrating from $t - \ve$ to $t$.
	
	For the other implication, 
	fix $n \in \N$ and a partition $\{t_{i}\}_{i = 0}^{n}$ of $[a,b]$ with $t_{0} = a$ and $t_{i}-t_{i-1} = \frac{b-a}{n}$. We then have
	\begin{align*}
		&\norm{w_{N} - z_{N}}_{L^{1}(D\times (t_{i-1},t_{i}))} \\
		&\leq \int_{t_{i-1}}^{t_{i}}\norm{w_{N}(\cdot,s) - w_{N}(\cdot,t_{i})}_{L^{1}(D)}\,ds +  \int_{t_{i-1}}^{t_{i}}\norm{w_{N}(\cdot,t_{i})-z_{N}(\cdot,t_{i})}_{L^{1}(D)}\,ds \\ 
		& \quad\quad\quad\quad\quad\quad\quad+\int_{t_{i-1}}^{t_{i}}\norm{z_{N}(\cdot,t_{i}) - z_{N}(\cdot,s)}_{L^{1}(D)}\,ds \\
		&\leq 2\left(\tfrac{b-a}{n}\right)\cdot\om\left(\tfrac{b-a}{n} + \tfrac{1}{N}\right) + \tfrac{b-a}{n}\norm{w_{N}(\cdot,t_{i})-z_{N}(\cdot,t_{i})}_{L^{1}(D)} .
	\end{align*}
	Therefore, 
	\begin{align*}
		\norm{w_{N} - z_{N}}_{L^{1}(D\times (a,b))} &= \sum_{i = 1}^{n} \norm{w_{N} - z_{N}}_{L^{1}(D\times (t_{i-1},t_{i}))} \\
		&\leq 2\left(b-a\right)\cdot\om\left(\tfrac{b-a}{n}+\tfrac{1}{N}\right) + \tfrac{b-a}{n}\sum_{i=1}^{n}\norm{w_{N}(\cdot,t_{i})-z_{N}(\cdot,t_{i})}_{L^{1}(D)} 
	\end{align*}
	If for all $t \in [a,b]$, $\norm{w_{N}(\cdot,t) - z_{N}(\cdot,t)}_{L^{1}(D)} \xrightarrow{N \to \infty} 0$, then it follows from the above that 
	\begin{equation*}
		\limsup_{N \to \infty}\norm{w_{N} - z_{N}}_{L^{1}(D\times (a,b))} \leq 2\left(b-a\right)\cdot\om\left(\tfrac{b-a}{n}\right),
	\end{equation*}
	 and so we conclude by taking $n \to \infty$.
\end{proof}

Our lemma yields the following corollary.

\begin{cor} \label{cor.aetslice}
	For an interval $I \subset \R$, let $w,z \in L^{1}(\R^{d} \times I) \cap C(I;L^{1}(\R^{d}))$. Then
	\begin{equation*}
		w \overset{a.e.}{=} z \text{ in }\R^{d} \times I \text{ if and only if } w(\cdot,t) \overset{a.e.}{=} z(\cdot,t) \text{ in } \R^{d} \text{ for all } t \in I.
	\end{equation*}
\end{cor}
\begin{proof}
	It suffices to prove the corollary assuming $I$ is compact; so write $I=[a,b]$ for $a,b \in \R$. 
	Fix a compact $D \subset \R^{d}$. Because $w, z \in C([a,b];L^{1}(\R^{d}))$, both $w$ and $z$ have temporal moduli of continuity on $\R^{d}\times[a,b]$. By Remark \ref{rmk.WLOGmod}, without loss of generality, we can assume these moduli to be the same. Thus, we apply Lemma \ref{lem.tslice} to the constant sequences $w_{N} \equiv w$ and $z_{N} \equiv z$ to obtain
	\begin{equation*}
		\norm{w - z}_{L^{1}(D \times [a,b])} = 0 \text{ if and only if } \norm{w(\cdot,t)-z(\cdot,t)}_{L^{1}(D)}=0 \text{ for all } t \in [a,b].
	\end{equation*}
	Because $D$ and $[a,b]$ are arbitrary, the result follows.
\end{proof}

Finally, we state a result which, under certain hypotheses, allows us to upgrade a spatial modulus of continuity to a temporal modulus of continuity.

\begin{lem}[Kru{\v{z}}kov's interpolation lemma \cite{Kruzhkov1969ResultsCT}] \label{lem.KIL}
	Let $z \in L^{\infty}(\R^{d}\times[0,T])$. Assume that $z$ has a spatial modulus of continuity $\nu$ on $\R^{d}\times[0,T]$. Given $C>1$, if  $t_1, t_2 \in [0,T]$ are such that for all $\vp \in C^{\infty}_{c}(\R^{d})$,
	\begin{equation} \label{e.KILhyp}
		\left| \int_{\R^{d}} (z(x,t_2)-z(x,t_{1})) \vp(x) \, dx\right| \leq C\cdot\left(\sum_{i = 1}^{d} \norm{\partial_{x_{i}}^{2}\vp}_{L^{\infty}(\R^{d})}\right)\cdot |t_{2}-t_{1}|,
	\end{equation}
	then for all $\ve > 0$, 
	\begin{equation*}
		\norm{z(\cdot,t_{2}) - z(\cdot,t_{1})}_{L^{1}(\R^{d})} \leq 2C \cdot \left( \frac{|t_{2}-t_{1}|}{\ve^2} + \nu(\ve)\right).
	\end{equation*}
\end{lem}
It is worth noting that hypothesis \eqref{e.KILhyp} is a simplified version of the more general result stated in \cite{Kruzhkov1969ResultsCT}.

\subsection{Diffuse Initial Conditions} 
We consider the case of diffuse initial conditions, meaning we consider the PDE
\begin{equation} \label{e.pmev}
	\begin{cases}
		v_{t} - \frac{1}{2d}\Delta(v^{m+1}) = 0 &\text{in }Q_{T} \\
		v(x,0) = v_{0}(x) &\text{for }x \in \R^{d},
	\end{cases}
\end{equation}
where $v_{0} \in L^{1}(\R^{d}) \cap L^{\infty}(\R^{d})$. In this case, the discrete approximations $v_{N}$ satisfy some useful properties, as outlined in the following lemma. To understand the lemma statement, recall the definition of $\mathcal{S}_N$ from \eqref{e.SN} with the choice of $A(u)=u^{m+1}$, given by 
\begin{equation*}\mathcal{S}_{N}v(k\xmesh)= v(k\xmesh)+\frac{\tmesh}{2d (\xmesh)^2}\sum_{\ell \sim k}\left[(v(\ell \xmesh))^{m+1}-(v(k\xmesh))^{m+1}\right], 
\end{equation*}
and $1/N=\tmesh=\xmesh^{1/\beta}$. We recall that for $v_{0}\in L^{1}(\R^{d})$, we say that $v_{N}$ is generated by $\mathcal{S}_{N}$ with initial condition $v_{0}$ has initial condition defined by \eqref{e.initL1}. 
\begin{lem} \label{lem.modulii}
	Let $v_{0} \in L^{1}(\R^{d}) \cap L^{\infty}(\R^{d})$ where $v_{0} \geq 0$, and $v_{0}$ has a spatial modulus of continuity $\nu$. Let $v_{N}$ be generated by $\mathcal{S}_{N}$ with initial condition $v_{0}$. Then, there exists $N_{0} = N_{0}(d,m,\norm{v_{0}}_{L^{\infty}(\R^{d})}) \in \N$ such that the following hold. 
	\begin{enumerate}[(i)]
		\item For all $N \geq N_{0}$, and $t \geq 0$, 
		\begin{equation*}
			\norm{v_{N}(\cdot,t)}_{L^{\infty}(\R^{d})} \leq \norm{v_{0}}_{L^{\infty}(\R^{d})}\quad\text{and}\quad\norm{v_{N}(\cdot,t)}_{L^{1}(\R^{d})} \leq \norm{v_{0}}_{L^{1}(\R^{d})}.
		\end{equation*}
			
		\item For all $N \in \N$, $v_{N}$ admits $\nu$ as a spatial modulus of continuity on $\R^{d}\times[0,T]$.
		\item The tail sequence $(v_{N})_{N \geq N_{0}}$ admits an approximate temporal modulus of continuity on $\R^{d}\times[0,T]$.\label{vN.apptemp}
		\end{enumerate}
\end{lem}
\begin{proof}
	To prove (i), fix $t>0$, let $n = \floor{t/\tmesh}$, $\La := \norm{v_{0}}_{L^{\infty}(\R^{d})}$, and define
	\begin{equation*}
		N_{0} := (4(m+1)\La^{m})^{\frac{1}{dm\beta}}.
	\end{equation*}
	For $N \geq N_{0}$, $N$ satisfies \eqref{e.cfl}, and so Lemma \ref{lem.SN}(i, ii) gives that
	\begin{equation*}
	0 \leq v_{N}(\cdot, t) = \mathcal{S}^{n}_{N}v_{0} \leq \bigg\|\frac{1}{\square_{N}}\int_{\square_{N}(\cdot)} v_{0}(x)\, dx\bigg\|_{L^{\infty}(\Z^{d}\xmesh)}\leq \La = \norm{v_{0}}_{L^{\infty}(\R^{d})}.
	\end{equation*}
	Furthermore, choosing $w_{N} := v_{N}(\cdot,0)$ and $w'_{N} \equiv 0$ in Lemma \ref{lem.SN}(iii) gives that
	\begin{equation*}
		\norm{v_{N}(\cdot,t)}_{L^{1}(\R^{d})} = \norm{\mathcal{S}^n_{N}(v_{N}(\cdot,0))}_{L^{1}(\R^{d})} \leq \norm{v_{N}(\cdot,0)}_{L^{1}(\R^{d})} = \norm{v_{0}}_{L^{1}(\R^{d})},
	\end{equation*}
	and so (i) is proved.

	To prove (ii), let $N \in \N$. We remark that directly by definition \eqref{e.SN}, $\mathcal{S}_{N}$ is translation invariant, meaning that for $y \in \R^{d}$, 
	\begin{equation*}
		\big[\mathcal{S}_{N}v_{N}(\cdot + y)\big](x) = \big[\mathcal{S}_{N}v_{N}\big](x+y).
	\end{equation*}
	We use the translation invariance of $\mathcal{S}_{N}$, along with Lemma \ref{lem.SN}(iii), to obtain that for all $t \geq 0$, $y \in \R^{d}$, 
	\begin{equation*}
		\norm{v_{N}(\cdot + y, t) - v_{N}(\cdot,t)}_{L^{1}(\R^{d})} \leq \norm{v_{0}(\cdot+ y) - v_{0}(\cdot)}_{L^{1}(\R^{d})} \leq \nu(|y|).
	\end{equation*}
	
	To show (iii), we will use Kru{\v{z}}kov's interpolation lemma (Lemma \ref{lem.KIL}). Let $\vp \in C_{c}^{\infty}(\R^{d})$, and fix $n_{1}, n_{2} \in \N_{0}$. We will show that there is a constant $C>1$ such that
	\begin{equation*}
		\left| \int_{\R^{d}} (v_{N}(x,n_{2}\tmesh)-v_{N}(x,n_{1}\tmesh)) \vp(x) \, dx\right| \leq C\cdot\left(\sum_{i = 1}^{d} \norm{\partial_{x_{i}}^{2}\vp}_{L^{\infty}(\R^{d})}\cdot\right) |n_{2}\tmesh-n_{1}\tmesh|.
	\end{equation*}
	First, we use the fact that $v_{N}(\cdot,t)$ is $\Z^{d}$-piecewise constant for all $t \geq 0$ to calculate
	\begin{align*}
		&\int_{\R^{d}} (v_{N}(x,n_{2}\tmesh)-v_{N}(x,n_{1}\tmesh)) \vp(x) \, dx \\
		&= \sum_{k \in \Z^{d}} \int_{\square_{N}(k\xmesh)}(v_{N}(x,n_{2}\tmesh)-v_{N}(x,n_{1}\tmesh)) \vp(x) \, dx \\
		&= \sum_{k \in \Z^{d}} (v_{N}(k\xmesh,n_{2}\tmesh)-v_{N}(k\xmesh,n_{1}\tmesh)) \int_{\square_{N}(k\xmesh)}\vp(x) \, dx.
	\end{align*}
	Define $\overline{\vp}_{N}: \xgrid \to \R$ by $\overline{\vp}_{N}(k\xmesh) :=  \int_{\square_{N}(k\xmesh)}\vp(x)\,dx$. Introducing the shorthand $V^{n}_{k} := v_{N}(k\xmesh,n\tmesh)$, we can rewrite the last display as
	\begin{equation*}
		\int_{\R^{d}} (v_{N}(x,n_{2}\tmesh)-v_{N}(x,n_{1}\tmesh)) \vp(x) \, dx = \sum_{k \in \Z^{d}} (V^{n_{2}}_{k}-V^{n_{1}}_{k})\overline{\vp}_{N}(k\xmesh).
	\end{equation*}
	Assume without loss of generality that  $n_{2} \geq n_{1}$. By a telescoping sum, we can write
	\begin{equation} \label{e.KIL1}
		\int_{\R^{d}} (v_{N}(x,n_{2}\tmesh)-v_{N}(x,n_{1}\tmesh)) \vp(x) \,dx = \sum_{n=n_{1}}^{n_{2}-1}\sum_{k \in \Z^{d}} (V^{n+1}_{k}-V^{n}_{k})\overline{\vp}_{N}(k\xmesh).
	\end{equation}
	For a fixed $n \in \N_{0}$, we use the scheme satisfied by $v_{N}$ to calculate
	\begin{align} \label{e.beforeFTC}
		&\sum_{k \in \Z^{d}} (V^{n+1}_{k}-V^{n}_{k})\overline{\vp}_{N}(k\xmesh) \notag\\
		&= \frac{\tmesh}{2d(\xmesh)^{2}} \sum_{k \in \Z^{d}} \sum_{i=1}^{d} \left[(V^{n}_{k+e_{i}})^{m+1}-2(V^{n}_{k})^{m+1}+(V^{n}_{k-e_{i}})^{m+1}\right]\overline{\vp}_{N}(k\xmesh) \notag\\
		&= \frac{\tmesh}{2d(\xmesh)^{2}}  \sum_{k \in \Z^{d}} \sum_{i=1}^{d} (V^{n}_{k})^{m+1}\left[\overline{\vp}_{N}((k+e_{i})\xmesh) - 2\overline{\vp}_{N}(k\xmesh) + \overline{\vp}_{N}((k-e_{i})\xmesh)\right],
	\end{align}
	where in the last step we used summation by parts and the fact that $\vp$ has compact support. For a fixed $i \in [d]$, we remark that 
	\begin{align*}
	&|\overline{\vp}_{N}((k+e_{i})\xmesh) - 2\overline{\vp}_{N}(k\xmesh) + \overline{\vp}_{N}((k-e_{i})\xmesh)| \\
		& \leq \int_{\square_{N}(k\xmesh)} \left|\left(\vp(x+e_{i}\xmesh)-\vp(x)\right) - \left(\vp(x) - \vp(x-e_{i}\xmesh)\right)\right|\,dx\\
		 &= \int_{\square_{N}(k\xmesh)}\bigg| \int_{0}^{\xmesh}\left[\partial_{x_{i}}\vp(x + s_{1}e_{i}) - \partial_{x_{i}}\vp(x + (s_{1}-\xmesh)e_{i})\right]\,ds_{1}\bigg| dx \\
		&= \int_{\square_{N}(k\xmesh)}\bigg|\int_{0}^{\xmesh}\int_{0}^{\xmesh}\partial_{x_{i}}^{2}\vp(x + (s_{1}+s_{2}-\xmesh)e_{i})\,ds_{1}\,ds_{2}\bigg|\, dx.
	\end{align*}
	Using now that $|\square_{N}(k\xmesh)|=(\xmesh)^{d}$, we conclude that
		\begin{equation*}
		|\overline{\vp}_{N}((k+e_{i})\xmesh) - 2\overline{\vp}_{N}(k\xmesh) + \overline{\vp}_{N}((k-e_{i})\xmesh)| \leq \norm{\partial_{x_{i}}^{2}\vp}_{L^{\infty}(\R^{d})}(\xmesh)^{d+2}.
	\end{equation*}
	Returning to \eqref{e.beforeFTC}, we deduce that 
	\begin{align*}
		\left|\sum_{k \in \Z^{d}} (V^{n+1}_{k}-V^{n}_{k})\overline{\vp}_{N}(k\xmesh)\right| &\leq \frac{\tmesh}{2d(\xmesh)^{2}} \sum_{k \in \Z^{d}}  \sum_{i=1}^{d}(V^{n}_{k})^{m+1}\norm{\partial_{x_{i}}^{2}\vp}_{L^{\infty}(\R^{d})}(\xmesh)^{d+2} \\
		&= \frac{\tmesh}{2d} \cdot\left((\xmesh)^{d}\sum_{k \in \Z^{d}} (V^{n}_{k})^{m+1}\right) \cdot \left( \sum_{i=1}^{d}\norm{\partial_{x_{i}}^{2}\vp}_{L^{\infty}(\R^{d})}\right).
	\end{align*}
	Because $V^{n}_{k}:= v_{N}(k\xmesh,n\tmesh)$ and $v_{N}(\cdot,n\tmesh)$ is $\Z^{d}$-piecewise constant, 
	\begin{align*}
		(\xmesh)^{d}\sum_{k \in \Z^{d}} (V^{n}_{k})^{m+1} &= \int_{\R^{d}}(v_{N}(x,n\tmesh))^{m}v_{N}(x,n\tmesh) \, dx \\
		&\leq \norm{v_{N}(\cdot,n\tmesh)}_{L^{\infty}(\R^{d})}^{m} \norm{v_{N}(\cdot,n\tmesh)}_{L^{1}(\R^{d})}.
	\end{align*}
	By property (i) since $N\geq N_{0}$, we combine the previous two displays to obtain a constant $C_{v_{0}}$, where $C_{v_{0}} = C_{v_{0}}(d,m,\norm{v_{0}}_{L^{\infty}(\R^{d})}, \norm{v_{0}}_{L^{1}(\R^{d})}) > 0$, such that
	\begin{equation*}
		\left|\sum_{k \in \Z^{d}} (V^{n+1}_{k}-V^{n}_{k})\overline{\vp}_{N}(k\xmesh)\right|  \leq C_{v_{0}} \tmesh \left( \sum_{i=1}^{d}\norm{\partial_{x_{i}}^{2}\vp}_{L^{\infty}(\R^{d})}\right).
	\end{equation*}
	Thus, returning to \eqref{e.KIL1}, we have
	\begin{align*}
		\left|\int_{\R^{d}} (v_{N}(x,n_{2}\tmesh)-v_{N}(x,n_{1}\tmesh)) \vp(x) \,dx \right| &= \sum_{n=n_{1}}^{n_{2}-1} \left|\sum_{k \in \Z^{d}} (V^{n+1}_{k}-V^{n}_{k})\overline{\vp}_{N}(k\xmesh)\right| \\
		&\leq C_{v_{0}}|n_{2}\tmesh-n_{1}\tmesh| \left( \sum_{i=1}^{d}\norm{\partial_{x_{i}}^{2}\vp}_{L^{\infty}(\R^{d})}\right).
	\end{align*}
	Next, an application of Kru{\v{z}}kov's interpolation lemma (Lemma \ref{lem.KIL}), with the choice $\ve = |n_{2}\tmesh-n_{1}\tmesh|^{1/3}$, yields, for all $n_{1},n_{2} \in \N$,
	\begin{align*}
		\int_{\R^{d}} |v_{N}(x,n_{2}\tmesh)-v_{N}(x,n_{1}\tmesh)| \vp(x) \,dx&\leq 2C_{v_{0}}\big(|n_{2}\tmesh-n_{1}\tmesh|^{1/3} + \nu\big(|n_{2}\tmesh-n_{1}\tmesh|^{1/3}\big) \big) \\
		&:= \om(|n_{2}\tmesh-n_{1}\tmesh|),
	\end{align*}
	where we remark that $\om = \om(d,m,\norm{v_{0}}_{L^{\infty}(\R^{d})}, \norm{v_{0}}_{L^{1}(\R^{d})},\nu)$. 
	
	To conclude, we consider times which are not on the grid $\tgrid$. Let $t_{1},t_{2} \in [0,T]$. We have, by setting $n_{1} = \floor{t_{1}/\tmesh}$ and $n_{2} = \floor{t_{2}/\tmesh}$, that
	\begin{equation*}
		|n_{2}\tmesh - n_{1}\tmesh| \leq |t_{2}-t_{1}| + 2\tmesh.
	\end{equation*}
	Because $v_{N}$ is piecewise constant in time, we conclude that,
	\begin{align*}
		\int_{\R^{d}} |v_{N}(x,t_{2})-v_{N}(x,t_{1})| \,dx &= \int_{\R^{d}} |v_{N}(x,n_{2}\tmesh)-v_{N}(x,n_{1}\tmesh)| \,dx\\
		&\leq \om(|n_{2}\tmesh-n_{1}\tmesh|) \\
		&\leq \om(|t_{2}-t_{1}| + 2\tmesh)\\
		&\leq \om(2(|t_{2}-t_{1}| +\tfrac{1}{N})),
	\end{align*}
	where we used the fact that $\tmesh = N^{-1}$ in the final step.
\end{proof}
We now state a convergence result due to Karlsen and Risebro \cite{karlsenrisebro01}, which also yields an existence result for entropy solutions of the PME with ``diffuse'' initial conditions.
\begin{thm} \cite[Theorem 4.2]{karlsenrisebro01} \label{t.kr}
	Let $\La > 0$, and $A: [0,\La] \to \R$ be a function such that
	\begin{equation}\label{e.ahyp}
		A \text{ is non-decreasing, } A(0) = 0, \text{ and }A\text{ is Lipschitz continuous}.
	\end{equation}
	Let $v_{0} \in L^{1}(\R^{d}) \cap \BD{\La}(\R^{d})$, and let $v_{N}$ be generated by $\mathcal{S}_{N}$ with initial condition $v_{0}$. If there exists $\ve \in (0,1)$ such that for all $N$ sufficiently large, 
	\begin{equation} \label{e.cfleps}
		2N^{-\frac{dm}{dm+2}}\max_{r \in [0,\La]}A'(r) < 1 - \ve,
	\end{equation}
	then there is a function $v: \R^{d}\times [0, \infty)\rightarrow \R$ (independent of $\ve$) such that for any $T>0$, $v_{N} \to v$ in $L^{1}_{\text{loc}}(Q_{T})$, and $v$ is the unique entropy solution to 
	\begin{equation*}
	\begin{cases}
			v_{t}-\frac{1}{2d}\Delta(A(v))=0&\text{in $\R^{d}\times (0, \infty)$},\\
			v(0,x)=v_{0}(x)&\text{in $\R^{d}$}.
		\end{cases}
	\end{equation*}

\end{thm}

\begin{remark} \label{rmk.KRdiffs}
	Here we make a few remarks about how we adapted the statement of \cite[Theorem 4.2]{karlsenrisebro01} to our situation. First, we note that the authors consider a more general PDE of the form $u_{t} + \divergence f(u) = \frac{1}{2d}\Delta A(u)$. We state their result with the choice $f \equiv 0$ for simplicity. 
	
	Second, they consider a slightly different set of assumptions on the function $A$. In \cite{karlsenrisebro01}, $A$ is defined in $\R$ rather than $[0,\La]$, $A$ is locally Lipschitz rather than Lipschitz, and in the condition \eqref{e.cfleps}, $\max_{r \in [0,\La]}A'(r)$ is replaced by $\sup_{r \in \R}A'(r)$. Because we seek to use $A(r) = r^{m+1}$, we restrict the domain of $A$, $\mathcal{D}(A)$, so that $\mathcal{D}(A) := [0,\La]$, to ensure that $A$ is non-decreasing and that $\max_{r \in \mathcal{D}(A)}A'(r) < \infty$. This domain restriction is made without loss of generality: the proof \cite[Theorem 4.2]{karlsenrisebro01} only relies on an analysis of $A(v_{N})$ and we know that $v_{N}(\cdot, t)\in \BD{\La}(\R^{d})$ whenever $v_{0} \in \BD{\La}(\R^{d})\cap L^1(\R^{d})$ (for $N$ sufficiently large, by Lemma \ref{lem.modulii}(ii)).
	
	Lastly, as discussed in the remarks preceding the statement of \cite[Theorem 4.2]{karlsenrisebro01}, this theorem is an existence result for nonnegative entropy solutions with initial condition $v_{0} \in L^{1}(\R^{d}) \cap L^{\infty}(\R^{d})$, so long as the function $A$ satisfies \eqref{e.ahyp}, which are conditions independent of $\Lambda$. 
\end{remark}

The following corollary rephrases the above convergence result for our specific case, when $A(u)=u^{m+1}$. 

\begin{cor} \label{c.diffuse}
	Let $v_{0} \in L^{1}(\R^{d}) \cap L^{\infty}(\R^{d})$ with $v_{0} \geq 0$. Let $v$ be the unique nonnegative entropy solution to \eqref{e.pmev} and let $v_{N}$ be generated by $\mathcal{S}_{N}$ with initial condition $v_{0}$. Then
	\begin{equation*}
		v_{N} \xrightarrow{L^{1}_{\text{loc}}(Q_{T})} v \text{ and }v_{N}(\cdot,t) \xrightarrow{L^{1}_{\text{loc}}(\R^{d})} v(\cdot,t)\text{ for all }t \in [0,T].
	\end{equation*}
\end{cor}

\begin{proof}
	Taking $\La = \norm{v_{0}}_{\infty}$, there is an a.e.\@ representative for $v_{0}$ which satisfies $v_{0}\in \BD{\La}(\R^{d})$. Let $A(u) := u^{m+1}$. Because $\max_{u\in[0,\La]}A'(u) = A'(\La) = (m+1)\La^{m}$, we see that the condition \eqref{e.cfleps} is satisfied with $\ve = \frac{1}{2}$, so long as
	\begin{equation}
		N^{\frac{dm}{dm+2}} \geq 4(m+1)\La^{m},
	\end{equation}
	which is also a sufficient condition to guarantee monotonicity of $\mathcal{S}_{N}$ (as seen in Lemma \ref{lem.SN}). Therefore, the hypotheses of Theorem \ref{t.kr} are satisfied and thus $v_{N} \to v$ in $L^{1}_{\text{loc}}(Q_{T})$. 
	
	Next, we show the final claim of $L^{1}_{\loc}(\R^{d})$ convergence on time slices. By Lemma \ref{lem.modulii}(iii), there is some $N_{0} = N_{0}(m,d,\norm{v_{0}}_{L^{\infty}(\R^{d})}) \in \N$ so that $\{v_{N}\}_{N \geq N_{0}}$ admits an approximate temporal modulus of continuity $\om$ on $\R^{d}\times[0,T]$. Because $v$ is an $L^{1}_{\text{loc}}(Q_{T})$-limit point of $\{v_{N}\}_{N \geq N_{0}}$, Lemma \ref{lem.L1loccmpct} gives that $\om$ is also a temporal modulus of continuity for $v$. Thus, we apply Lemma \ref{lem.tslice} to the sequence $w_{N} := v_{N}$ and the constant sequence $z_{N} := v$ to conclude that we also have the convergence $v_{N}(\cdot,t) \xrightarrow{L^{1}_{\text{loc}}(\R^{d})} v(\cdot,t)$, for all $t \in [0,T]$.
\end{proof}

\subsection{\bf Convergence to the Barenblatt} \label{sec.convtoBaren}
Recall that $1/N=\tmesh=\xmesh^{1/\beta}$. We begin this section by proving some important properties about $u_{N}: \R^{d} \times [0,\infty)\rightarrow [0, \infty)$, the unique $(\xgrid\times\tgrid)$-piecewise constant function defined by
\begin{equation*}
	u_{N}(k\xmesh,n\tmesh) := \frac{1}{|\square_{N}|} \p{X^{n} = k} = \frac{1}{(\xmesh)^{d}}p(k,n).
\end{equation*}
Recall that Remark \ref{rmk.uNtwodefs} showed that $u_{N}$ can equivalently be defined as the function $u_{N}$ generated by $\mathcal{S}_{N}[A]$ with $A(u) = u^{m+1}$ and initial measure $\delta$. 
\begin{prop} \label{prop.uNfacts}
	For $N \in \N$, let $u_{N}$ be generated by $\mathcal{S}_{N}[A]$ with $A(u) = u^{m+1}$ with initial measure $\delta$. There is a constant $C=C(m,d)>0$ such that the following holds: 
	\begin{enumerate}[(i)]
		\item \label{uN.mass} $\int_{\R^{d}}u_{N}(x,t)\,dx = 1$ for all $t \in [0,\infty)$.
		\item \label{uN.Linfty} For all $t \in (0,\infty)$ and $N \geq 2/t$,
		\begin{equation*}
			0 \leq u_{N}(\cdot,t) \leq Ct^{-d\beta}.
		\end{equation*}
		\item \label{uN.conc} For all $t \in (0,\infty)$, $N \geq 2/t$, and $r > 2\sqrt{d}\xmesh$,
		\begin{equation*}
			\int_{|x| > r} u_{N}(x,t) \,dx\leq C\exp\left(- rt^{-\beta}/C\right).
		\end{equation*}
		\item \label{uN.BV} For all $t \in (0,\infty)$, $N \geq 2/t$, and $r > 2\sqrt{d}\xmesh$,
		\begin{equation*}
			[u_{N}(\cdot,t)]_{\TV(B_{r})} \leq Ct^{-d\beta}r^{d-1}.
		\end{equation*}
	\end{enumerate}
\end{prop}
\begin{proof}
	\begin{enumerate}[(i)]
		
		\item Let $t \in [0,\infty)$. Because $u_{N}(\cdot,t)$, is $\xgrid$-piecewise constant, 
		\begin{align*}
			\int_{\R^{d}}u_{N}(x,t)\,dx = (\xmesh)^{d}\sum_{k \in \Z^{d}} u_{N}(k\xmesh,t) = \sum_{k \in \Z^{d}} p(k, \floor{t/\tmesh}) = \sum_{k \in \Z^{d}}\p{X^{\floor{t/\tmesh}} = k} = 1.
		\end{align*}
		\item By Theorem \ref{prop.pfacts}\eqref{p.Linfty}, there is a constant $C = C(m,d) > 0$, such that
		\begin{equation*}
			0 \leq p(k,n) \leq Cn^{-d\beta} \text{ for all } k \in \Z^{d}, n \in \N.
		\end{equation*}
		Fix $t \in (0,\infty)$, $x \in \R^{d}$, and integer $N \geq 2/t$. Choose $k \in \Z^{d}$ such that $x \in \square_{N}(k\xmesh)$ and let $n = \floor{t/\tmesh}$. Since $\xmesh = (\tmesh)^{\beta}$, we have
		\begin{equation*}
			0 \leq u_{N}(x,t) = (\tmesh)^{-d\beta} p(k,n) \leq (\tmesh)^{-d\beta} Cn^{-d\beta} = C (n\tmesh)^{-d\beta}.
		\end{equation*}
		We conclude using the fact that $N \geq 2/t$, meaning $\frac12 t \geq \tmesh$, so $n \tmesh \geq t - \tmesh \geq \frac12 t.$
		
		\item Since $r > 2\sqrt{d}\xmesh$, we have that $r - \sqrt{d}\xmesh > \frac{1}{2}r$, and thus,
		\begin{equation*}
			\int_{|x| >r} u_{N}(x,t)\,dx \leq (\xmesh)^{d}\sum_{|k\xmesh| > r - \sqrt{d}\xmesh} u_{N}(k\xmesh,t)\leq \sum_{|k\xmesh| \geq \frac12 r}p(k,\floor{t/\tmesh}).
		\end{equation*}
		Let $n = \floor{t/\tmesh}$. Since $\xmesh = \tmesh^{\beta}$, 
		\begin{equation*}
			\sum_{|k\xmesh| \geq \frac12 r}p(k,\floor{t/\tmesh}) = \sum_{|k| \geq \frac12 r\tmesh^{-\beta}}p(k,n) = \sum_{|k| \geq \frac12 r(n\tmesh)^{-\beta}n^{\beta}}p(k,n).
		\end{equation*}
		Combining the previous two displays, we can apply Theorem \ref{prop.pfacts}\eqref{p.conc} to conclude that there is some constant $C = C(m,d)>0$, such that:
		\begin{equation*}
			\int_{|x| >r} u_{N}(x,t)\,dx \leq \sum_{|k| \geq \frac12 r(n\tmesh)^{-\beta}n^{\beta}}p(k,n) \leq C\exp(-r(n\tmesh)^{-\beta}/C).
		\end{equation*}
		We conclude by noting that $\tmesh^{-1}= N \geq 2/t$ means $\frac12 t \geq \tmesh$ and so $n\tmesh \geq t - \tmesh  \geq \frac 12 t$.
		
		\item By Theorem \ref{prop.pfacts}\eqref{p.BV}, there is a constant $C = C(m,d) > 0$ such that
		\begin{equation} \label{e.pBValpha}
			[p(\cdot,n)]_{\TV(B_{\alpha}\cap\Z^{d})} \leq C\alpha^{d-1}n^{-d\beta} \text{ for all } \alpha > 0, n \in \N.
		\end{equation}
		Define the $s$-superlevel set of $u_{N}$ as $E^{s}:= \{x \in \R^{d} : u_{N}(x,t) > s\}$, and $E^{s}_{r} := E^{s} \cap B_{r}$. Using one of the equivalent definitions for total variation, we have
		\begin{equation} \label{e.TVdef2}
			[u_{N}(\cdot,t)]_{\TV(B_{r})} = \int_{-\infty}^{\infty} \mathcal{H}^{d-1}(\partial E^{s}_{r}) \, ds,
		\end{equation}
		where $\mathcal{H}^{d-1}$ is the $(d-1)$-dimensional Hausdorff measure (see for example \cite[Section 5.4, Theorem 5.4.4 and Remark 5.4.2]{ziemer}). Because $u_{N}(\cdot,t)$ is $\xgrid$-piecewise constant, it follows that
		\begin{equation*}
			\partial E^{s}_{r} \subset \bigcup_{k \in \Z^{d}} \partial\square_{N}(k\xmesh).
		\end{equation*}
		For $k \in \Z^{d}$, and $\ell \sim k$, define $\partial_{\ell}\square_{N}(k\xmesh)$ to be the face of the $d$-hypercube $\square_{N}(k\xmesh)$ which has normal vector $\ell-k$. We can re-write
		\begin{equation*}
			\bigcup_{k \in \Z^{d}} \partial\square_{N}(k\xmesh) = \bigcup_{k \in \Z^{d}}\bigcup_{\ell \sim k} \partial_{\ell}\square_{N}(k\xmesh).
		\end{equation*}
		We remark that for every $k \sim \ell$, $\partial_{\ell}\square_{N}(k\xmesh) = \partial_{k}\square_{N}(\ell\xmesh)$. Hence, the right hand side union (up to a set of $\mathcal{H}^{d-1}$ measure zero) exactly ``double counts'' each hypercube face. Therefore, we can re-write the integral \eqref{e.TVdef2} as
		\begin{equation*}
			[u_{N}(\cdot,t)]_{\TV(B_{r})} = \frac{1}{2}\sum_{k \in \Z^{d}} \sum_{\ell \sim k}\int_{-\infty}^{\infty} \mathcal{H}^{d-1}(\partial E^{s}_{r} \cap \partial_{\ell} \square_{N}(k\xmesh)) \, ds.
		\end{equation*}
		Since $E^{s}_{r}:= E^{s} \cap B_{r} \subset B_{r}$, if $|k\xmesh| \geq 2r$ then $|k\xmesh|> r + 2\sqrt{d}\xmesh$ and so for all $\ell\sim k$ in this range,
			$\partial E^{s}_{r} \cap \partial_{\ell} \square_{N}(k\xmesh) = \emptyset.$
		Thus, we can localize the sum over $k$ in the above display as
		\begin{align} \label{e.TVupperbnd}
			[u_{N}(\cdot,t)]_{\TV(B_{r})} &= \frac12\sum_{|k| < 2r(\xmesh)^{-1}} \sum_{\ell\sim k}\int_{-\infty}^{\infty} \mathcal{H}^{d-1}(\partial E^{s}_{r} \cap \partial_{\ell} \square_{N}(k\xmesh)) \, ds \notag\\
			&\leq \frac12\sum_{\substack{|k|,|\ell| < 4r(\xmesh)^{-1} \\ \ell \sim k}} \int_{-\infty}^{\infty} \mathcal{H}^{d-1}(\partial E^{s} \cap \partial_{\ell} \square_{N}(k\xmesh)) \, ds,
		\end{align}
		where we emphasize that we replaced $E^{s}_{r}$ by $E^{s}$ in the last step. 
		
		Next, for a fixed $k \in \Z^{d}$ and $i \in [d]$, we analyze the term $\int_{-\infty}^{\infty} \mathcal{H}^{d-1}(\partial E^{s} \cap \partial_{k+e_{i}} \square_{N}(k\xmesh)) \, ds$. Define
		\begin{align*}
			s_{\min} &:= \min\{u_{N}(k\xmesh,t), u_{N}((k+e_{i})\xmesh,t)\}\\
			s_{\max} &:= \max\{u_{N}(k\xmesh,t), u_{N}((k+e_{i})\xmesh,t)\},
		\end{align*}	
		so $u_{N}(\cdot,t)$ has a jump discontinuity of size $s_{\max}-s_{\min}$ across the face $\partial_{k+e_{i}} \square_{N}(k\xmesh)$. Because of this jump discontinuity, we have
		\begin{equation*}
			\partial E^{s} \cap \partial_{k+e_{i}} \square_{N}(k\xmesh) = \begin{cases}
				\partial_{k+e_{i}} \square_{N}(k\xmesh) &\text{ for } s \in (s_{\min},s_{\max}]\\
				\emptyset &\text{ otherwise}.
			\end{cases}
		\end{equation*}
		Hence, we calculate that 
		\begin{align*}
			\int_{-\infty}^{\infty} \mathcal{H}^{d-1}(\partial E^{s} \cap \partial_{k+e_{i}} \square_{N}(k\xmesh))\,ds
			&= \int_{s_{\min}}^{s_{\max}} \mathcal{H}^{d-1}(\partial_{k+e_{i}} \square_{N}(k\xmesh))\,ds \\
			&= (s_{\max}-s_{\min})\mathcal{H}^{d-1}(\partial_{k+e_{i}} \square_{N}(k\xmesh)) \\
			&= |u_{N}((k+e_{i})\xmesh,t) - u_{N}(k\xmesh,t)| (\xmesh)^{d-1} \\
			&= (\xmesh)^{-1} |p(k+e_{i},\floor{t/\tmesh}) - p(k,\floor{t/\tmesh})|.
		\end{align*}
		Letting $n = \floor{t/\tmesh}$, we combine the above display with \eqref{e.TVupperbnd}, letting $\ell=k\pm e_{i}$, to obtain that
		\begin{equation*}
			[u_{N}(\cdot,t)]_{\TV(B_{r})} \leq (\xmesh)^{-1} [p(\cdot,n)]_{\TV(B_{4r(\xmesh)^{-1}}\cap\Z^{d})}.
		\end{equation*}
		By applying \eqref{e.pBValpha} and $\xmesh = (\tmesh)^{\beta}$, this implies that
		\begin{equation*}
			[u_{N}(\cdot,t)]_{\TV(B_{r})} \leq (\xmesh)^{-1} C (r (\xmesh)^{-1})^{d-1}n^{-d\beta} = C r^{d-1}(n\tmesh)^{-d\beta}.
		\end{equation*}
		Finally, we conclude by again using that $n \tmesh \ge t/2$. 
	\end{enumerate}
\end{proof}
\begin{lem} \label{lem.uNmodulii}
	Fix $\eta \in (0,T)$. Let $u_{N}$ be generated by $\mathcal{S}_{N}$ with initial measure $\delta$.  Then $(u_{N})_{N\geq 2/\eta}$ admits a spatial modulus of continuity $\nu_{\eta}$ on $\R^{d}\times[\eta,T]$, and $(u_{N})_{N \geq 2/\eta}$ admits an approximate temporal modulus of continuity $\om_{\eta}$ on $\R^{d} \times [\eta,T]$.
\end{lem}
\begin{proof}
	Let $t \in [\eta,T]$. We first show that $u_{N}(\cdot,t)$ admits a spatial modulus of continuity $\nu_{\eta}$. Let $|y| \leq 1$ and $r \geq 4$. Then by partitioning $\R^{d}$, 
	\begin{multline} \label{e.uNsmod}
		\norm{u_{N}(\cdot + y,t) - u_{N}(\cdot,t)}_{L^{1}(\R^{d})} = \norm{u_{N}(\cdot+y,t)-u_{N}(\cdot,t)}_{L^{1}(B_{\sfrac{r}{2}})} \\+  \norm{u_{N}(\cdot+y,t)-u_{N}(\cdot,t)}_{L^{1}(B_{\sfrac{r}{2}}^{c})}.
	\end{multline}
	We begin by bounding the second term of \eqref{e.uNsmod}. Note that for $x \in \R^{d}$ with $|x| > r/2$, we have $|x+y| \geq |x|-|y| > r/2 - 1 \geq r/4$. Because $u_{N}$ is nonnegative, Proposition \ref{prop.uNfacts}\eqref{uN.conc} gives that there is a $C = C(m,d) > 0$ such that
	\begin{equation} \label{e.uNsmod1}
		\int_{|x| > r/2}|u_{N}(x+y,t)-u_{N}(x,t)|\,dx \leq 2\int_{|x| > r/4}u_{N}(x,t)\,dx \leq C \exp\left(-r\eta^{-\beta}/C\right),
	\end{equation}
	where we used $t \geq \eta$ in the last step. (Note that Proposition \ref{prop.uNfacts}\eqref{uN.conc} requires that $N>2/t$, which is satisfied since $N>2/\eta$.) 
	Next, we bound the first term in the left hand side of \eqref{e.uNsmod}. Note that for $x \in \R^{d}$ with $|x| < r/2$ and since $|y|<1$, we have $|x+y| < r$. By Proposition \ref{prop.uNfacts}\eqref{uN.BV}, there exists a constant $C = C(m,d) > 0$ such that $[u_{N}(\cdot,t)]_{\TV(B_{r})} \leq Ct^{-d\beta}r^{d-1} \leq C\eta^{-d\beta}r^{d-1}$, so by Proposition \ref{prop.uNfacts}(i), we conclude that $u_{N}(\cdot,t) \in \BV(B_{r})$. Therefore, we can approximate $u_{N}(\cdot,t)$ by a sequence of smooth functions $\{\vp^{i}\}_{i \in \N} \subset C^{\infty}(B_{r})$ (see, for example, \cite[Theorem 5.3.3]{ziemer}) such that 
	\begin{equation*}
		\lim_{i\to \infty}\norm{\vp^{i}-u_{N}(\cdot,t)}_{L^{1}(B_{r})} = 0 \text{ and } \lim_{i\to \infty}\norm{D\vp^{i}}_{L^{1}(B_{r})} = [u_{N}(\cdot,t)]_{\TV(B_{r})},
	\end{equation*}
	where $D\vp$ is the gradient of $\vp$. Note that for all $i \in \N$, 
	\begin{align*}
		\norm{\vp^{i}(\cdot+y)-\vp^{i}}_{L^{1}(B_{r/2})} 	&= \int_{B_{r/2}}\left|\int_{0}^{1}D\vp^{i}(x+sy)\cdot y\,ds\right|\,dx \\
		&\leq |y|  \int_{B_{r/2}}\int_{0}^{1}\left|D\vp^{i}(x+sy)\right|\,ds\,dx \\
		&= |y|  \int_{0}^{1}\norm{D\vp^{i}}_{L^{1}(B_{r/2}(sy))}\,ds \\
		&\leq |y| \cdot \norm{D\vp^{i}}_{L^{1}(B_{r})},
	\end{align*}
	where in the last step we used that $B_{r/2}(sy) \subset B_{r}$ because $|sy| \leq |y| \leq 1 \leq r/2$. 
	Taking $i \to \infty$ then yields
	\begin{equation*}
		\norm{u_{N}(\cdot+ y,t) - u_{N}(\cdot,t)}_{L^{1}(B_{r/2})} \leq   [u_{N}(\cdot,t)]_{\TV(B_{r})}\cdot |y| \leq C\eta^{-d\beta}r^{d-1}|y|.
	\end{equation*}
	Combining this with \eqref{e.uNsmod1}, we return to \eqref{e.uNsmod} to obtain that for all $|y| \leq 1$ and $r \geq 4$, 
	\begin{equation*}
		\norm{u_{N}(\cdot + y,t) - u_{N}(\cdot,t)}_{L^{1}(\R^{d})}\leq C\eta^{-d\beta}r^{d-1}|y| + C \exp\left(-r\eta^{-\beta}/C\right).
	\end{equation*}
	Choosing $r = 4|y|^{-1/d}$ in the above gives
	\begin{equation*}
		\norm{u_{N}(\cdot + y,t) - u_{N}(\cdot,t)}_{L^{1}(\R^{d})}\leq C\eta^{-d\beta}|y|^{1/d} + C \exp\left(-|y|^{-1/d}\eta^{-\beta}/C\right).
	\end{equation*}
	Because the above bound tends to $0$ as $|y| \to 0$, we may define the spatial modulus $\nu_{\eta}$ as
	\begin{equation*}
		\nu_{\eta}(|y|) = \begin{cases}
			0 & |y|=0\\
			C\eta^{-d\beta}|y|^{1/d} + C \exp\left(-|y|^{-1/d}\eta^{-\beta}/C\right) & 0 < |y| \leq 1\\
			\max\left\{C\eta^{-d\beta} + C \exp\left(-\eta^{-\beta}/C\right), 2\right\} &|y|> 1.
		\end{cases}
	\end{equation*}
To obtain an approximate temporal modulus of continuity $\om_{\eta}$ on $\R^{d} \times[\eta,T]$, we appeal to Kru{\v{z}}kov's interpolation lemma (Lemma \ref{lem.KIL}) in a manner similar to Lemma \ref{lem.modulii}(iii) albeit only for times in the interval $[\eta,T]$. The only difference is that that proof requires an initial condition $v_{0} \in L^{\infty}(\R^{d})$ to obtain bounds of the form
	\begin{equation*}
		\norm{v_{N}(\cdot,t)}_{L^{\infty}(\R^{d})} \leq \norm{v_{0}}_{L^{\infty}(\R^{d})} \text{ for } t > 0.
	\end{equation*}
	We can avoid such an assumption for $(u_{N})_{N \geq 2/\eta}$ by using Proposition \ref{prop.uNfacts}\eqref{uN.Linfty} to replace these bounds by bounds of the form
	\begin{equation*}
		\norm{u_{N}(\cdot,t)}_{L^{\infty}(\R^{d})}\leq Ct^{-d\beta} \leq C\eta^{-d\beta} \text{ for }t \in [\eta,T].\qedhere
	\end{equation*}
\end{proof}
The preceding lemma allows us to conclude that $(u_{N})_{N \geq 2/\eta}$ satisfies the hypotheses of Lemma \ref{lem.L1loccmpct} on the domain $\R^{d} \times [\eta,T]$. In Lemma \ref{lem.uNetaconv}, we will show that any $L^{1}_{\text{loc}}(\R^{d}\times [\eta, T])$-limit point of this sequence must be a nonnegative entropy solution of the porous medium equation. However, to prove that lemma, we will need the following result.

\begin{lem} \label{lem.conc}
	Let $u_{N}$ be generated by $\mathcal{S}_{N}[A]$ with $A(r)=r^{m+1}$ and initial measure $\delta$. Fix $t > 0$. Assume that there is a subsequence $(u_{N_i})_{i \in \N}$ such that $u_{N_{i}}(\cdot,t) \to w_{t}$ in $L^{1}_{\loc}(\R^{d})$ for some $w_{t} \in L^{1}(\R^{d})$. Then, 
	\begin{equation} \label{e.wunitmass}
		\int_{\R^{d}}w_{t}(x)\,dx = 1,
	\end{equation} 
	and there exists $C = C(m,d) > 0$ such that for all $r > 0$,
	\begin{equation} \label{e.wconc}
		 \int_{|x| > r} w_{t}(x)\,dx \leq C\exp\left(-rt^{-\beta}/C\right).
	\end{equation}
	Finally, we have that $u_{N_{i}}(\cdot,t) \to w_{t}$ in $L^{1}(\R^{d})$ (instead of merely in $L^{1}_{\loc}(\R^{d})$).
\end{lem}
\begin{proof}
We remark that $u_{N_{i}} \geq 0$, and so $w_{t} \geq 0$ almost everywhere. 
We first prove \eqref{e.wunitmass}.
Applying the hypotheses and Proposition \ref{prop.uNfacts}(i), we have
\begin{equation*}
	\int_{\R^{d}} w_{t}(x) \, dx = \lim_{r\to \infty} \int_{|x| \leq r} w_{t}(x) \, dx = \lim_{r\to \infty} \lim_{i\to\infty} \int_{|x| \leq r} u_{N_{i}}(x,t) \, dx \leq \lim_{r\to \infty} \lim_{i\to\infty} 1 = 1.
\end{equation*}
Next, we use Proposition \ref{prop.uNfacts} (statements (i) and (iii)) to obtain that for all $r > 0$, 
\begin{align}\label{e.wbulk} 
	\int_{|x| \leq r} w_{t}(x) \,dx &= \lim_{i\to\infty} \int_{|x| \leq r} u_{N_{i}}(x,t) \, dx \notag\\
	&= \lim_{i \to \infty}\left(1 - \int_{|x|> r} u_{N_{i}}(x,t) \, dx \right) \notag\\
	&\geq 1 - C\exp(-rt^{-\beta}/C),
\end{align}
where we remark that the hypotheses $N _{i}> 2/t$ and $r > 2\sqrt{d}\xmesh$ are satisfied for all $N_{i}$ large enough.
Taking $r \to \infty$ in the above gives that $\int_{\R^{d}} w_{t}(x)\,dx \geq 1$ and so \eqref{e.wunitmass} is proved. We note that \eqref{e.wconc} follows directly from \eqref{e.wunitmass} and \eqref{e.wbulk}.

For the convergence in $L^{1}(\R^{d})$, choose $r \geq 2\sqrt{d}(2/t)^{-\beta}$ so that $r \geq 2\sqrt{d}\xmesh$ for all $N \geq 2/t$. We combine Proposition \ref{prop.uNfacts}\eqref{uN.conc} and \eqref{e.wconc} to conclude that for $N_{i} \geq 2/t$, 
\begin{align*}
	\norm{u_{N_{i}}(\cdot,t) - w_{t}}_{L^{1}(\R^{d})} &= \norm{u_{N_{i}}(\cdot,t) - w_{t}}_{L^{1}(B_{r})} +  \norm{u_{N_{i}}(\cdot,t) - w_{t}}_{L^{1}(B^{c}_{r})} \\
	&\leq \norm{u_{N_{i}}(\cdot,t) - u^{\eta}(\cdot,t)}_{L^{1}(B_{r})} +  2C\exp\left(-rt^{-\beta}/C\right).
\end{align*}
Taking $N_{i}\to \infty$ followed by $r \to \infty$ in the above gives $u_{N_{i}}(\cdot,t) \to w_{t}$ in $L^{1}(\R^{d})$.
\end{proof}

\begin{lem} \label{lem.uNetaconv}
	Fix $\eta \in (0,T)$. Let $u_{N}$ be generated by $\mathcal{S}_{N}[A]$ with $A(r)=r^{m+1}$ and initial measure $\delta$. There is a subsequence $(u_{N^{\eta}_{n}})_{n \in \N}$ and function $u^{\eta}$ such that
	\begin{equation*}
		u_{N^{\eta}_{n}} \xrightarrow[n\to \infty]{L^{1}_{\loc}(\R^{d}\times[\eta,T])} u^\eta \text{ and, for all }t \in [\eta,T], u_{N^{\eta}_{n}}(\cdot,t) \xrightarrow[n\to\infty]{L^{1}_{\loc}(\R^{d})} u^\eta(\cdot,t).
	\end{equation*} Furthermore, $u^{\eta}$ is a nonnegative entropy solution to $u^{\eta}_{t} - \tfrac{1}{2d}\Delta((u^{\eta})^{m+1}) = 0 \text{ in } \R^{d}\times (\eta,T)$.
\end{lem}
\begin{proof} 
	Let $N_{\eta} \in \N$ be such that $N_{\eta} \geq \frac{2}{\eta}$. Let $\nu_{\eta}$ and $\om_{\eta}$ be respectively a spatial modulus of continuity and an approximate temporal modulus of continuity for $(u_{N})_{N \geq N_{\eta}}$ on $\R^{d} \times [\eta,T]$, the existence of which is guaranteed by Lemma \ref{lem.uNmodulii}. The existence of these moduli of continuity, along with Proposition \ref{prop.uNfacts} (statements (i) and (ii)), entail that $(u_{N})_{N \geq N_{\eta}}$ satisfies the conditions of Lemma \ref{lem.L1loccmpct} on the domain $\R^{d} \times [\eta,T]$. Therefore, there is a subsequence $(u_{N^{\eta}_{n}})_{n \in \N}$ such that $u_{N^{\eta}_{n}} \to u^{\eta}$ in $L^{1}_{\loc}(\R^{d} \times [\eta,T])$ for some $u^{\eta} \in L^{1}(\R^{d}\times[\eta,T])\cap L^{\infty}( \R^{d}\times[\eta,T])\cap C([\eta,T];L^{1}(\R^{d}))$. To simplify the notation, we drop the subsequence $(u_{N^{\eta}_{n}})_{n\in\N}$ and assume without loss of generality that $u^{\eta}$ is a limit point of the entire sequence $(u_{N})_{N \geq N_{\eta}}$. 
	The conclusion of Lemma \ref{lem.L1loccmpct} also yields that $u^{\eta}$ has $\nu_{\eta}$ and $\om_{\eta}$ as a spatial and temporal modulus of continuity, respectively. Hence, we apply Lemma \ref{lem.tslice} to the sequence $(u_{N})_{N\geq N_{\eta}}$ and the constant sequence $(u^{\eta})_{N \in \N}$, deducing that
	\begin{equation} \label{e.uNetaslice}
		\text{for all }t\in [\eta, T], u_{N}(\cdot,t) \xrightarrow{L^{1}_{\loc}(\R^{d})}u^{\eta}(\cdot,t).
	\end{equation}
	
	Now we show that $u^{\eta}$ is a nonnegative entropy solution. By Proposition \ref{prop.uNfacts}\eqref{uN.Linfty}, there is a constant $C = C(m,d) > 0$ such that
	\begin{equation} \label{e.uNbdd}
		\text{for all }t\in [\eta,T], \quad0 \leq u_{N}(\cdot,t) \leq Ct^{-d\beta}.
	\end{equation}
	The convergence \eqref{e.uNetaslice} then implies that
	\begin{equation} \label{e.uetabdd}
		\text{for all }t\in [\eta, T], \quad0 \leq u^{\eta}(\cdot,t) \leq Ct^{-d\beta} \text{ a.e.\@ in } \R^{d}.
	\end{equation}
	Therefore, $u^{\eta}(\cdot,\eta)$ is nonnegative and belongs to $L^{\infty}(\R^{d})$, so by Theorem \ref{t.kr} (in particular, the last part of Remark \ref{rmk.KRdiffs}), there exists a unique nonnegative entropy solution $v$ solving
		\begin{equation} \label{e.veta}
		\begin{cases}
			v_{t} - \frac{1}{2d}\Delta(v^{m+1}) = 0 &\text{in } \R^{d} \times (0,T-\eta) \\
			v(\cdot,0) = u^{\eta}(\cdot,\eta) &\text{in } \R^{d}.
		\end{cases}
	\end{equation}
	By Remark \ref{rmk.solshift}, if we can show that $u^{\eta}(x,t) = v(x,t-\eta)$ for a.e.\@ $(x,t) \in \R^{d}\times(\eta,T)$, we can conclude that $u^{\eta}$ is itself a nonnegative entropy solution as desired. 

Let $v_{N}$ be generated by $\mathcal{S}_{N}$ with initial condition $u^{\eta}(\cdot,\eta)$. Because of \eqref{e.uetabdd}, we can apply the convergence result for diffuse initial conditions, Corollary \ref{c.diffuse}, to obtain that
	\begin{equation} \label{e.vNetaslice}
		v_{N}(\cdot,t) \xrightarrow{L^{1}_{\loc}(\R^{d})}v(\cdot,t)\text{ for }t \in [0,T-\eta].
	\end{equation}
	Since $u^{\eta}(\cdot,\eta)$ has a spatial modulus of continuity, Lemma \ref{lem.modulii}(iii) implies that $(v_{N})_{N \geq N_{\eta}}$ has an approximate temporal modulus of continuity, possibly after increasing $N_{\eta}$ depending on $\norm{u^{\eta}(\cdot,\eta)}_{L^{\infty}(\R^{d})}$. Without loss of generality (by Remark \ref{rmk.WLOGmod}), we take this approximate temporal modulus of continuity to be $\om_{\eta}$. 
	
	Next, let $D \subset \R^{d}$ be compact. For $t \in [\eta,T]$, we calculate
	\begin{align*}
		&\norm{u^{\eta}(\cdot,t) - v(\cdot,t-\eta)}_{L^{1}(D)} \\
		&\leq \norm{ (u^{\eta} - u_{N})(\cdot,t)}_{L^{1}(D)} + 
		\norm{u_{N}(\cdot,t) - v_{N}(\cdot,t-\eta) }_{L^{1}(D)} + 
		\norm{(v_{N} - v)(\cdot,t-\eta)}_{L^{1}(D)}.
	\end{align*}
	By \eqref{e.uNetaslice} and \eqref{e.vNetaslice}, we have
	\begin{equation} \label{e.uetav}
		\norm{u^{\eta}(\cdot,t) - v(\cdot,t-\eta)}_{L^{1}(D)} \leq 
		\limsup_{N\to\infty}\norm{u_{N}(\cdot,t) - v_{N}(\cdot,t-\eta)}_{L^{1}(D)}.
	\end{equation}
	To analyze this upper bound, fix $N \geq N_{\eta}$ and let $\ell_{N} = \floor{(t-\eta)/\tmesh}$. Because $\norm{v_{0}}_{L^{\infty}(\R^{d})} = \norm{u^{\eta}(\cdot,\eta)}_{L^{\infty}(\R^{d})} \leq C\eta^{-d\beta}$, Lemma \ref{lem.modulii}(i) gives that $\norm{v_{N}(\cdot,t)}_{L^{\infty}(\R^{d})} \leq C\eta^{-d\beta}$ for all $t \in [0,T-\eta]$. Furthermore, we have that $\norm{u_{N}(\cdot,t)}_{L^{\infty}(\R^{d})} \leq C\eta^{-d\beta}$ by \eqref{e.uNbdd}. Hence, by possibly taking $N_{\eta}$ larger so that $N \ge N_{\eta}$ satisfies \eqref{e.cfl} with $\Lambda := C\eta^{-d\beta}$, we apply Lemma \ref{lem.SN}(iii) to conclude that
	\begin{align} \label{e.uNvN}
		&\norm{u_{N}(\cdot,t) - v_{N}(\cdot,t-\eta)}_{L^{1}(D)} \\ \notag
		&\leq \norm{u_{N}(\cdot,t) - v_{N}(\cdot,t-\eta)}_{L^{1}(\R^{d})} \\\notag
		&= \norm{\mathcal{S}^{\ell_{N}}_{N}u_{N}(\cdot - \ell_{N}\tmesh)- \mathcal{S}^{\ell_{N}}_{N}v_{N}(\cdot, 0)}_{L^{1}(\R^{d})} \\\notag
		& \leq \norm{u_{N}(\cdot, t-\ell_{N}\tmesh) - v_{N}(\cdot, 0)}_{L^{1}(\R^{d})} \\\
		& \leq \norm{u_{N}(\cdot, t-\ell_{N}\tmesh) - u_{N}(\cdot,\eta)}_{L^{1}(\R^{d})} + \norm{u_{N}(\cdot,\eta) - v_{N}(\cdot, 0)}_{L^{1}(\R^{d})}.\label{e.uNvN2}
	\end{align}
	For the second term in the upper bound, the fact that $u_{N}$ and $v_{N}$ are piecewise constant gives
	\begin{equation*}
		\norm{u_{N}(\cdot,\eta) - v_{N}(\cdot,0)}_{L^{1}(\R^{d})} = (\xmesh)^{d} \sum_{k \in \Z^{d}} |u_{N}(k\xmesh,\eta) - v_{N}(k\xmesh,0)|.
	\end{equation*}
	Since $u_{N}$ is piecewise constant and $v_{N}$ has initial condition $u^{\eta}(0,\eta)$, we have
	\begin{equation*}
	u_{N}(k\xmesh,\eta) = \frac{1}{|\square_{N}|}\int_{\square_{N}}u_{N}(y+k\xmesh,\eta)\,dy , \text{and }  v_{N}(k\xmesh,0)=\frac{1}{|\square_{N}|}\int_{\square_{N}}u^{\eta}(y+k\xmesh,\eta)\,dy.
	\end{equation*}
	Combining the two displays above, we obtain
	\begin{align*}
		&\norm{u_{N}(\cdot,\eta) - v_{N}(\cdot,0)}_{L^{1}(\R^d)} \\&= (\xmesh)^{d} \sum_{k \in \Z^{d}} \left|\frac{1}{|\square_{N}|}\int_{\square_{N}}u_{N}(y+k\xmesh,\eta)\,dy - \frac{1}{|\square_{N}|}\int_{\square_{N}}u^{\eta}(y+k\xmesh,\eta)\right| \\&= \sum_{k \in \Z^{d}} \left|\int_{\square_{N}} \bigg(u_{N}(y+k\xmesh,\eta) - u^{\eta}(y+k\xmesh,\eta)\bigg)\,dy\right| \\
		&\leq \sum_{k \in \Z^{d}} \int_{\square_{N}} \left|u_{N}(y+k\xmesh,\eta) - u^{\eta}(y+k\xmesh,\eta)\right| \,dy\\
		&= \int_{\R^{d}} \left|u_{N}(y,\eta) - u^{\eta}(y,\eta)\right| \,dy \to 0,
	\end{align*}
	where in the last step, we use the final statement of Lemma \ref{lem.conc} to bootstrap the convergence $u_{N}(\cdot,\eta) \to u^{\eta}(\cdot,\eta)$ in $L^{1}_{\loc}(\R^{d})$ to convergence in $L^{1}(\R^{d})$. 
	
	For the first term in \eqref{e.uNvN2}, we use $\om_{\eta}$, an approximate temporal modulus of continuity for $(u_{N})_{N\geq\eta}$ on $\R^{d}\times[\eta,T]$, to obtain
	\begin{equation*}
		\norm{u_{N}(\cdot,t-\ell_{N}\tmesh) - u_{N}(\cdot,\eta)}_{L^{1}(\R^{d})} \leq \om_{\eta}(|t-\ell_{N}\tmesh - \eta| + \tfrac{1}{N}) \xrightarrow{N \to \infty} 0,
	\end{equation*}
	where we used that 	$\ell_{N}\tmesh = \floor{(t-\eta)/\tmesh}\tmesh  \xrightarrow{N \to \infty} (t-\eta)$. 
	Therefore, combining the previous two displays, we return to \eqref{e.uetav} to obtain
	\begin{equation*}
		\norm{u^{\eta}(\cdot,t) - v(\cdot,t-\eta)}_{L^{1}(D)} = 0  \text{ for all }t \in [\eta,T].
	\end{equation*}
	Since $D \subset \R^{d}$ was arbitrary, $u^{\eta}(x,t) = v(x,t-\eta)$ for a.e.\@ $(x,t) \in \R^{d}\times(\eta,T)$. Thus we conclude that $u^{\eta}$ is a nonnegative entropy solution of the porous medium equation with initial condition $\delta$.
\end{proof}

We now present a convergence result for $(u_{N})_{N>0}$ satisfying a finite difference scheme with initial measure $\delta$, which converges to the Barenblatt solution of the porous media equation.

\begin{thm} \label{t.mainpde}
	Let $u_{N}$ be generated by $\mathcal{S}_{N}[A]$ with $A(r)=r^{m+1}$ and initial measure $\delta$. We recall that the Barenblatt solution $\bar{u}$ is given by
	$
		\bar{u}(x,t):=t^{-d\beta}(C-\ga|x|^{2}t^{-2\beta})_{+}^{\frac{1}{m}}.
	$
	Then \begin{equation*}
		u_{N} \xrightarrow[]{L^{1}_{\loc}(Q_{T})} \bar{u} \text{ and, for all }t \in (0,T], u_{N}(\cdot,t) \xrightarrow[]{L^{1}_{\loc}(\R^{d})} \bar{u}(\cdot,t).
	\end{equation*}
\end{thm}

\begin{proof}
	In this proof, we frequently work with restrictions of the functions $u_N$ to a subdomain of the form $\R\times I$ for some interval $I$. We abuse notation by continuing to refer to such restrictions as $u_N$, rather than writing $u_N|_{\R\times I}$. Now that this convention has been pointed out, it should cause no confusion.

	Fix any increasing sequence of positive integers $(M_i)_{i \ge 1}$. 
	We first use a diagonalization argument to show that there exists a subsequence $(N_i)_{i \in\N}$ of $(M_i)_{i \in \N}$ such that
	\begin{equation} \label{e.utildeconv}
		u_{N_{i}} \xrightarrow[]{L^{1}_{\loc}(Q_{T})} \tilde{u} \text{ and, for all }t \in (0,T], u_{N_{i}}(\cdot,t) \xrightarrow[]{L^{1}_{\loc}(\R^{d})} \tilde{u}(\cdot,t),
	\end{equation}
	where $\tilde{u} \in L^{1}_{\loc}(Q_{T})$ is a distributional solution to $\tilde{u}_{t} - \frac{1}{2d}\Delta(\tilde{u}^{m+1}) = 0$. To this end, fix a decreasing sequence $(\eta_{i})_{i\in\N} \subset (0,T)$ with $\lim_{i\to\infty}\eta_{i} = 0$. By applying Lemma \ref{lem.uNetaconv} for each $\eta_{i}$, we have that, for each $i \in \N$, there exists an increasing sequence $(N^{\eta_{i}}_{j})_{j \in \N} \subset \N$ such that, as $j \to \infty$, 
	\begin{equation*}
		u_{N^{\eta_{i}}_{j}} \xrightarrow{L^{1}_{\loc}(\R^{d}\times(\eta_{i},T))} u^{\eta_{i}} \text{ and, for all }t \in [\eta_{i},T], u_{N^{\eta_{i}}_{j}}(\cdot,t) \xrightarrow{L^{1}_{\loc}(\R^{d})} u^{\eta_{i}}(\cdot,t),
	\end{equation*}
	where $u^{\eta_{i}}$ is an entropy solution to $u^{\eta_{i}}_{t} - \frac{1}{2d}\Delta((u^{\eta_{i}})^{m+1}) = 0$ in $\R^{d}\times(\eta_{i},T)$. Furthermore, these subsequences may be constructed iteratively so that $(N^{\eta_{i+1}}_{j})_{j \in \N}$ is a subsequence of $(N^{\eta_{i}}_{j})_{j \in \N}$ for all $i \in \N$.
	It follows that for all $i \in \N$,
	\begin{equation*}
		u_{N^{\eta_{i+1}}_{j}} \xrightarrow[j\to\infty]{L^{1}_{\loc}(\R^{d}\times(\eta_{i},T))}u^{\eta_{i}}\quad\text{ and }\quad u_{N^{\eta_{i+1}}_{j}} \xrightarrow[j\to\infty]{L^{1}_{\loc}(\R^{d}\times(\eta_{i+1},T))}u^{\eta_{i+1}} .
	\end{equation*}
	This means that 
	\begin{equation} \label{e.domainex}
		u^{\eta_{i+1}} \overset{a.e.}{=} u^{\eta_{i}}\text{ in }\R^{d} \times (\eta_{i},T).
	\end{equation}
	In other words, $u^{\eta_{i+1}}$ extends $u^{\eta_{i}}$ to the larger domain $\R^{d}\times(\eta_{i+1},T)$ (up to a set of measure zero). Next, we define $\tilde{u}: \R^{d} \times (0,T]\to \R$ as the a.e. limit of these domain extensions, meaning
	\begin{equation*}
		\tilde{u}(x,t) := \begin{cases}
			u^{\eta_{0}}(x,t) &\text{ if } \eta_{0} \leq t \\
			u^{\eta_{i}}(x,t) &\text{ for } t \in [\eta_{i+1},\eta_{i}), i \in \N.
		\end{cases}
	\end{equation*}
	
	We now show that $\tilde{u} \in C((0,T];L^{1}(\R^{d}))$. Because $u^{\eta_{i}} \in C([\eta_{i},T];L^{1}(\R^{d}))$ for every $i \in \N$, we combine \eqref{e.domainex} with Corollary \ref{cor.aetslice} to conclude that for all $i \in \N$, $t \geq \eta_{i}$ and $j \geq i$, $u^{\eta_{j}}(\cdot,t) \overset{a.e.}{=} u^{\eta_{i}}(\cdot,t)$. Thus, for all $i \in \N$, and $t \geq \eta_{i}$, $\tilde{u}(\cdot,t) \overset{a.e.}{=} u^{\eta_{i}}(\cdot,t)$, and so $\tilde{u} \in C([\eta_{i},T];L^{1}(\R^{d}))$. Therefore, because $\eta_{i}\to 0$, $\tilde{u} \in C((0,T];L^{1}(\R^{d}))$.
	
	Next, note that for all $i \in \N$, using \eqref{e.domainex} and the definition of $\tilde{u}$, we have that $\tilde{u} \overset{a.e.}= u^{\eta_{i}}$ in $\R^{d} \times (\eta_{i},T)$. Therefore, for all $i \in \N$, $\tilde{u}$ is an entropy solution of the porous medium equation in $\R^{d} \times (\eta_{i},T)$, which means, by Proposition \ref{prop.entropdist} that for all $i \in \N$, $\tilde{u}$ is a distributional solution in $\R^{d} \times (\eta_{i},T)$. Hence, because $\eta_{i} \to 0$, $\tilde{u}$ is a distributional solution in $Q_{T}$. Finally, setting $N_{i}:=N^{\eta_{i}}_{i}$, the sequence of functions $(u_{N_i})_{i \in \N}$ then satisfies \eqref{e.utildeconv}.

	Next, we show that $\tilde{u}$ is a.e.\@ equal to the Barenblatt solution $\bar{u}$. Let $\vp \in C_{c}(\R^{d})$. Because $\vp$ is compactly supported and continuous, it is uniformly continuous and thus admits a modulus of continuity, which we denote by $\rho$. Fix $t \in (0,T]$ and $\ve > 0$. By Lemma \ref{lem.conc},
	\begin{equation*}
		\int_{\R^{d}} \tilde{u}(x,t) \, dx = 1 \quad \text{ and }\quad \int_{|x| > \ve} \tilde{u}(x,t) \,dx \leq C\exp(-\ve t^{-\beta}/C).
	\end{equation*}
	This implies
	\begin{align*}
		\left|\int_{\R^{d}}\tilde{u}(x,t)\vp(x)\,dx - \vp(0) \right|&= \left|\int_{\R^{d}}\tilde{u}(x,t)(\vp(x)-\vp(0))\,dx \right| \\
		&\leq \int_{|x| \leq \ve}\tilde{u}(x,t)|\vp(x)-\vp(0)|\,dx + \int_{|x| > \ve}\tilde{u}(x,t)|\vp(x)-\vp(0)|\,dx \\
		&\leq \rho(\ve) \int_{|x| \leq \ve}\tilde{u}(x,t)\,dx + 2 \norm{\vp}_{L^{\infty}(\R^{d})} \int_{|x| > \ve}\tilde{u}(x,t)\,dx \\
		&\leq \rho(\ve) + 2 \norm{\vp}_{L^{\infty}(\R^{d})} C\exp(-\ve t^{-\beta}/C).
	\end{align*}
	Taking $t \to 0^{+}$, followed by $\ve \to 0^{+}$ in the above display gives
	\begin{equation*}
		\lim_{t\to0^{+}}\int_{\R^{d}}\tilde{u}(x,t)\vp(x)\,dx = \vp(0),
	\end{equation*}
	and hence $\tilde{u}$ has initial trace $\delta$. Therefore, by the uniqueness of distributional solutions (Theorem \ref{thm.distunique}), $\tilde{u} = \bar{u}$ a.e.\@ in $\R^{d} \times (0,T)$.
	
	To conclude that the two solutions are a.e.\@ equal on a fixed time slice, note that earlier in the proof, we showed $\tilde{u} \in C((0,T];L^{1}(\R^{d}))$. It is straightforward to use the definition of $\bar{u}$ in \eqref{hformula0} to see that that $\bar{u} \in C((0,T];L^{1}(\R^{d}))$, so by Corollary \ref{cor.aetslice}, it follows that for $t \in (0,T]$, $\tilde{u}(\cdot,t) = \bar{u}(\cdot,t)$ a.e.\@ in $\R^{d}$. 
	Hence, we can re-write \eqref{e.utildeconv} as
	\begin{equation*}
		u_{N_{i}} \xrightarrow[]{L^{1}_{\loc}(Q_{T})} \bar{u} \text{ and, for all }t \in (0,T], u_{N_{i}}(\cdot,t) \xrightarrow[]{L^{1}_{\loc}(\R^{d})} \bar{u}(\cdot,t).
	\end{equation*}
Thus, for any increasing sequence $(M_i)_{i \in \N}$ there exists a further subsequence $(N_i)_{i \in \N}$ along which the above convergences both hold; the result follows.
\end{proof}

We are now ready to prove Theorem \ref{t.main}
\begin{proof}[Proof of Theorem \ref{t.main}]
	
	For $N \in \N$ let $u_N$ be generated by $\cS_N[A]$ with $A(r)=r^{m+1}$ and  initial condition $\delta$. Choosing any $T \geq 1$, by Theorem \ref{t.mainpde} we have that $u_{N}(\cdot,1) \xrightarrow{L^{1}_{\text{loc}}(\R^{d})} \bar{u}(\cdot,1)$.
	
	Now we show convergence in distribution. Let $D \subset \R^{d}$ be a continuity set, meaning $|\partial D|=0$, and take $N \geq 2$. Since $u_{N}(k\xmesh,n\tmesh) = |\square_{N}|^{-1}\p{X^{n} = k}$, 
	\begin{align*}
		\p{\frac{1}{N^{1/(dm+2)}}X^{N} \in D} &= \p{\xmesh X^{N} \in D} \\
		&=\sum_{k\xmesh \in (D \cap \xgrid)}\p{X^{N}= k} \\
		&=\sum_{k\xmesh \in (D \cap \xgrid)} u_{N}(k\xmesh,1) |\square_{N}|,
	\end{align*}
	where in the last step we used that $|\square_{N}| = (\xmesh)^{d}$ and $N \tmesh = 1$. To write this last term as an integral of $u_{N}$ over $D$, we need to account for the discretization of the domain. Define
	\begin{equation*}
		D_{N} := \bigcup_{\square_{N}(k\xmesh) \subset D}\square_{N}(k\xmesh).
	\end{equation*}
	Continuing from our previous calculation, we have
	\begin{equation*}
		\p{\frac{1}{N^{1/(dm+2)}}X^{N} \in D} =\sum_{k\xmesh \in D \cap \xgrid} u_{N}(k\xmesh,1) |\square_{N}| =\int_{D_{N}}u_{N}(x,1)\,dx.
	\end{equation*}
	On the other hand, by the definition of the random variable $B$,
	\begin{equation*}
		\p{B \in D} = \int_{D}\bar{u}(x,1)\,dx.
	\end{equation*}
	Combining the above two displays, this implies
	\begin{align*}
		&\left| \p{\frac{1}{N^{1/(dm+2)}}X^{N} \in D} - \p{B \in D}\right| \\
		&= \left| \int_{D_{N}}u_{N}(x,1)\,dx - \int_{D}\bar{u}(x,1)\,dx\right| \\
		&\leq \left| \int_{D_{N}}u_{N}(x,1)\,dx - \int_{D}u_{N}(x,1)\,dx\right| + \left| \int_{D}u_{N}(x,1)\,dx - \int_{D}\bar{u}(x,1)\,dx\right|.
	\end{align*}
	For the latter term in the upper bound, since $u_{N}(\cdot,1) \xrightarrow[N\to\infty]{L^{1}_{\text{loc}}(\R^{d})} \bar{u}(\cdot,1)$, 
	\begin{equation*}
		\left| \int_{D}u_{N}(x,1)\,dx - \int_{D}\bar{u}(x,1)\,dx\right| \xrightarrow[N\to \infty]{} 0.
	\end{equation*}
	For the other term, for $N \geq 2$, by taking $t = 1$ in Proposition \ref{prop.uNfacts}\eqref{uN.Linfty}, we have that $\norm{u_{N}(\cdot,1)}_{L^{\infty}(\R^{d})} \leq C$ and so
	\begin{equation*}
		\left| \int_{D_{N}}u_{N}(x,1)\,dx - \int_{D}u_{N}(x,1)\,dx\right| = \left|\int_{D\setminus D_{N}}u_{N}(x,1)\,dx\right| \leq C |D \setminus D_{N}|.
	\end{equation*}
	Since $\limsup_{N\to\infty}|D\setminus D_{N}| = 0$, we combine the previous three displays to conclude that
	\begin{equation*}
		\p{\frac{1}{N^{1/(dm+2)}}X^{N} \in D} \xrightarrow[N\to\infty]{} \p{B \in D}\, ,
	\end{equation*}
	which implies that $N^{-1/(dm+2)}X^{N} \rightarrow B$ in distribution. 
\end{proof}

\section{Appendix}\label{app}

\subsection{A lower bound on the breakdown time}

In this section, we consider properties of $b$-lazy random walks. Let $Y_n$ be a $b$-lazy symmetric random walk in $d$ dimensions started from initial distribution $h(\cdot)=\indc_{0}(\cdot)$ (i.e. started from the origin). Let $q^{(b)}_n(\x) = \p{Y_n = \x}$.

For the convenience of the reader, we recall that $T(R,b)$ as in \eqref{defoftrb} is given by
\begin{equation*}
    T(R, b) := \min\left\{n \ge 0: \text{there exists } k \in \mathbb Z^d \text{ with }|\x| > R\text{ and }\Delta q^{(b)}_n(\x) < 0\right\}.\end{equation*}
In particular, since $q^{(b)}_{n+1} = q^{(b)}_n + \frac{b}{2d}\Delta q^{(b)}_n$, where $\Delta q^{(b)}_n(\x) = \sum_{\y \sim \x} (q^{(b)}_n(\y) - q^{(b)}_n(\x))$ is the discrete Laplacian, we see that
$$\Delta q_n^{(b)} = {q^{(b)}_{n+1} - q^{(b)}_n \over b/2d},$$
so $T(R, b)$ can be interpreted as the first time $n$ where for some $|\x|>R$, the sequence $(\p{Y_{n}=\x})_{n\geq 0}$ decreases in time. 

\subsubsection{Non-lazy random walk.}
\gdef\lapconst{\Lambda}
We begin by proving a related statement about a non-lazy random walk in $d$ dimensions, which we will later relate to lazy random walks. Let $\left\{\bar{Y}_{n}\right\}$ be the simple, symmetric random walk on $\Z^{d}$ started at the origin. That is, $\bar{Y}_0 = 0$, and $\bar{Y}_{n+1} - \bar{Y}_n$ is chosen uniformly and independently from $\{\pm e_{i}\}_{i=1}^{d}$. 

In this case, it does not make sense to compute $T(R,b)$ for $\bar{Y}$, because $\p{\bar{Y}_{n}=\x}=0$ whenever $n + \sum_{i=1}^{d} \x_i$ is odd. Nevertheless, we are able to compare $\p{\bar{Y}_n = \x}$ and $\p{\bar{Y}_{n-2} = \x}$. 
\begin{thm}\label{simple}Let $\bar{Y}_n$ be as above. There is a constant $\lapconst=\Lambda(d) > 0$ such that, if $2 \le n \le \lapconst|\x|^2$, then $$\left(1 + {1 \over n}\right)\p{\bar{Y}_{n-2} = \x} \le \p{\bar{Y}_n = \x}.$$\end{thm}

In order to prove Theorem \ref{simple}, we require a technical lemma which we first state (and delay its proof until after the proof of Theorem \ref{simple}). 

\begin{lem}\label{simple.estim}Suppose $\x, m \in \mathbb Z^d$ and $|\x_i| \le m_i$ with $\sum_{i=1}^{d} m_i = n-2$. Let $\lapconst := (8d^4)^{-1}$. If $2 \le n \le \lapconst|\x|^2$, then $${n(n-1) \over d^3} \sum_{i=1}^{d} {1 \over ((m_i+2)^2 - \x_i^2)} \ge 1 + {1 \over n}.$$\end{lem}

Equipped with Lemma \ref{simple.estim}, we now present the proof of Theorem \ref{simple}. 

\begin{proof}[Proof of Theorem \ref{simple}] Suppose we are given a path $(\bar{Y}_0, \ldots, \bar{Y}_n)\in (\Z^{d})^{n}$. Let $M(n) \in \mathbb Z^d$ denote the \emph{move vector} of the path up to time $n$, whose $i$th-coordinate is given by
\begin{equation*}
M(n)_i = \#\{0 \le j < n: \bar{Y}_{j+1} - \bar{Y}_j = \pm e_i\}.
\end{equation*}
Observe that for each $i=1,\ldots, d$, $M(n)_i$ is the number of times that the $i$-th coordinate has changed, up to time $n$.

For $k,m\in \Z^{d}$ and $n\in \N$, let $P_n(\x; m)$ be the probability that $\bar{Y}_n = \x$ and $M(n) = m$. These probabilities have an exact formula under certain hypotheses on $m$: if 
\begin{equation}\label{e.movec}
\begin{cases}
\sum_{i=1}^{d} m_i = n,\\
\text{$m_i \ge |\x_i|$, for $i=1, \ldots, d$},\\
\text{$m_i+\x_i$ is even for $i=1,\ldots,d$},
\end{cases}
\end{equation}
then
$$P_n(\x; m) = \left(\begin{gathered}
    n\\[-.2em]
    {\displaystyle {m_1 + \x_1 \over 2}, {m_1 - \x_1 \over 2}, \ldots, {m_d + \x_d \over 2}, {m_d - \x_d \over 2}}
    \end{gathered}\right) (2d)^{-n},$$
where $\genfrac{(}{)}{0pt}{}{n}{a,\,b,\,c, \ldots}$ denotes the multinomial coefficient. If any of those three conditions in \eqref{e.movec} fail, then the probability $P_n(\x;m)$ is zero.

Given $k,m \in \Z^d$, if $P_{n-2}(\x;m)>0$, then for any $i\in \left\{1, \ldots, d\right\}$, $P_{n}(\x;m+2e_i)>0$, and the above exact formula yields the identity 
\begin{align}\label{eq:pnidentity}
P_{n-2}(\x; m) &= {\displaystyle 4d^2 \left({m_i + 2 + \x_i \over 2}\right) \left({m_i + 2 - \x_i \over 2}\right) \over n(n-1)} P_n(\x; m + 2e_i) \nonumber\\
&= {d^2 \over n(n-1)} ((m_i + 2)^2 - \x_i^2) P_n(\x; m + 2e_i).\nonumber
\end{align}
Rearranging the above expression and taking an average over $i = 1, \ldots, d$, we obtain that for any valid vector $m$ for a path of length $n-2$, 
$$\left[{n(n-1) \over d^3}\sum_{i=1}^{d} {1 \over ((m_i+2)^2 - \x_i^2)}\right] P_{n-2}(\x; m) = {1 \over d} \sum_{i=1}^{d} P_n(\x; m + 2e_i).$$

By Lemma~\ref{simple.estim}, for $\lapconst = (8d^4)^{-1}$, since $2 \le n \le \lapconst|\x|^2$ by hypothesis, we conclude $$\left(1 + {1 \over n}\right) P_{n-2}(\x; m) \le {1 \over d} \sum_{i=1}^{d} P_n(\x; m + 2e_i).$$
Summing over $m\in \Z^{d}$ such that $P_{n-2}(k,m)>0$, we obtain
$$
\left(1 + {1 \over n}\right) \sum_{m:P_{n-2}(\x,m)>0} P_{n-2}(\x; m) \le {1 \over d} \sum_{m:P_{n-2}(\x,m)>0} \sum_{i=1}^{d} P_n(\x; m + 2e_i).
$$ 
The left side is $(1+1/n) \p{\bar{Y}_{n-2} = \x}$, and each move vector $m'$ with $\sum_{i=1}^{d} m'_i = n$ can be written as a sum $m + 2e_i$ in at most $d$ ways, so the right side is no greater than $\sum_{m'} P_n(\x; m') = \p{\bar{Y}_n = \x}$. This proves the theorem.
\end{proof}

We now return to the proof of Lemma \ref{simple.estim}. 
\begin{proof}[Proof of Lemma \ref{simple.estim}] For ease of notation, let $y_i = m_i+2$. We begin by factoring
\begin{equation}\label{e.factor}\sum_{i=1}^d {1 \over y_i^2 - \x_i^2} = \sum_{i=1}^d {1 \over y_i^2} + \sum_{i=1}^d {\x_i^2 \over (y_i^2-\x_i^2)y_i^2}.
\end{equation}
We estimate each term in the right hand side from below. Let $\mu := d^{-1} \sum_{i=1}^{d} y_i $. We apply Jensen's inequality with the function $f(x) = 1/x^2$:
$$\sum_{i=1}^{d} {1 \over y_i^2} = d \left({1 \over d} \sum_{i=1}^{d} {1 \over y_i^2}\right) \ge {d \over \mu^2} = {d^3 \over (\sum_{i=1}^{d} y_i)^2}.$$

Since $\sum_{i=1}^{d}m_{i}=n-2$, $\sum_{i=1}^{d} y_i=n+2d-2$, so\begin{align}
    {n(n-1) \over d^3} \sum_{i=1}^{d} {1 \over y_i^2}
        &\ge {n(n-1) \over (\sum_{i=1}^{d} y_i)^2}
        \notag
        \\&= {n(n-1) \over (n+2d-2)^2}
        \notag
        \\&= \left(1-{1 \over n}\right)\left(1+{2d-2 \over n}\right)^{-2}.
        \notag
\end{align} 
Then, since $(1+(2d-2)/n)^{-2} \ge (1-(2d-2)/n)^{2}\geq 1-(4d-4)/n$, we have 
$$
{n(n-1) \over d^3} \sum_{i=1}^{d} {1 \over y_i^2} 
\ge
\left(1-\frac1n\right)\left(1-\frac{4d-4}{n}\right)
\ge 1 - {4d-3 \over n}.
$$

We estimate the second term in \eqref{e.factor} in the following way. Since $\sum_{i=1}^{d} m_i=n - 2$, for each $i=1, \ldots, d$, $y_i = m_i + 2 \le n$. Moreover, $\x_i^2/(y_i^2 - \x_i^2)y_i^2$ is at least $\x_i^2/y_i^4$. Therefore\begin{align}
    {n(n-1) \over d^3} \sum_{i=1}^{d} {\x_i^2 \over (y_i^2 - \x_i^2)y_i^2}
    &\ge {n(n-1) \over d^3} \sum_{i=1}^{d} {\x_i^2 \over y_i^4}.
    \notag\\&\ge {n(n-1) \over d^3 n^4} \sum_{i=1}^{d} \x_i^2 \notag\\&\ge{1\over 2d^3n^2} |\x|^2\label{simple.estimb}
\end{align}
where in the final line, we used the estimate $n(n-1) \ge 1/2n^2$ because $n \ge 2$. Our assumption is that $n \le \lapconst |\x|^2$, so the estimate in (\ref{simple.estimb}) is bounded from below by $1/2\lapconst d^3 n = 4d/n$. Combining this with the prior display, we obtain $${n(n-1) \over d^3} \sum_{i=1}^{d} {1 \over y_i^2 - \x_i^2} \ge 1 - {4d-3 \over n} + {4d \over n} = 1 + {3 \over n}.$$This proves the lemma.
\end{proof}

Before returning back to the lazy random walk, we first prove a preparatory lemma which tells us the asymptotic behaviour of $q_n^{(1/2)}(\x)$ when $n = \lfloor |\x|^2 /16d^4 \rfloor$ and $\x \to \infty$.

Recall that $q^{(1/2)}_n(\x) = \mathbb P\left\{X_1 + \cdots + X_n = \x\right\}$ is the probability distribution at step $n$ of a $\frac12$-lazy random walk in $\mathbb{Z}^d$. Since $\mathbb E[X_1]$ is the zero vector and the covariance matrix is $\text{Cov}[X_1] = \text{Id}/2d$, the central limit theorem yields that $(X_1 + \cdots + X_n) / \sqrt{n} \to_d N(0, \text{Id}/2d)$. 

However, we will require a sharper version of this convergence. By a result of Bhattacharya and Ranga Rao \cite[Chapter 5, Theorem 22.1]{bhattacharya}, it follows that
\begin{equation}\label{weightlowerbound.a}\sup_{k \in \mathbb Z^d} (1 + |\x|^2)\left|q^{(1/2)}_n(k) - \varphi_n(\x)\right| = o(n^{-d/2})\end{equation}
where $\varphi_n(\x) := (d / \pi n)^{d/2}e^{-d|\x|^2/n}$ is the continuous probability density of $N(0, n\text{Id}/2d)$.

(Note: we only need the leading order, so in the notation of \cite{bhattacharya}, we set $s = 2$. Also note that the definition on page 52 of \cite{bhattacharya} says $\widetilde P_0 \equiv 0$, but in fact one should set $\widetilde P_0 \equiv 1$, and $P_0(-\phi; \{\chi_\nu\})$ is then the standard normal distribution.)

We are now ready to state the precise bound on $q_{n}^{(1/2)}$. 
\begin{lem}\label{weightlowerbound} Let $\eta > 0$. There are positive constants $k_{1}=k_{1}(\eta, d)$, $a=a(\eta, d)$ such that for any $\x \in \mathbb Z^d$ with $|\x| \ge k_{1}$,
$$q_n^{(1/2)}(\x) \ge \frac{a}{n^{d/2}} \qquad \text{where } n := \lfloor |\x|^2\eta^{-1}\rfloor.$$\end{lem}

\begin{proof}
Let $\varepsilon_n$ be the scaled maximum difference between the two functions $q_n^{(1/2)}$ and $\varphi_n$ over lattice points in $\mathbb Z^d$, i.e. $$\varepsilon_n := n^{d/2} \sup_{\x \in \mathbb Z^d} \left|q_n^{(1/2)}(\x) - \varphi_n(\x)\right|.$$ The supremum here is bounded above by the supremum in (\ref{weightlowerbound.a}), so $\varepsilon_n$ vanishes as $n \to \infty$.

Let $a := (1/2) (d/\pi)^{d/2} e^{-2d\eta}$. Let $n_0$ be the smallest positive integer so that $\varepsilon_n \le a$ for $n \ge n_0$. Such an integer exists because $\varepsilon_n \to 0$. Let $k_1: = \sqrt{2\eta n_0}$ so that $k_{1}^{2}\eta^{-1}=2n_{0}$.

Suppose $|\x| \ge k_{1}$. We first observe that, for any $\al\geq 1$, then $\lfloor \al \rfloor \ge \al/2$. Thus for $|\x|\geq k_{1}$, $|\x|^2\eta^{-1} \ge k_{1}^2\eta^{-1} = 2n_0 \ge 1$, so by definition $n = \lfloor |\x|^2\eta^{-1} \rfloor \ge |\x|^2 (2\eta)^{-1}$.

We conclude that for $n= \lfloor |\x|^2\eta^{-1} \rfloor$, $n\geq |\x|^2(2\eta)^{-1} \ge k_{1}^2 (2\eta)^{-1}= n_0$, and thus $\varepsilon_n\le{a}$. Furthermore, since $n \ge |\x|^2(2\eta)^{-1}$, this implies $d|\x|^2/n \le 2d\eta$, and hence, 
$$\varphi_n(\x) = \left(d \over \pi n\right)^{d/2} e^{-d|\x|^2/n} \ge \left(d \over \pi n\right)^{d/2} e^{-2d\eta} = {2a \over n^{d/2}}$$

The definition of the error $\varepsilon_n$ implies $|q_n^{(1/2)}(\x) - \varphi_n(\x)| \le \varepsilon_n/n^{d/2}$, and thus $q_n^{(1/2)}(\x) \ge \varphi_n(\x) - \varepsilon_{n}/n^{d/2}$. Combine this with the upper bound on $\varepsilon_n$ and the lower bound $\varphi_n(\x)$, we obtain 
\begin{equation}\label{weightlowerbound.b} q_n^{(1/2)}(\x) \ge \varphi_n(\x) - {\varepsilon_n \over n^{d/2}} \ge \varphi_n(\x) - {a \over n^{d/2}} \ge {a \over n^{d/2}}.\end{equation}\end{proof}

We now convert two-step information about a non-lazy random walk $\bar{Y}_{n}$ into one-step information about a 1/2-lazy random walk $Y_{n}$.

\begin{lem}\label{half}If $|\x| \ge 3$ and $1 \le n \le \lapconst|\x|^2$ for $\La=\La(d)$ as in Theorem \ref{simple}, then $$\left(1 + {1 \over 8n}\right)q^{(1/2)}_{n-1}(\x) \le  q^{(1/2)}_n(\x).$$\end{lem}

\gdef\hatq{{\widehat q}} 
\noindent \emph{Proof.} Throughout this proof, the point $\x$ is fixed, so we suppress it from the notation whenever there is no ambiguity. Let $\omega_1, \omega_2, \ldots$ be a sequence of IID $\text{Bernoulli}(1/2)$-distributed $\left\{0,1\right\}$-valued random variables. Then $q_n^{(1/2)}$ can be interpreted as an expectation of the non-lazy (simple) random walk $q^{(1)}_{\cdot}$ at a random time:
\begin{align*}
q^{(1/2)}_n &= \E{q^{(1)}_{\omega_1 + \cdots + \omega_n}}.\end{align*}
If $n=1,2$, then $q^{(1/2)}_{n-1}(k)$ is zero because $|\x| \ge 3$, so the conclusion of the lemma is vacuously true. If $n \ge 2$, then $\omega_1 + \omega_2 = 0,1,2$ with respective probabilities $1/4, 1/2, 1/4$, so
\begin{align*}
q^{(1/2)}_n &= \frac14 \E{q^{(1)}_{\omega_3 + \cdots + \omega_n} + 2q^{(1)}_{1 + \omega_3 + \cdots + \omega_n} + q^{(1)}_{2 + \omega_3 + \cdots + \omega_n}}\notag
\\&= \frac14 \E{\hatq_{2+\omega_3 + \cdots + \omega_n}}
\end{align*} where $\hatq_n := q^{(1)}_n + 2q^{(1)}_{n-1} + q^{(1)}_{n-2}$. The value of this depends on parity:$$\hatq_n(k) = \begin{cases}q^{(1)}_n(k) + q^{(1)}_{n-2}(k)&\text{ if }n+\sum k_i\text{ is even}\\2q^{(1)}_{n-1}(k)&\text{ if }n+\sum \x_i\text{ is odd.}\end{cases}$$

If $3 \le n \le \lapconst |\x|^2$, then Theorem~\ref{simple} holds for both $n$ and $n-1$: $$q_n^{(1)} \ge \left(1 + {1 \over n}\right)q_{n-2}^{(1)} \qquad q_{n-1}^{(1)} \ge \left(1 + {1 \over n-1}\right)q_{n-3}^{(1)} \ge \left(1 + {1 \over n}\right) q_{n-3}^{(1)}.$$ If $n + \sum \x_i$ is even, then $(n-1) + \sum \x_i$ is odd, so \begin{align*}\hatq_n &= q^{(1)}_n + q^{(1)}_{n-2} \\&\ge \Big(1+\frac{1}{n}+1\Big) q^{(1)}_{n-2} \ge \Big(1 + \frac{1}{2n}\Big) \hatq_{n-1}.\end{align*}
If $n + \sum \x_i$ is odd, then by similar reasoning, \begin{align*}\hatq_n &= 2q^{(1)}_{n-1} \\&= \Big(1 + \frac{1}{4n}\Big) q^{(1)}_{n-1} + \Big(1-\frac{1}{4n}\Big) q^{(1)}_{n-1} \\&\ge \Big(1 + \frac{1}{4n}\Big) q^{(1)}_{n-1} + \Big(1-\frac{1}{4n}\Big)\Big(1+\frac{1}{n}\Big) q^{(1)}_{n-3}  \\&\ge\Big(1 + \frac{1}{4n}\Big)(q^{(1)}_{n-1} + q^{(1)}_{n-3})\\&=\Big(1 + \frac{1}{4n}\Big) \hatq_{n-1},\end{align*}
where in the fourth line, we used the estimate that since $n\geq 1$, $(1-(4n)^{-1}) (1+n^{-1}) = 1 + 3(4n)^{-1} - (4n^2)^{-1} \ge 1+(2n)^{-1}$. 

In both cases, we get the inequality $(1+(4n)^{-1})\hatq_{n-1} \le \hatq_n$, and thus, since $2+\omega_{3}\cdots+\omega_{n-1}\leq n-1\leq n$, we have
\begin{align*} 
\Big(1+\frac{1}{4n}\Big) \E{\hatq_{2 + \omega_3 + \cdots + \omega_{n-1}}}
&\leq \frac{1}{4}\Big(1+\frac{1}{4(2 + \omega_3 + \cdots + \omega_{n-1})}\Big) \E{\hatq_{2 + \omega_3 + \cdots + \omega_{n-1}}}\\
&\leq \E{\hatq_{2 + \omega_3 + \cdots + \omega_{n-1}+1}}\\
&=2\E{\hatq_{2 + \omega_3 + \cdots + \omega_{n-1}+\omega_{n}}}-\E{\hatq_{2 + \omega_3 + \cdots + \omega_{n-1}}}. 
\end{align*}
We arranging and multiplying by $1/4$, we obtain 
\begin{equation*}
\frac{1}{4}\Big(2+\frac{1}{4n}\Big) \E{\hatq_{2 + \omega_3 + \cdots + \omega_{n-1}}}\leq \frac{2}{4}\E{\hatq_{2 + \omega_3 + \cdots + \omega_{n-1}+1}}, 
\end{equation*}
and hence
\begin{equation*}
\Big(2+\frac{1}{4n}\Big)q_{n-1}^{(1/2)}\leq 2q_{n}^{(1/2)},
\end{equation*}
which implies the result.
\qed 

Finally, we are ready to prove Theorem \ref{inc}. Recall that we aim to show that there are constants $\x_2$ such that , if the lazy parameter $b\leq 1/8$, then for $|\x| \ge \x_2$ and $n \le \lapconst |\x|^2/6b$ for $\La$ as in Theorem \ref{simple}, $$q^{(b)}_n(\x) \le q^{(b)}_{n+1}(\x).$$

\begin{proof}[Proof of Theorem \ref{inc}]

Let $\nu_1, \nu_2, \ldots$ be independent $\left\{0,1\right\}$-valued random variables that are $\text{Bernoulli}(2b)$-distributed. Consider the following process on $\mathbb Z^{d}$. Starting from the origin, at step $n$, if $\nu_n = 1$, then we take a $1/2$-lazy random walk step, i.e. with probability $1/2$, we take a unit step in one of the $2d$ directions, or with probability $1/2$, we stay where we are. If $\nu_n = 0$, we stay put.

On the one hand, this is just a $b$-lazy random walk, because we jump to a uniformly chosen neighbour with probability $b$ at each step. On the other hand, at the $n$-th step, we have taken $\nu_1 + \cdots + \nu_n$ steps in a $1/2$-lazy random walk. Consequently, $$q_n^{(b)}(\x) = \mathbb E\left[q_{\nu_1 + \cdots + \nu_n}^{(1/2)}(\x)\right].$$
We may equivalently express this as $$q^{(b)}_n(\x) = \sum_{m=0}^\infty q^{(1/2)}_m(\x) \,\mathbb P\{\nu_1 + \cdots + \nu_n = m\}.$$
Let $\delta q_j^{(1/2)}(k):= q^{(1/2)}_{j+1}(\x) - q^{(1/2)}_j(\x)$. Then since $|k|\geq k_{2}$ to be chosen, by a telescoping series, we have 
\begin{align*}\sum_{m=0}^\infty q^{(1/2)}_m(\x) \mathbb P\{\nu_1 + \cdots + \nu_n = m\}&=\sum_{m=0}^{\infty} \sum_{j =0}^{m-1} \delta q^{(1/2)}_j(k) \mathbb P\{\nu_1 + \cdots + \nu_n = m\} \\&=\sum_{j=0}^{\infty} \delta q^{(1/2)}_j(k)\sum_{m=j+1}^{\infty} \mathbb P\{\nu_1 + \cdots + \nu_n = m\} \\&=\sum_{j=0}^{\infty} \delta q_j ^{(1/2)}(k)\mathbb P\{\nu_1 + \cdots + \nu_n \ge j+1\}.\end{align*} 
That means that the difference of successive probabilities is \begin{align*}q^{(b)}_{n+1}(\x) - q^{(b)}_n(\x) &= \sum_{j=0}^{\infty} \delta q_j^{(1/2)}(k) \big(\mathbb P\{\nu_1 + \cdots + \nu_{n+1} \ge j+1\} - \mathbb P\{\nu_1 + \cdots + \nu_n \ge j+1\}\big) \\&= \sum_{j=0}^{\infty} \delta q_j^{(1/2)}(k)\mathbb P\{\nu_1 + \cdots + \nu_n = j, \nu_{n+1} = 1\} \\&= 2b \sum_{j=0}^{\infty} \delta q_j^{(1/2)}(k) \mathbb P\{\nu_1 + \cdots + \nu_n = j\}.\end{align*}This is a weighted sum of the differences $\delta q_j^{(1/2)}(k)$. By Lemma~\ref{half}, if $j + 1 \le \lapconst |\x|^2$, then $$\delta q^{(1/2)}_j (k)\geq{1 \over 8(j+1)} q_j^{(1/2)}(\x) \ge {1 \over 8\lapconst |\x|^2} q_j^{(1/2)}(\x).$$

Applying this lower bound to the prior display, for the indices $j = 0, \ldots, \lfloor \lapconst |\x|^2 \rfloor - 1$, we obtain \begin{equation}\begin{aligned}[t]\label{lowb} q^{(b)}_{n+1}(\x) - q^{(b)}_n(\x) &\ge {b \over 4\lapconst |\x|^2} \sum_{j=0}^{\lfloor \lapconst |\x|^2 \rfloor - 1} q^{(1/2)}_j(k) \,\mathbb P\{\nu_1 + \cdots + \nu_n = j\}\\&\qquad+2b\sum_{j=\lfloor \lapconst |\x|^2 \rfloor}^\infty \delta q_j^{(1/2)}(k) \,\mathbb P\{\nu_1 + \cdots + \nu_n = j\}\\
&=: S_{1}+S_{2}.\end{aligned}\end{equation} 
We want to show that the right hand side is nonnegative. The rough idea is that $\mathbb P\{\nu_1 + \cdots + \nu_n = j\}$ should decrease very quickly as $j$ increases when $j$ is much larger than $\mathbb E[\nu_1 + \cdots + \nu_n] = 2bn \le \lapconst |\x|^2/3$.

It turns out that we get exponential decay if $j \ge \lapconst |\x|^2/2$, as we now show. The probabilities in (\ref{lowb}) have a simple formula, $\mathbb P\{\nu_1 + \cdots + \nu_n = j\} = (2b)^j (1-2b)^{n-j} \genfrac{(}{)}{0pt}{}{n}{j}
$. In particular, the ratio of consecutive probabilities is $${\mathbb P\{\nu_1 + \cdots + \nu_n = j+1\} \over \mathbb P\{\nu_1 + \cdots + \nu_n = j\}} = {\genfrac{(}{)}{0pt}{}{n}{j+1}
 (2b)^{j+1} (1-2b)^{n-j-1} \over \genfrac{(}{)}{0pt}{}{n}{j}
 (2b)^j (1-2b)^{n-j}} = {2b(n-j) \over (1-2b)(j+1)}.$$ If $j + 1 \ge \lapconst |\x|^2 / 2$,  then since by assumption $n \le \lapconst |\x|^2 / 6b$ and $b \le 1/8$, then
$${2b(n-j) \over (1-2b)(j+1)} \le {2bn \over (3/4) (\lapconst |\x|^2/2)} \le \frac{\Lambda |k|^{2}}{3}\frac{4}{3}\frac{2}{\Lambda |k|^{2}}= {8 \over 9},$$
which implies that 
\begin{equation*}
{\mathbb P\{\nu_1 + \cdots + \nu_n = j+1\} \over \mathbb P\{\nu_1 + \cdots + \nu_n = j\}}\leq \frac{8}{9}. 
\end{equation*}
We now iterate this estimate. Let $\ga := \lfloor \lapconst |\x|^2 / 2 \rfloor$. If $j \ge \ga$, then \begin{equation}\label{inc.chain}\mathbb{P}\{\nu_1 + \cdots + \nu_n = j\} \le (8/9) \mathbb{P}\{\nu_1 + \cdots + \nu_{n} = j-1\} \le \cdots \le (8/9)^{j-\ga} \rho\end{equation}where $\rho := \mathbb{P}\{\nu_1 + \cdots + \nu_n = \ga\}$.

We will use this to compare $|S_{2}|$ with $S_{1}$.  Let $a, \x_1$ be the constants from Lemma \ref{weightlowerbound} with $\eta=\La^{-1}$. Let $\x_2\geq \max\{1, 4/\lapconst, \x_1\}$, and sufficiently large so that, for $c$ to be chosen,
\begin{equation}\label{ak0bound}
{a \over 4\lapconst^{1+d/2} |\x|^{2+d}} \ge c\left(\frac{8}{9}\right)^{\tfrac{\lapconst |\x|^2}{2}}\qquad \text{ for }|\x| \ge \x_2.
\end{equation}
This is possible because the right-hand side decreases exponentially, while the left-hand side only decreases polynomially.

The term with index $\gamma$ is part of the first sum, because $\ga \le \lapconst |\x|^2/2 \le \lfloor \lapconst |\x|^2 \rfloor$.\footnote{The assumptions on the constant $\x_2$ certainly imply that $\lapconst |\x|^2 / 2 \ge 2$, so $\lapconst |\x|^2 / 2 \le \lapconst |\x|^2 - 1 \le \lfloor \lapconst |\x|^2 \rfloor$.} Every term in the first sum is nonnegative, so $S_1$ is at least equal to the $\ga$ term. Then by an application of Lemma~\ref{weightlowerbound}, and the choice of constants, we have
\begin{align}S_1&\ge {b \over 4\lapconst |\x|^2} q_{\ga}^{(1/2)}(\x) \,\rho \ge {ab \over 4\lapconst |\x|^2 \ga^{d/2}} \,\rho\ge {ab \over 4\lapconst^{1+d/2} |\x|^{2 +d}} \,\rho.\label{s1bound}{}\end{align}

On the other hand, for $|S_{2}|$, we claim there exists an absolute constant $c$ such that, by (\ref{inc.chain}), 
\begin{align}\label{s2bound}
    |S_2| &\le 2b\sum_{j=\lfloor \lapconst |\x|^2 \rfloor}^\infty \delta q_j^{(1/2)}(k)\mathbb P\{\nu_1 + \cdots + \nu_n = j\}\le 4b\sum_{j=\lfloor \lapconst |\x|^2 \rfloor}^\infty (\tfrac{8}{9})^{j-\ga}\rho\le cb(\tfrac{8}{9})^{\lapconst |\x|^2/2}\rho.
\end{align}
The first term of the sum is $(8/9)^{\lfloor \lapconst |\x|^2 \rfloor - \lfloor \lapconst |\x|^2/2 \rfloor} \rho$, so $c = 4\times 9 \times (9/8)$ would make the inequality true, for example.

By (\ref{ak0bound}), the lower bound (\ref{s1bound}) is larger than the upper bound (\ref{s2bound}), so $$q^{(b)}_{n+1}(\x) - q^{(b)}_n \ge S_1 + S_2 \ge S_{1}-|S_{2}|\geq 0.$$ Therefore $q^{(b)}_n(\x) \le q^{(b)}_{n+1}(\x)$ as long as $n \le \lapconst |\x|^2/6b$. 
\end{proof}

\section*{Acknowledgements}

LAB acknowledges support from the NSERC Discovery Grants program and from the Canada Research Chairs program. This material is based in part upon work supported by the National Science Foundation under Grant No. DMS-1928930, while LAB was in residence at the Simons Laufer Mathematical Sciences Institute in Berkeley, California, during the semester of Spring 2025. GB from the NSERC Alexander Graham Bell Canada Graduate Scholarship. JL acknowledges support from NSERC Discovery Grant 2018-06371 and 2025-05575, and the Canada Research Chairs program 2018-00154 and 2023-00081.

\bibliographystyle{abbrv}
\bibliography{scmd}

\end{document}